%% file: main.tex
\documentclass[11pt]{amsart}
\makeatletter
\usepackage{manfnt} 
\usepackage[nohug]{ptdiagrams}
\usepackage{import}
\import{./}{preamble}
\usepackage{tikz-cd}
\tikzcdset{scale cd/.style={every label/.append style={scale=#1},
    cells={nodes={scale=#1}}}}
\def\cdga{\rm CDGA}
\def\cs{{\rm cSpec}}
\def\Cc{{\cC}}
\def\cM{{\mathcal{M}}}
\def\bcol{{\mathbf{colim}}}
\def\bli{{\mathbf{lim}}}
\def\Q{{\mathbb Q}}
\newcommand{\onto}{\,\,\twoheadrightarrow\,\,}
\newcommand{\la}{\label}

\newcommand{\cSpec}{\mathrm{cSpec}}

\newcommand{\Map}{\mathrm{Map}}

\newcommand{\bcolim}{\boldsymbol{\mathrm{colim}}}
\newcommand{\blim}{\boldsymbol{\mathrm{lim}}}
\newcommand{\hocolim}{\mathrm{hocolim}}

\newcommand{\ux}{c(X)}

\newcommand{\ua}{c(A)}

\newcommand{\BBox}{\Box\!\Box}

\newcommand{\cube}{\mathord{\mbox{\mancube}}}

\def\deg{\mathrm{deg}\,}

\newcommand{\id}{\mathtt{Id}}
\newcommand{\into}{\hookrightarrow}

\newcommand{\FunDD}{\Fun((\Delta^1 \vee \Delta^1)^{\rm op}, \,\Cc)}
\newcommand{\Top}{\mathrm{Top}}

\newcommand{\cAff}{\mathrm{cAff}}

\def\I{\mathcal{I}}
\def \Dc {\mathscr{D}}

\def\bF{\mathbb F}
\def\bI{\mathbb I}
\def\bN{\mathbb N}
\def\O{\mathbb O}
\newcommand{\bQ}{\boldsymbol{Q}}
\def\bS{\mathbb S}
\def\bZ{\mathbb Z}

\def\cQ{\mathscr{Q}}

\def\bL{\boldsymbol{L}}
\def\bR{\boldsymbol{R}}

\def\Lc{\mathscr{L}}
\def\Rc{\mathscr{R}}

\def\Z{\mathbb{Z}}
\def\Q{\mathbb{Q}}

\def\sl2{{\mathfrak{s}\mathfrak{l}}_2}

\def\Hom{\mathrm{Hom}}
\def\mod{\mathrm{mod}}

\def\Fun{\mathrm{Fun}}
\def\Tria{\mathtt{Tria}}
\def\Gan{\mathtt{Gan}}
\def\Cof{\mathtt{Cof}}
\def\Fib{\mathtt{Fib}}

\def\sset{\mathrm{Set}_{\Delta}}

\def\cdga{\mathrm{CDGA}}

\def\Cat{\mathtt{Cat}}

\def \cF{\mathcal{F}}

\def\cA{\mathcal{A}}
\def\cR{\mathcal{R}}
\def\cK{\mathcal{K}}
\def\Cc{\mathscr{C}}

\def\dDel{\mathbf{\Delta}}

\makeatletter
\newcommand*{\doublerightarrow}[2]{\mathrel{
  \settowidth{\@tempdima}{$\scriptstyle#1$}
  \settowidth{\@tempdimb}{$\scriptstyle#2$}
  \ifdim\@tempdimb>\@tempdima \@tempdima=\@tempdimb\fi
  \mathop{\vcenter{
    \offinterlineskip\ialign{\hbox to\dimexpr\@tempdima+1em{##}\cr
    \rightarrowfill\cr\noalign{\kern.5ex}
    \rightarrowfill\cr}}}\limits^{\!#1}_{\!#2}}}
\newcommand*{\triplerightarrow}[1]{\mathrel{
  \settowidth{\@tempdima}{$\scriptstyle#1$}
  \mathop{\vcenter{
    \offinterlineskip\ialign{\hbox to\dimexpr\@tempdima+1em{##}\cr
    \rightarrowfill\cr\noalign{\kern.5ex}
    \rightarrowfill\cr\noalign{\kern.5ex}
    \rightarrowfill\cr}}}\limits^{\!#1}}}
\makeatother
\begin{document} 
\title{Ganea decompositions of classifying spaces}
\author{Yuri Berest}
\address{Department of Mathematics,
Cornell University, Ithaca, NY 14853-4201, USA}
\email{berest@math.cornell.edu}
\author{Yun Liu}
\address{Department of Mathematics,
Indiana University,
Bloomington, IN 47405, USA}
\email{yl224@iu.edu}
\author{Ajay C. Ramadoss}
\address{Department of Mathematics,
Indiana University,
Bloomington, IN 47405, USA}
\email{ajcramad@iu.edu}
%

\begin{abstract}
We study homotopy decompositions of the classifying spaces $BG$ of compact connected Lie groups obtained by (relative) fiber-cofiber construction. Given a pair of Borel fibrations $ F \to E \to BG $  and 
$F' \to E' \to BG $, this construction yields a tower (telescope) of 
spaces $ X_{m}(F,F') $ over $BG$ indexed by $  \Z_+ $ that converges in the sense that $\hocolim \,(X_{m})\,$ is weakly homotopy equivalent to $BG$. 
We determine cohomological conditions on the fibrations that produce the spaces $X_{m}(F,F')$ with properties similar to those of the spaces of quasi-invariants of Weyl groups constructed in \cite{berest2023topological}. We prove that, under these conditions, the resulting homotopy decompositions of $BG$ are sharp (over $\Q$), the spaces $X_{m}(F,F')$ are rationally formal and Cohen-Macaulay, their cohomology rings being finite rank free modules over $H^*(BG, \Q)$. We construct many examples which include the fundamental (maximal torus) fibration $ G/T \to BT \to BG $ as well as the universal fibration $\, E_{\rm com}G_{\bf 1} \to B_{\rm com}G_{\bf 1} \to  BG \,$ for the classifying space $B_{\rm com}G$ of commuting elements in $G$ introduced in \cite{adem2015classifying}. In most cases, we give an explicit presentation for the (equivariant) cohomology rings in terms of characteristic classes and compute the (equivariant) $K$-theory of the spaces involved.
The paper contains an Appendix, where we re-examine the topological fiber-cofiber construction in an abstract setting, proving an $\infty$-categorical extension of the classical Ganea Theorem.
\end{abstract}
\maketitle



\section{Introduction}
The invariant theory of finite reflection groups has its origin in algebraic topology. 
In the 1950 ICM report \cite{Chev50}, C. Chevalley presented the results of his calculations of the rational Betti numbers of exceptional simple Lie groups. A. Borel \cite{borel1953cohomologie} reformulated these results in terms of classifying spaces of compact Lie groups
and proved the following general theorem relating  cohomology of these spaces to polynomial invariants of Weyl groups.
\begin{thm}[Borel]
\label{BThm}
Let $G$ be a compact connected Lie group with maximal torus $T$, and let $p: BT \to BG$ denote the map of classifying spaces induced by the inclusion $T \hookrightarrow G$.

$(1)$ The algebra homomorphism $ p^*: H^*(BG, \Q) \to H^*(BT, \Q) $ induced by $p$ on rational cohomology is injective, and its image coincides with the $W$-invariant subalgebra
\begin{equation}
\label{I2}
H^*(BG, \Q) \cong H^*(BT, \Q)^W \, ,   
\end{equation}
where $W = W_G(T) $ is the Weyl group of $G$ associated to $T$.

$(2)$ Both $H^*(BT,\Q)$ and $H^*(BG,\Q)$ are free polynomial algebras, with $H^*(BT,\Q)$ being a free module over $H^*(BG,\Q)$ of rank $|W|$.
\end{thm}
Shortly after, Chevalley \cite{Che55} gave an algebraic explanation to Borel's results by proving
\begin{thm}[Chevalley]
\label{CThm} 
Let $W$ be a finite group generated by reflections in a finite-dimensional vector space $V$ over a field $\,\Bbbk$ of characteristic zero, and let $ \Bbbk[V] = {\rm Sym}_{\Bbbk}(V^*) $ denote the $($graded$)$ polynomial algebra of $V$. Then the invariant subalgebra $ \Bbbk[V]^W \subseteq \Bbbk[V] $ is itself a polynomial algebra, and $ \Bbbk[V]$ is a free module over $ \Bbbk[V]^W $ of rank $|W|$.
\end{thm}
Note that Chevalley's Theorem applies in the topological situation, because the Weyl group $W = W_G(T)$ is generated by reflections in the $\Q$-vector space $ V = \pi_1(T) \otimes \Q$, and the latter is naturally isomorphic to $H_2(BT,\Q)$ by the rational Hurewicz Theorem. Consequently, 
\begin{equation}
\label{BTV}
H^*(BT, \Q) \cong {\rm Sym}_{\Q}(V^*) = \Q[V]\, .
\end{equation}
The above results laid the foundation for many important developments in algebra and topology over the past seventy years. On the topological side, we should single out the work on the Steenrod Problem on realizing polynomial algebras as cohomology rings over ${\mathbb F}_p$
(see \cites{Steen71, CE74, AW80, DMW92, Not99}), the theory of homotopy decompositions of classifying spaces of Lie groups (see \cites{JM92, JMO92}) and the discovery of $p$-compact groups as natural ($p$-local) generalizations of compact Lie groups (see  \cites{DW94}). On the algebraic side, much deep and beautiful work has been done in representation theory, combinatorics and the theory of multivariable special functions related to finite reflection groups. In particular, we can mention the discovery of Dunkl differential-reflection operators (see \cites{Du89, Hec91, DO03}), the construction of double affine Hecke algebras and its use in the proof of the famous Macdonald Conjectures (see \cites{Chered95, Chered96}), and the development of representation theory of double affine complex (and more generally, symplectic) reflection algebras (see \cites{EG02,GGOR,BEG03}).

A remarkable generalization of the classical invariant polynomials appeared in mathematical physics (see \cites{chalykh1990commutative, chalykh1993integrability}). Given a finite Coxeter group $W$ acting in its reflection representation $V$, denote by $ \cA_W = \{H\}$ the set of reflection hyperplanes of $W$ in $V$ and assign
to each $H \in \cA_W $ a non-negative integer $ m_{H} $ --- the `multiplicity' of $H$ --- in such a way that the hyperplanes in the same $W$-orbit in $ \cA_W $ have
the same multiplicities, i.e. $ m_{w(H)} = m_H $ for all $w \in W $. With these data in hand, define the {\it quasi-invariant polynomials of multiplicity} $m$ (in short, the {\it $m$-quasi-invariants} of $W$) by
\begin{equation}
\label{Qinv}
Q_m(W) \,:=\, \{ p \in \Bbbk[V]\, |\, 
s_{H}(p) \,\equiv\, p \, \,\mbox{mod}\,\langle \alpha_H \rangle^{2 m_H}\, , \, \, \forall\,H 
\in \cA_W \,\}\, ,
\end{equation}
where 
$ s_H  $ stands for the reflection operator on $ {\Bbbk}[V]$ with respect to the hyperplane $ H \in \cA_W $, $\, \alpha_H \in V^* $ is a linear form
defining $H$, i.e. $ H = \ker(\alpha_H)$, and $ \langle \alpha_H \rangle $ is the (principal) ideal in $ \Bbbk[V] $ generated by $ \alpha_H $. It is easy to see that  $Q_m(W)$ is a graded subalgebra of $ \Bbbk[V] $ stable under the action of $W$ and such that $  Q_m(W)^W = \Bbbk[V]^W  $ for all $\,m = \{m_H\}\,$. Moreover, we obviously have $ Q_0(W) = \Bbbk[V] $, and $\, Q_m(W) \supseteq Q_{m'}(W) \,$ whenever $ m_H \le m_H' $ for all $ H \in \cA_W$. Thus, with $m$ varying, the quasi-invariants define a canonical algebra filtration
\begin{equation}
\label{qfilt}
\Bbbk[V] = Q_0(W)\supseteq\,\ldots\, \supseteq Q_m(W) \supseteq  Q_{m'}(W) \supseteq \,\ldots\,
\supseteq Q_m(W)^W = \Bbbk[V]^W \, ,
\end{equation}
interpolating between all polynomials $(m=0)$ and the $W$-invariants $(m \to \infty) $.

The algebras $Q_m(W)$ are not polynomial and do not admit a simple combinatorial description in general. Nevertheless, they carry many interesting structures related to representation theory (see, e.g., \cites{BEG03, berest2011quasi}) and possess nice ring-theoretic properties, the most surprising of which is the following:
\begin{thm}[\cites{FV02, EG02b, BEG03, berest2011quasi}]
\label{TGoren}
For any $($$W$-invariant$)$ collection of multiplicities $ m $, $\,Q_m(W) $ is a free module over $ \Bbbk[V]^W $ of rank $ |W| $. Consequently, $\,Q_m(W) $ is a graded Cohen-Macaulay
algebra of $($Krull$)$ dimension $ \dim(V)$. 
\end{thm}
For $m=0$, \Cref{TGoren}  reduces to  Chevalley's \Cref{CThm}; however, the simple algebraic proof given in \cite{Che55} for $ \Bbbk[V] $ does not work for $Q_m(W)$ when $m \not= 0$. For arbitrary $m$,  \Cref{TGoren} was first established in the case of dihedral groups (the Coxeter groups of rank two) in \cite{FV02}, and for a general finite Coxeter $W$, it was then proven (by different methods) in \cite{EG02b} and \cite{BEG03}. The quasi-invariants can be defined for arbitrary pseudo-reflection groups, in which case \Cref{TGoren} takes its most general form proved in  \cite{berest2011quasi}.

Now, if $W$ is a Weyl group of a compact connected Lie group $G$, in view of Borel's \Cref{BThm},  it is natural to ask whether the algebras \eqref{Qinv} and the associated filtration \eqref{qfilt} arise topologically from some homotopy decomposition of the classifying space $BG$: namely, is there a diagram of $W$-spaces $\, X_*(G,T): \, {\mathcal M}_W \to {\rm Top} \,$ indexed by the poset of multiplicities $m$:
\begin{equation}
\label{Qtow}
BT = X_0(G,T) \,\to \,\ldots \,\to\, X_m(G,T) \,\xrightarrow{\pi_{m,m'}}\, X_{m'}(G,T)\,\to\,\ldots\, \to BG
\end{equation}
and equipped with natural $W$-maps $\, p_m: X_m(G,T) \to BG \,$, such that
\begin{equation}
\label{hocolBG}
{\rm hocolim}_{m \in {\mathcal M}_W}\, [X_m(G,T)] \,\simeq\, BG\, ,
\end{equation}
and the associated (contravariant) diagram of cohomology algebras
\begin{equation*}
H^*(BT,\, \Bbbk) = H^*(X_0,\,\Bbbk) \,\leftarrow\, \ldots\, \leftarrow\,
H^*(X_m,\,\Bbbk) \, \xleftarrow{\pi^*_{m,m'}}\, H^*(X_{m'},\,\Bbbk) \leftarrow \ldots \leftarrow H^*(BG,\,\Bbbk)
\end{equation*}
isomorphic, under the identification \eqref{BTV}, to the filtration \eqref{qfilt}? In \cite{berest2023topological}, the first and third authors reformulated the above question axiomatically (by listing all desired properties of the spaces $X_m(G,T)$) and gave a complete answer in the rank one case: for $ G = SU(2) $. The answer turned out to be very simple: the spaces $X_m(G,T)$, satisfying \eqref{Qtow}, \eqref{hocolBG} --- and all other axioms of \cite{berest2023topological} --- do exist, are unique (at least, rationally), and can be constructed explicitly. To that end, the authors of \cite{berest2023topological} 
used a formal construction known in homotopy theory as the {\it fiber-cofiber construction}. It starts with a (homotopy) fibration sequence $\,F \xrightarrow{j} X \xrightarrow{p} \,B\,$ over a (well-)pointed base $B$ and produces 
a tower of fibrations over the same base:
\begin{equation}
\label{Gatow}
\begin{diagram}[small, tight]
F               & \rTo  &  F_1      &  \rTo  & F_2 & \rTo& \ldots \\
\dTo^{j}        &       & \dTo^{j_1} &        & \dTo^{j_2} \\
X               & \rTo^{\pi_0}  &  X_1      &  \rTo^{\pi_1}  & X_2 & \rTo & \ldots \\
\dTo^{p}        &       & \dTo^{p_1} &        & \dTo^{p_2} \\
B               & \rEq  &  B     &  \rEq  & B& \rEq & \ldots \\
\end{diagram}
\end{equation}
where the spaces $X_m$ and $F_m$ are defined inductively by 
\begin{equation}
\label{XmFm}
X_m := {\rm hocof}_*(j_{m-1})\, ,\quad F_m :=  {\rm hofib}_*(p_m)\, ,\quad \forall\,m \ge 1\,.
\end{equation}
Now, for $ G = \SU(2) $, the diagram of spaces \eqref{Qtow} arises by simply applying the above procedure to Borel's fundamental fibration sequence (see \cite{berest2023topological}*{Theorem 3.9})
\begin{equation}
\label{Bfib}
G/T \to BT \xrightarrow{p} BG \, .
\end{equation}

At this point, we pause to remark that the fiber-cofiber construction was introduced by T. Ganea \cites{ganea1965generalization, ganea1967lusternik} as a tool to study his famous conjectures in algebraic topology. It is well known and widely used in the theory of Lusternik-Schnirelmann (LS) categories of spaces (see, e.g., \cites{cornea2003lusternik}); however, somewhat surprisingly, it does not play a role --- and it appears to have not been even investigated --- in the literature on homotopy decompositions of classifying spaces of  Lie groups and their $p$-local analogs. 

Unfortunately, for higher rank Lie groups, the situation is more complicated. In general, applying the fiber-cofiber construction to the Borel fibration \eqref{Bfib} does not produce spaces with the desired properties. As shown in \cite{berest2023topological}*{Example 3.8}, if $ {\rm rk}(G) > 1 $, already the first cofiber $ X_1 = {\rm hocof}(G/T \hookrightarrow BT) $ of \eqref{Bfib} has rational cohomology that fails to be a Cohen-Macaulay ring, and hence, in view of \Cref{TGoren}, cannot represent the algebras $Q_m(W)$. Realizing the rings of quasi-invariants for arbitrary compact Lie groups  
requires additional topological tools and a more elaborate construction. One such construction has been recently proposed in \cite{berest2025quasiflag}, to which we refer the interested reader.

On the other hand, in homotopy theory, there are known several generalizations of the classical fiber-cofiber construction  (see, e.g., \cites{Doe93, HL94, Doe98}). Motivated by \cite{berest2023topological}, one can reverse the logic and ask if any of these produce homotopy decompositions of classifying spaces with good properties, similar to those of the spaces of quasi-invariants \eqref{Qtow}.

In the present paper, we examine the simplest (and arguably most natural) topological generalization of the Ganea construction: the  join of fibrations. Given a compact connected Lie group $G$ and a pair of $G$-spaces, $F$ and $F'$, we consider the Borel fibration sequences
\begin{eqnarray}
\label{FFf} &F \to X \xrightarrow{p} BG \\
\label{PFf} &F' \to X' \xrightarrow{p'} BG
\end{eqnarray}
where $ X := EG \times_G F $ and $ X' := EG \times_G F' $ are the spaces of homotopy $G$-orbits
in $F$ and $F'$. The {\it join} of \eqref{FFf} and 
\eqref{PFf} is then defined by taking the homotopy pushouts of the rows of the following natural diagram
%
\[\begin{tikzcd}
	F & {F \times F'} & {F'} \\
	X & {X \times^h_{BG}X'} & {X'} \\
	BG & BG & BG
	\arrow[from=1-1, to=2-1]
	\arrow[from=1-2, to=1-1]
	\arrow[from=1-2, to=1-3]
	\arrow[from=1-2, to=2-2]
	\arrow[from=1-3, to=2-3]
	\arrow[from=2-1, to=3-1]
	\arrow[from=2-2, to=2-1]
	\arrow[from=2-2, to=2-3]
	\arrow[from=2-2, to=3-2]
	\arrow[from=2-3, to=3-3]
	\arrow[equals, from=3-2, to=3-1]
	\arrow[equals, from=3-2, to=3-3]
\end{tikzcd}\]
where $ X \times^h_{BG}X' $ is the homotopy fiber product of
the spaces $X$ and $X'$ over $BG$.
This way we get a new Borel fibration sequence
\begin{equation}
    \label{F*fib}
F_1(F,F') \to X_1(F,F') \to BG 
\end{equation}
where $ F_1(F,F') = F \ast F'$ is the classical join of the spaces $F$ and $F'$ equipped with the
diagonal $G$-action and $ X_1(F,F') := EG \times_G F_1 $ is the homotopy quotient of $F_1$ with respect to that action. 
Just as the classical fiber-cofiber construction, this construction can be iterated {\it ad infinitum} by letting inductively
$$
F_m(F,F') := F_{m-1}(F,F') \ast F' \quad \mbox{and} \quad X_m(F,F') := EG \times_G F_m(F,F') \, , \quad  \forall\, m \ge 1\,.  
$$
This produces a tower (mapping telescope) of spaces
\begin{equation}
\label{QtowF}
X = X_0(F,F') \,\to \,\ldots \,\to\, X_m(F,F') \,\xrightarrow{\pi_m}\, X_{m+1}(F,F')\,\to\,\ldots\, \to BG
\end{equation}
such that
\begin{equation}
\label{equivF}
{\rm Tel}\,[X_*(F,F')] = \underset{m\to \infty}{\rm hocolim}\,[\,X_m(F,F')\,] \,\simeq\,  BG
\end{equation}
Now, if we let $ Q_m(F,F') $ denote the rational cohomology rings of the spaces \eqref{QtowF}:
\begin{equation}
\label{QF} 
Q_m(F,F') := H^*(X_m(F,F'), \Q) = H_G^*(F\ast F' \ast \stackrel{m}{\ldots}\ast F', \, \Q)
\end{equation}
then each $ Q_m(F,F') $ is naturally a module over $ H^*(BG, \Q) \cong \Q[V]^W $ via the algebra homomorphism induced  by the Borel quotient map $ X_m(F,F') \to BG $ at step $m$. We can therefore ask when the algebras $ Q_m(F,F') $ have the same properties as
the classical rings of quasi-invariants $Q_m(W)$: more precisely, for which pairs of spaces $F$ and $F'$, the result of \Cref{TGoren} holds for $ Q_m(F,F') $?

The following provides a (partial) answer to this question (see \Cref{thm:borel-tow} and \Cref{cor:htpy-decomp}).

\begin{thm}
\label{ThmI}
Assume that $ F $ is a $\Q$-finite $G$-CW complex with rational cohomology vanishing in odd degrees and assume that $ F' $ has $G$-homotopy type of an odd-dimensional sphere 
$ \bS^{2n-1} $ equipped with a transitive $G$-action. Then, for all $ m \ge 0 $, the algebras $Q_m(F,F')$ are free modules over $\,\Q[V]^W$ of rank $\,\dim_{\Q}[H^{\ast}(F,\,\Q)]\,$ and hence $($graded$)$ Cohen-Macaulay. Each map $ \pi_m $ in the diagram \eqref{QtowF} induces an {\rm injective} algebra homomorphism on rational cohomology, defining the descending filtration
\begin{equation*}
Q_0(F,F') \hookleftarrow  Q_1(F,F')  \hookleftarrow \ldots  
\hookleftarrow Q_m(F,F') \xhookleftarrow{\,\pi_m^*\,} Q_{m+1}(F,F') \hookleftarrow \ldots
\end{equation*}
such that  $ \,\varprojlim \,Q_m(F,F') \cong \Q[V]^W $. 
Consequently, \eqref{QtowF}, \eqref{equivF} provides a {\rm sharp}
homotopy decomposition of $BG$ $($with respect to rational cohomology$)$  in the sense that the associated Bousfield-Kan spectral sequence degenerates, with all higher limits vanishing:
$$
{\varprojlim}^{i}\,H^*(X_m(F,F'), \Q)\,=\,0 \, ,\quad \forall\,i \ge 1\,.
$$
\end{thm}
\Cref{ThmI} may seem formal; however, by choosing appropriate $F$ and $F'$, it allows us to construct topologically a number of interesting examples. First, note that the assumption  on $F'$ can be restated by saying the fibration \eqref{PFf} is a spherical fiber bundle
\begin{equation}
\label{Pfib'}
G/H \to BH \to BG
\end{equation}
where $H$ is a closed subgroup of $G$, such that $\,G/H\cong S^{2n-1}\,$ for $ n\ge 1 $. Such pairs   $(G,H)$ have been classified in the literature, and their list is not very long  (see \Cref{thm:RHS}). If we assume that $ \pi_1(G)$ is torsion-free, then there are essentially the following 5 cases:
\begin{enumerate}
    \item $G=\U(k)\, ,\, H=\U(k-1),\, G/H= \bS^{2k-1}\, (k\ge 2)\,$. 
    \item $G=\SU(k)\, ,\, H=\SU(k-1),\, G/H= \bS^{2k-1}\, (k\ge 2)\,$.
    \item $G=\SP(k)\, ,\, H=\SP(k-1),\, G/H= \bS^{4k-1}\, (k\ge 1)\,$.
    \item $G=\Spin(7)\, ,\, H=G_2,\, G/H=\bS^{7}$.
    \item $G=\Spin(9)\, ,\, H=\Spin(7),\, G/H=\bS^{15}$.
\end{enumerate}

As for the $G$-CW complex $F$, we will consider the two spaces: 
\begin{itemize}
\item
$ F = G/T $, the classical flag manifold of $G$, for which the fibration \eqref{FFf} is just the Borel fundamental fibration \eqref{Bfib};
\item 
$\,F = E_{\rm com}G_{\bf 1}\,$ is the identity component of (the total space of) the universal bundle for the transitionally commutative principal $G$-bundles, in which case 
 \eqref{FFf} reads
 \begin{equation}
 \label{PFib'}
 E_{\rm com}G_{\bf 1} \,\to\, B_{\rm com}G_{\bf 1}  \,\to\,  BG\, ,
 \end{equation}
where $ B_{\rm com}G $ is the {\it classifying space for commuting elements} (in short, {\it for commutativity}) in $G$ and $ B_{\rm com}G_{\bf 1} $ is its connected component of the identity (i.e., the point  in $ B_{\rm com}G $ representing the trivial commuting elements $ \{{\bf 1}_n: {\mathbb Z}^n \to G\}_{n \ge 0} $). 
\end{itemize}
The classifying space  $ B_{\rm com}G $ and 
the associated universal bundle $ E_{\rm com}G $ were defined for arbitrary topological groups in \cite{adem2012commuting} and were studied extensively  in recent years. For compact connected Lie groups $G$, it is natural to focus on the identity components  of these spaces,
$ B_{\rm com}G_{\bf 1} $ and $\,E_{\rm com}G_{\bf 1}$, which were 
introduced and studied in \cite{adem2015classifying}. Among other things, this last paper 
computes the rational cohomology of the space $\,E_{\rm com}G_{\bf 1}$ (see \cite{adem2015classifying}*{Corollary 7.4}):
\begin{equation}
\label{EGcom}   
H^*(E_{\rm com}G_{\bf 1},\,\Q)\,\cong\,[ H^*(G/T,\,\Q) \otimes H^*(G/T,\,\Q)]^W
\end{equation}
which shows that the assumptions of our \Cref{ThmI} hold for $F = E_{\rm com}G_{\bf 1}$. 

We compute the rational homotopy types of the $G$-spaces $F_m$ and the corresponding homotopy quotients $ X_m(F, G/H)$.
As result, we give canonical presentations
for the rational cohomology rings of $ X_m(F, G/H)$ (i.e., the algebras \eqref{QF}) 
in terms of the Euler classes of the spherical bundles \eqref{Pfib'}; they have presentations \eqref{prop:QmGTH} that
are topological analogs of the algebraic definition \eqref{Qinv} of the classical quasi-invariants of $W$.
Extending the results of \cite{berest2023topological}, we also compute explicitly the $T$-equivariant cohomology and
the equivariant $K$-theory of the spaces $ F_m $ for the flag manifold cases.
We remark that the original results of \cite{berest2023topological} correspond to the case
$G = \SU(2)$ and $ H = \{e\} $, when (for $ F = G/T$) we recover the classical algebras of quasi-invariants of $W = \bZ/2\bZ$. However, even in that case --- for $ F = E_{\rm com}G_{\bf 1} $ --- we get new interesting analogs of quasi-invariants that did not seem to appear in  the earlier literature.

\Cref{ThmI} admits a natural generalization that adds a combinatorial complexity and further extends our class of examples. Instead of using the same two fibrations, \eqref{FFf} and \eqref{PFf}, in the above join construction, we can vary them at each step. To be precise, for a compact connected Lie group $G$, we fix a maximal torus $T$ and consider an ascending sequence of closed subgroups: $\,T \subseteq G_1 \subseteq G_2 \subseteq \ldots \subseteq G_l \subseteq G \,$ such that, for each $G_i $, there exists an $\,H_i \subseteq G_i \,$ from the above list (i.e. $ G_i/H_i \cong \bS^{2k_i -1}$ for some $k_i \ge 1$). Then, 
we choose for each pair $(G_i, H_i)$ an integer (multiplicity) $m_i \ge 0$ and,
starting with $ F = G/T $, construct inductively the spaces
$\,
F_i^{(m)} = F_{i-1}^{(m)} \ast^{m_i}_{G/G_i} (G/H_i) \,
$,  
where $\,\ast^{m_i}_{G/G_i}\,$ denotes the iterated ($m_i$-fold) relative join over $G/G_i$. Each of these spaces carries a natural (diagonal) $G$-action, hence, by taking the homotopy quotients, we can define  
\begin{equation}
\label{XFGH}
X_{(m)}(F, G_1/H_1, \ldots, G_l/H_l) := F^{(m)}_{hG} \,\simeq\, BT \ast^{m_1}_{BG_1} BH_{1}  \ast^{m_2}_{BG_2}\,\ldots \,\ast^{m_l}_{BG_l} BH_l \,. 
\end{equation}
In this way, we get a diagram of spaces \eqref{XFGH} indexed by the $l$-tuples of integers $ 
m = (m_1, m_2, \ldots, m_l) \in {\mathbb Z}_+^l\,$ whose rational cohomology rings enjoy the same properties as the algebras \eqref{QF}, although having a more complicated combinatorial structure. We should mention that this join construction is one of the key ingredients
of the definition of quasi-flag manifolds realizing the classical algebras $Q_m(W) $ of quasi-invariants of Weyl groups in \cite{berest2025quasiflag} (see \Cref{eg:quasi-partial}).

\subsection*{Appendix}
The paper includes an Appendix that may (and probably should) be read independently of 
the rest. In this Appendix, we re-examine the classical Ganea construction in the context of
$\infty$-categories. In an abstract homotopy-theoretic setting (specifically, in the setting of Quillen model categories), the Ganea construction was studied  in the work of J.-P. Doeraene \cites{Doe93, Doe98}, K. Hess and J.-M. Lemaire \cite{HL94}, D. Tanr\`e \cite{DT95} who used it as a tool for defining generalized Lusternik-Schnirelmann (LS) invariants. Our results can be viewed as an extension and refinement of this work, although we state them in a different way. For any (finitely) bicomplete, pointed $\infty$-category $ \Cc $, we construct a canonical $\infty$-functors
$\Xi: \Fun(\Delta^1, \Cc) \to \Fun(\cube, \Cc) $ (see \Cref{GaTh}) and $ \Xi^{\rm rel}: \FunDD \to \Fun(\cube, \Cc) $ (see \Cref{RGaTh})
which take values in the $\infty$-subcategory of cubical diagrams in $\Cc $ spanned by
Mather cubes (in the sense of \cite{kerodon}*{Definition~7.7.5.1}). As is well-known, 
constructing a functor between $\infty$-categories amounts to solving a homotopy coherence problem: that is, making coherent choices for higher dimensional homotopies (simplices). We approach this problem by redefining  basic functors. Specifically, we extend the standard fiber and cofiber functors defined (for pointed $\infty$-categories) in \cite{HA} to adjoint functors on diagram categories  (see \Cref{Propadj}) and use the associated natural transformations (the adjunction unit and counit) to build higher dimensional homotopy coherent diagrams. In the $\infty$-categorical setting, the natural transformations are themselves represented by $\infty$-functors, which (in our case) possess additional (`cubical') symmetries. Making use of these additional symmetries allow us to construct the $\infty$-categorical Mather cubes in a functorial way. As an application, we construct a
functorial Ganea tower on morphisms in an arbitrary $\infty$-category $\Cc$ and,
using the recent work of M. Anel, G. Biedermann, E. Finster and A. Joyal \cite{ABFJ20}, prove that this tower converges, providing a generalization of the classical Ganea decomposition for any hypercomplete $\infty$-topos (see \Cref{Gacor}).
We mention that our approach is inspired by the work of M. Groth and J. $\check{\rm S}$t\'ov\'i$\check{\rm c}$ek on abstract representation theory in stable $\infty$-categories
(see \cites{GS18, GR19}) and also the recent work of T. Dyckerhoff, M. Kapranov, and Y. Soibelman on spherical functors (see \cites{DKS23, DKSS24}).

\subsection*{Organization of the paper}
In \Cref{sec:cAff}, we construct rational algebraic models for our basic fibrations, 
which properties we then use in \Cref{sec:tower} to prove \Cref{ThmI}. In \Cref{sec:eg}, 
we study two  families of examples mentioned above: one related to
the maximal torus fibration \eqref{Bfib} and the other to the universal fibration 
\eqref{PFib'} for the classifying space of commuting elements in $G$. In \Cref{sec:gen}, 
we extend our main \Cref{ThmI} to families of fibrations with total space \eqref{XFGH}. 
The paper ends with \Cref{sec:ganea} on $\infty$-categorical Ganea construction which contents we have summarized above.

\subsection*{Acknowledgments}{\footnotesize
Part of this research was conducted while the first author was visiting the Simons Laufer Mathematical Sciences Institute (formerly known as MSRI) as Research Professor for the program ``Noncommutative Algebraic Geometry'' in Spring 2024. He is very grateful to the Institute for its hospitality and financial support under the NSF Grant DMS-1928930. The work of Yu.B. was also partially supported by the Simons Collaboration Grant 712995. }


\section{Relative Ganea construction for coaffine stacks}
\label{sec:cAff}
In this section, we work in the $\infty$-category of coaffine stacks over $\Q$. Coaffine stacks, introduced by Toen \cite{ToenCA} and studied by Lurie \cites{luriedag8,luriedag13}, provide a convenient algebro-geometric framework for rational homotopy theory. (We refer the  interested reader to  \cite{berest2025quasiflag}*{Appendix B} for a gentle introduction to 
the theory of coaffine stacks and and its applications.) 
The Ganea construction in $\infty$-categories that we use in this section is introduced in \Cref{AS3}.


\subsection{Coaffine stack of relative quasi-invariants}

Fix $W$ a finite reflection group with representation $V$ (over $\Q$), together with a codimension-one subspace $V_H\subseteq V$ and a maximal finite reflection subgroup $W_H\subset W$ preserving $V_H$.
Equip the polynomial algebra $\Q[V]$ with the {\it cohomological} grading by setting
$$ {\rm deg}(x)=2\, ,\qquad \forall\,x\,\in\,V^{\ast}\, .$$
With this grading and the trivial differential, we may regard $\Q[V]$ and $\Q[V_H]$ as objects of $\cdga^{> 0}_{\Q}$. The invariant subalgebras $\Q[V]^{W}$ and $\Q[V_H]^{W_H}$ inherit the grading and hence also define objects of $\cdga^{>0}_{\Q}$. We can therefore define the corresponding coaffine stacks by
\begin{eqnarray*}
V^c\,:=\, \cs\,\Q[V]\,, & & \,\,V^c/\!/W\,:=\,\cs\,\Q[V]^{W}\,, \\
V_H^c\,:=\,\cs\, \Q[V_H]\,,  & &\,\, V_H^c/\!/W_H\,:=\,\cs\,\Q[V_H]^{W_H}\, .
\end{eqnarray*}  
There is a canonical morphism of coaffine stacks 
\begin{equation}\label{eq:fib-Wn}
    \begin{tikzcd}[sep=small]
    V^c \ar[r] & V^c/\!/W
    \end{tikzcd}
\end{equation}
and we have the following commutative diagram 
\begin{equation*}
    \begin{tikzcd}
    V_H^c \ar[r]\ar[d] & V_H^c/\!/W_H \ar[d] & \\
    V^c \ar[r] &  V^c/\!/W_H \ar[r] & V^c/\!/W
    \end{tikzcd}
\end{equation*}
which induces a morphism of coaffine stacks
\begin{equation}
\label{eq:fib-W}
   V_H^c/\!/W_H \to V^c/\!/W \, . 
\end{equation}

\begin{defn}
Given the two maps \eqref{eq:fib-Wn} and \eqref{eq:fib-W} above, we set $V_0^c:=V^c$
and inductively define the \emph{coaffine stack of relative quasi-invariants} by
\begin{eqnarray}
    \label{eq:Vcm}
    V^c_{m+1}(W,W_H) = V^c_m(W,W_H)  \ast_{V^c/\!/W} V_H^c/\!/W_H \, ,
\end{eqnarray}
where the relative Ganea construction is in the $\infty$-category ${\rm cAff}_{\Q}$ (see \Cref{AS3}).
\end{defn}

\begin{thm}
\label{thm:alg-gen-quasi}
    Suppose that the morphism \eqref{eq:fib-W} is induced by a surjective algebra homomorphism 
    $\Q[V]^W\twoheadrightarrow \Q[V_H]^{W_H}$ such that $\Q[V_H]^{W_H}=\Q[V]^W/(\theta)$. Then the coaffine stack of relative quasi-invariants admits an explicit description $V^c_m(W,W_H) = \cs\,Q_m(W,W_H)$, where
    \begin{equation}
    \label{eq:Pm-alt}
        Q_m(W, W_H) := \{\, f\in \Q[V] \, \vert \, s_{\alpha}(f)\equiv f \, \mod\, (\theta)^m, \, s_{\alpha}\in \cA_W \, \}
    \end{equation}
    and $\cA_W$ is the set of reflections in $W$. It inherits the $W$-action from $\Q[V]$.
\end{thm}
\begin{rmk}
The presentation of the algebras of relative quasi-invariants \eqref{eq:Pm-alt} and classical quasi-invariants \eqref{Qinv} are very similar but not the same. While both are Cohen-Macaulay, $Q_m(W)$ is Gorenstein but $Q_m(W,W_H)$ is not (see their Hilbert series in \Cref{cor:hilb}). 
The definition of $Q_m(W,W_H)$ depends on a choice of subgroup $W_H$ of $W$ while $Q_m(W)$ does not, and the multiplicity for $Q_m(W,W_H)$ is only a non-negative integer. 
\end{rmk}

\begin{prop}
\label{prop:gen-quasi}
    There is an alternative presentation 
    \begin{equation}\label{eq:coor-pre}
    Q_m(W,W_H) = \Q[V]^{W} + \theta^m \cdot \Q[V] \, .
    \end{equation}
\end{prop}
\begin{proof}
For simplicity we write $Q_m=Q_m(W,W_H)$ and $P_m = \Q[V]^{W} + \theta^m \cdot \Q[V]$.
It is easy to see that $ P_m \subseteq Q_m $, so we only need to show that $ Q_m \subseteq P_m $.
Since $\theta$ is $W$-invariant, the ideal $(\theta)^m\subseteq \Q[V]$ is $W$-invariant, thus $Q_m$ can be described as 
\[
Q_m = \{\, f\in \Q[V] \, |\, w(f) \equiv f \, \mod\, (\theta)^m, \forall w \in W \,\}\, . 
\]
Given any $ f \in Q_m $, we can decompose $ f $ as 
\[ f = \frac{1}{\vert W \vert} \sum_{w\in W}w(f) + \frac{1}{\vert W \vert} \sum _{w\in W}(1-w)(f) \]
where the first component is $ W $-invariant and the latter component is a sum of polynomials where each $ (1-w)(f) \in (\theta)^m $, so $ f\in P_m $. Thus we proved $ P_m = Q_m $ as desired.
\end{proof}
\begin{rmk}
\la{rmk:EqPre}
    A generalization of \Cref{prop:gen-quasi} holds for algebras 
    \begin{equation}
    \la{Qm1}    
    Q_m\coloneqq\{\, f\in Q_0 \, |\, s_{\alpha}(f)\equiv f \, \mod\, (\theta)^m,\, s_{\alpha}\in \cA_W \, \}
    \end{equation}
    where $Q_0$ is an $W$-algebra and $\theta\in Q_0^W$, i.e., $Q_m$ has an alternative presentation 
    \begin{equation}
    \la{Qm2}    
    Q_m = Q_0^W + \theta^m \cdot Q_0 \, .
    \end{equation}
    We will interchangeably switch between those two equivalent presentations, and for proof convenience we will mostly use the latter presentation.
\end{rmk}

We use the alternative presentation \eqref{eq:coor-pre} to show that $Q_m(W,W_H)$ is a free module over $\Q[V]^W$ and thus Cohen-Macaulay (see, e.g., \cite{bruns1998cohen}). 
%
\begin{prop}\label{prop:basis}
    Let $P_0 \supset P_1 \supset P_2 \supset \ldots \supset P_m \supset \ldots \supset P_{\infty}$
    be a descending sequence of $\Q$-algebras such that $P_m = P_{\infty} + \theta \cdot P_{m-1}$
    for some nonzero divisor $\theta\in P_{\infty}$. Then 
    $P_m = P_{\infty}+\theta^m \cdot P_0$.
    Furthermore, if we have a basis $\{b_i\}$ of the $\Q$-algebra $P_0$ as a free module over $P_{\infty}$ of rank $r$ with $b_1=1$, then 
    $\{1, \theta^mb_i \, \vert\, 1< i\leq r\}$
    forms a basis of $P_m$ as a free module over $P_{\infty}$. In particular, if $P_{\infty}$ is a polynomial algebra, $P_m$ is Cohen-Macaulay.
\end{prop}

\begin{proof}
    By induction on $m$, we have 
    $$P_m = P_{\infty}+\theta\cdot P_{m-1} = P_{\infty}+\theta\cdot (P_{\infty}+\theta^{m-1}\cdot P_0)=P_\infty+\theta^m\cdot P_0\, .$$ 
    We can write 
    $P_0 = \bigoplus_{1\leq i \leq r} b_i \cdot P_{\infty}$ and
    \begin{equation*}
        P_m  = P_{\infty} + \bigoplus_{1\leq i \leq r} \theta^m b_i \cdot P_{\infty} = P_{\infty} + \theta^m \cdot P_{\infty} + \bigoplus_{1< i \leq r} \theta^mb_i \cdot P_{\infty} = P_{\infty} + \bigoplus_{1< i \leq r} \theta^m b_i \cdot P_{\infty}
    \end{equation*}
    Note in the above decomposition, $b_i \cdot P_{\infty}\cap b_j \cdot P_{\infty} = \{0\} $ whenever $ i\neq j $, therefore $ P_{\infty} \cap \theta^m b_i \cdot P_{\infty} = \{0\} $ for $ i>1 $, so we can write 
    \[ 
    P_m = P_{\infty} \bigoplus \bigoplus_{1< i \leq r} \theta^m b_i \cdot P_{\infty}
    \]
    which shows $ P_m $ is a free $ P_{\infty} $-module of rank $ r $ with the desired basis.
\end{proof}

\begin{proof}[Proof of \Cref{thm:alg-gen-quasi}]
For simplicity we write $Q_m=Q_m(W,W_H)$.
We proceed by induction on \(m\). The case \(m=0\) is immediate. 
Assume \(m \ge 0\). By definition of the relative join,
\[
\begin{split}
V^c_{m+1}
&= \bcol \{ V^c_m \leftarrow V^c_m \times_{V^c/\!/W} H^c/\!/W_H \rightarrow H^c/\!/W_H \} \\
&\cong \cs\!\left(
\bli \{ Q_m \to Q_m \otimes^{\bL}_{\Q[V]^W} \Q[H]^{W_H}
\leftarrow \Q[H]^{W_H} \}
\right).
\end{split}
\]
Since \(Q_m\) is free over \(\Q[V]^W\), the derived tensor reduces to the ordinary tensor. Using
\(\Q[H]^{W_H} \cong \Q[V]^W/(\theta)\) with \(\theta\) a non-zero divisor, we obtain
\[
Q_m \otimes^{\bL}_{\Q[V]^W} \Q[H]^{W_H} \cong Q_m/(\theta)\, ,
\]
hence
\[
V^c_{m+1}
\cong \cs\!\left( Q_m \times_{Q_m/(\theta)} \Q[H]^{W_H} \right)\, .
\]

Consider the map
\[
\pi_{m+1}^\ast \colon
Q_m \times_{Q_m/(\theta)} \Q[H]^{W_H}
\to Q_m \hookrightarrow \Q[V]\, ,
\qquad (f,\alpha)\mapsto f\, .
\]
If \(\pi_{m+1}^\ast(f,\alpha)=\pi_{m+1}^\ast(g,\beta)\), then \(f=g\), and the fiber product condition
implies \(\widetilde{\alpha}-\widetilde{\beta} \in (\theta)\cap \Q[W]^W\) where $\widetilde{\alpha}, \widetilde{\beta}$ are liftings of $\alpha,\beta$ in $\Q[V]$. 
Since \(Q_m\) is free over \(\Q[V]^W\), $(\widetilde{\alpha}-\widetilde{\beta}) \in \theta \cdot \Q[V]$, so $\alpha=\beta\in \Q[H]^{W_H}$.
Thus \(\pi_{m+1}^\ast\) is injective, with image
\[
\Q[H]^{W_H} + \theta\cdot Q_m
= \Q[V]^W + \theta\cdot Q_m = Q_{m+1}\, ,
\]
where the last equality follows from \Cref{prop:basis} for \(P_0=\Q[V]\, ,P_\infty=\Q[V]^W\).
\end{proof}


\subsection{Examples}
\label{sec:caff-eg}

Classical examples include the pairs $(S_n,S_{n-1}), (B_n,B_{n-1})$. We present the symmetric group case, as the others are analogous. We also include an example of the exceptional pair $(B_3,D_2)$.
\begin{eg}
    For $V=k^n$ with $W=S_n$ acting by permutation of indices, and $H=k^{n-1}$ with restricted $S_{n-1}$-action,  
    we can choose generators $t_i$ such that 
    \begin{equation*}
    \begin{array}{ll}
        V^c = \cs \,\Q[t_1,\ldots,t_n]\, , & V^c/\!/W = \cs\,\Q[\theta_1,\ldots,\theta_n]\, ,\\
        H^c=\cs\,\Q[t_1,\ldots,t_{n-1}]\, ,& H^c/\!/W_H = \cs\,\Q[\widetilde{\theta}_1,\ldots,\widetilde{\theta}_{n-1}]\, ,
    \end{array}
    \end{equation*}
    where $\theta_i=\sigma_i(t_1,\ldots,t_n)$ and $\widetilde{\theta}_i=\sigma_i(t_1,\ldots,t_{n-1})$ the $i$-th elementary symmetric polynomials, and the map \eqref{eq:fib-W} takes $\theta_i \mapsto \widetilde{\theta}_i$ for $i<n$ and $\theta_n\mapsto 0$. We have
    \[
    Q_m(S_n,S_{n-1})= \{\, f\in \Q[t_1,\ldots,t_n] \, | \, s_{ij}(f) \equiv f \, \mod \, (\theta_n)^m\,, \forall s_{ij}\in S_n \, \} \, ,
    \]
    where $s_{ij}$ acts by permuting indices.
\end{eg}

\begin{eg}
\la{eq:B3D6}
    For $V=k^3$ with $B_3$ action, and $H\subset V$ the hyperplane $\{t_1+t_2+t_3=0\}$ with $D_6$ acting by 
    \begin{equation*}
        \begin{array}{rllcrll}
            s_{(1,-1,0)}: & t_1\mapsto -t_2\, , & t_3\mapsto -t_3\, ,  && 
            s_{(-1,2,-1)}: & t_1\mapsto t_3\, , & t_2\mapsto t_2\, , \\
            s_{(1,0,-1)}: & t_1\mapsto -t_3\, , & t_2\mapsto -t_2\, , && 
            s_{(-1,-1,2)}: & t_1\mapsto t_2\, , & t_3\mapsto t_3\, , \\
            s_{(0,1,-1)}: & t_2\mapsto -t_3\, , & t_1\mapsto -t_1\, , && 
            s_{(2,-1,-1)}: & t_2\mapsto t_3\, , & t_1\mapsto t_1\, . \\ 
        \end{array}
    \end{equation*}
    We have 
    \begin{equation*}
    \begin{array}{ll}
        V^c = \cs \,\Q[t_1,t_2,t_3]\, ,& V^c/\!/W = \cs\,\Q[\theta_1,\theta_2,\theta_3]\, ,\\
        H^c=\cs\,\Q[t_1,t_2,t_3]/(t_1+t_2+t_3), & H^c/\!/W_H = \cs\,\Q[\theta_1,\theta_3]\, ,
    \end{array}
    \end{equation*}
    where $\theta_1=t_1^2+t_2^2+t_3^2, \theta_2=t_1^2t_2^2+t_2^2t_3^2+t_3^2t_1^2$ and $\theta_3=t_1^2t_2^2t_3^2$. 
    We have 
    \[
    Q_m(B_3,D_6)= \{\, f\in \Q[t_1,t_2,t_3] \, |\, s_\alpha(f) \equiv f \, \mod\, (\theta_2)^m\, , \forall s_\alpha \in D_3 \, \} \, .
    \]
\end{eg}


\section{Ganea towers of Borel fibrations}
\label{sec:tower}

Throughout this section, we work in the (pointed) model category of (pointed) topological spaces.
Let $F\to X\to B$ and $F'\to X'\to B$ be two fibration sequences over the same base $B$. 
There is an induced fibration
\begin{equation}\la{GaTop}
    F\ast F' \to E\ast_B E' \to B
\end{equation}
where $X\ast_B X' \coloneqq \hcolim\{X\leftarrow X\times_B X' \to X'\}$. 
This is the \emph{relative Ganea (join) construction} in a pointed model category satisfying the (second) Mather's cube axiom (see, e.g. \cite{Doe98}*{Proposition 4.2}).
A comparison between the relative Ganea (join) constructions in the model categories and $\infty$-categories is discussed in \Cref{compcolim}.

Iterating this construction gives us a tower of fibration sequences analogous to \eqref{Gatow} 
where $F_m$ and $X_m$ are defined inductively by $F_0=F, F_{m+1}=F_m \ast F'$ and $X_0=X, X_{m+1}=X_m \ast_B X'$. We refer to this as a \emph{Ganea tower of fibrations}.

We focus on the case when the base is $B=BG$ for a compact connected Lie group $G$. We further assume
\begin{itemize}
    \item $F$ is a $\Q$-finite $G$-CW complex with rational cohomology vanishing in odd degrees, together with a Borel fibration
    \begin{equation}\label{eq:fib-F0}
        F \to X= EG \times_G F \to BG \, .
    \end{equation}
    In particular, $F$ is \emph{equivariant formal} in the sense that associated Serre spectral sequence collapses.
    \item $ F' $ has $G$-homotopy type of an odd-dimensional sphere $ \bS^{2n-1} $ equipped with an effective and transitive $G$-action, together with a Borel fibration 
    \begin{equation}\label{eq:fib-F'}
        F' \to X'= EG \times_G F' \to BG \, .
    \end{equation}
\end{itemize}

As both \eqref{eq:fib-F0} and \eqref{eq:fib-F'} are Borel fibrations, the relative Ganea (join) construction \eqref{GaTop} also gives us a Borel fibration (see, e,g, \cite{liu2023towers}*{Proposition 4.6.6}).
Therefore we get a tower of Borel fibrations analugous to \eqref{Gatow} where $F_0(F,F')=F, F_{m+1}(F,F') = F_m(F,F') \ast F'$, and $X_0(F,F')=X, X_m(F,F') = X_{m-1}(F,F')\ast_BX'$, such that 
\[X_m(F,F') \simeq EG\times_G F_m(F,F')\, .\]

\subsection{Main results}
Now we are ready to state our main result in this section.
\begin{thm}
\label{thm:borel-tow}
Let $G$ be a compact connected Lie group with maximal torus $T$ and Weyl group $W$.
Assume that 
\begin{enumerate}
    \item $ F $ is a $\Q$-finite $G$-CW complex with rational cohomology vanishing in odd degrees,
    \item $ F'$ has $G$-homotopy type of an odd sphere $\bS^{2n-1}$ equipped with a transitive $G$-action. 
\end{enumerate}
Then, for all $ m \ge 0 $, the algebras $Q_m(F,F')= H^\ast(X_m(F,F');\Q)$ are free modules over $H^\ast(BG;\Q)=\Q[V]^W$ of rank $\,\dim_{\Q}[H^{\ast}(F,\,\Q)]\,$ and hence $($graded$)$ Cohen-Macaulay. Each map $ \pi_m $ in the diagram \eqref{QtowF} induces an {\rm injective} algebra homomorphism on rational cohomology, defining the descending filtration
\begin{equation*}
Q_0(F,F') \hookleftarrow  Q_1(F,F')  \hookleftarrow \ldots  
\hookleftarrow Q_m(F,F') \xhookleftarrow{\,\pi_m^*\,} Q_{m+1}(F,F') \hookleftarrow \ldots
\end{equation*}
such that  $ \,\varprojlim \,Q_m(F,F') \cong \Q[V]^W $. Furthermore, $F_m(F,F')$ and $X_m(F,F')$ are formal spaces for $m>0$.
\end{thm}
As each $Q_m^j(F,F')= H^j(X_m(F,F');\Q)$ is finite, $\varprojlim^1_m Q_m^j(F,F')=0$ (see, e.g.,  \cite{BK72}*{Chapter IX \S 3 Corollary 3.5, p256}).
Consequently, 
\begin{cor}
\label{cor:htpy-decomp}
The diagram \eqref{QtowF} provides a {\rm sharp}
homotopy decomposition of $BG$ $($with respect to rational cohomology$)$  in the sense that the associated  Bousfield-Kan cohomology spectral sequence degenerates, with all higher limits vanishing:
$$
{\varprojlim}^{i}\,Q_m^j(F,F')\,=\,0 \, ,\quad \forall\,i \ge 1\,.
$$
\end{cor}

We have an explicit presentation of $Q_m(F,F')$.
\begin{prop}
\label{prop:presentation}
    Let $\theta\in H^{\ast}(BG;\Q)=\Q[V]^W$ be the Euler class of the spherical fiber bundle \eqref{eq:fib-F'}, then
    \begin{equation}
    \label{eq:HQ-Xm1}
        Q_m(F,F')= \{\, f\in Q_0(F,F') \, |\, s_\alpha(f)\equiv f \, \mod\, (\theta)^m \cdot Q_0(F,F') \, \} \, .
    \end{equation}
    If we have a basis $\{b_i\}_{i=1}^{d}$ of the $\Q$-algebra $Q_0(F,F')$ as a free module over $\Q[V]^W$ of rank $d=\dim_{\Q}H^{\ast}(F;\Q)$ with $b_1=1$, then 
    $\{1, b_i\theta^m,1< i\leq d\}$
    forms a basis of $Q_m(F,F')$ as a free module over $\Q[V]^W$. In particular, $Q_m(F,F')$ is Cohen-Macaulay for any $m\geq 0$.
\end{prop}
We prove \Cref{prop:presentation} using the presentation
\(
Q_m(F,F')= \Q[V]^W + \theta^m \cdot Q_0(F,F').
\)
For simplicity, we write $F_m=F_m(F,F')$ and $X_m=X_m(F,F')$ when there is no ambiguity.
\begin{rmk}
    If we restrict the $G$-actions on the fibers $F$ and $F'$ to $T$-actions and replace the two Borel fibrations \eqref{eq:fib-F0} and \eqref{eq:fib-F'} with 
    \begin{eqnarray}
        \label{eq:fib-F0-hT}
        F \to ET\times_T F \to BT \\
        \label{eq:fib-F'-hT}
        F'\to  ET\times_T F' \to BT
    \end{eqnarray}
    then an analogous result holds for $\bQ_m(F,F')\coloneqq H^\ast_T(F_m;\Q)$. In particular, 
    \begin{equation}
    \label{eq:HT-Xm}
        \bQ_m(F,F')= \{\, f\in \bQ_0(F,F') \, |\, s_\alpha(f)\equiv f \, \mod\, (\theta)^m \cdot \bQ_0(F,F') \, \} \, .
    \end{equation}
    Note $\theta\in H^{\ast}(BG;\Q)$ is the Euler class of the spherical fiber bundle \eqref{eq:fib-F'}. 
    Under the natural map $H^\ast(BG;\Q)\hookrightarrow H^\ast(BT;\Q)$, it can be identified with the image of the Euler class of the spherical fiber bundle \eqref{eq:fib-F'-hT}. We will further identify $\theta$ with its image in the structure map $H^{\ast}(BT)\hookrightarrow H_T(F;\Q)$.
\end{rmk}


\subsection{Proof of main results}
\subsubsection{Rational homotopy type of $F_m$} 
\label{sec:rat-join}
Before proving \Cref{thm:borel-tow}, we need to examine the rational homotopy type of the joins $F_m$ \cite{quillen1969rational}.
By \cite{felix2012rational}*{Theorem 24.5}, if $X$ is a path-connected space with reduced rational homology $\Tilde{H}_{\ast}(X)$ spanned by (finitely many) homology classes $\{v_i\}_{i \in I}$ of dimensions $|v_i| = n_i \geq 1$, then
$\Sigma(X)$ is rationally equivalent to the wedge of spheres $\bS^{n_i+1}$, one for each $i \in I$. By induction on $m$, we have 
\begin{equation} \label{eq:fmjoin}
    F_m = F\ast^m F' \simeq \Sigma(F\ast^{m-1}F')\wedge \bS^{2k-1} \simeq_{\Q} \Sigma\left(\bigvee_{i\in I}\bS^{n_i+2k(m-1)}\right)\wedge \bS^{2k-1} \simeq_{\Q} \bigvee_{i\in I}\bS^{n_i+2km}\, .
\end{equation}
If $F$ has only finite even dimensional rational cohomology, we see that $F_m$ is rationally equivalent to a wedge of even dimensional spheres and $\dim_{\Q}H^{\ast}(F_m;\Q)=\dim_{\Q}H^{\ast}(F;\Q)$. 

\subsubsection{Freeness property}
Since the fiber $F_m$ has only even degree rational cohomology, the Leray-Serre spectral sequence associated to the fibration $F_m \to X_m \to BG$ collapses. 
Applying \cite{mimura1991topology}*{Chapter III Theorem 4.2} to this fibration we see that 
\begin{equation}
\label{eq:LSSS}
     Q_m(F,F') = H^{\ast}(X_m;\Q) = H^{\ast}(F_m;\Q)\bigotimes H^\ast(BG;\Q)
\end{equation}
is a free module over $ H^\ast(BG;\Q) $ of rank $\dim_{\Q}H^{\ast}(F;\Q)$ and is concentrated in even degrees.

\subsubsection{Injectivity property} 
\begin{lem}
\label{lem:coh-pback}
There is an isomorphism of algebras
\begin{equation}
    H^{\ast}(X_m\times_{BG}BH;\Q) \cong H^{\ast}(X_m;\Q)\bigotimes_{H^\ast(BG;\Q)}H^{\ast}(BH;\Q) \,.
\end{equation}
\end{lem}
\begin{proof}
This follows from the Eilenberg-Moore spectral sequence and the fact that $H^{\ast}(X_m;\Q)$ is a finitely generated free module over $H^\ast(BG;\Q)$, see \cite{smith1967homological}*{Proposition 4.3, Corollary 4.4}.
\end{proof}

\begin{cor}
The odd cohomology of $ X_m\times_{BG}BH $ vanishes for any $m$. 
\end{cor}

As $X_{m+1}$ is defined as the homotopy pushout 
\( \hcolim \{X_m \leftarrow X_m\times_{BG} BH\to BH\} \),
we have a (Mayer-Vietoris) long exact sequence in rational cohomology (\cite{dugger2008primer}*{Section 18.1})
\begin{equation}
\label{eq:LES-Coh}
    \ldots \to H^{r}(X_{m+1}) \to H^{r}(X_m)\oplus H^{r}(BH) \to H^{r}(X_m\times_{BG}BH) \to H^{r+1}(X_m) \to \ldots
\end{equation}
In particular, as the spaces $ X_m, X_{m+1}, BH $ and $ X_m\times_{BG}BH $ have vanishing odd degree rational cohomology, the long exact sequence \eqref{eq:LES-Coh} splits into short exact sequences 
\begin{equation*}
    \begin{tikzcd}[sep=small]
    0 \ar[r] & H^{2t}(X_{m+1};\Q) \ar[r] &  H^{2t}(X_m;\Q)\oplus H^{2t}(BH;\Q) \ar[r,"\alpha_m"] & H^{2t}(X_m\times_{BG}BH;\Q) \ar[r] & 0\, .
    \end{tikzcd}
\end{equation*}
 
Therefore we can compute rational cohomology of $ X_m $ inductively by the following formula.
\begin{prop}\label{prop:Xm+1}
For $m\geq 0$, we have 
$$H^{\ast}(X_{m+1};\Q) \cong \,\ker\,(\, \alpha_m:H^{\ast}(X_m;\Q)\oplus H^{\ast}(BH;\Q) \to H^{\ast}(X_m\times_{BG}BH;\Q) \,)\, .$$
\end{prop}

It follows that we can identify $H^{\ast}(X_m;\Q)$ as a subalgebra of $H^{\ast}(X_0;\Q)$.
\begin{prop}\label{prop:inj}
    The map $\pi_m^{\ast}: H^{\ast}(X_{m+1};\Q) \hookrightarrow H^{\ast}(X_m;\Q) $ is injective for every $m\geq0$. 
\end{prop}
\begin{proof}
    We will use induction on $m$ to show that each $ H^{\ast}(X_{m+1};\Q) \hookrightarrow H^{\ast}(X_m;\Q) $ is an embedding. 
    Note we have natural maps of commutative diagrams
    \[\begin{tikzcd}
    X_m & X_m\times_{BG}BH \ar[l]\ar[r] & BH \ar[d,equal] \\
    X_{m-1} \ar[u] & X_{m-1}\times_{BG}BH \ar[u]\ar[l]\ar[r] & BH
    \end{tikzcd}\]
    which induces maps on homotopy colimits $ X_m \to X_{m+1} $, so we have a commutative diagram of short exact sequences 
    \begin{equation*}
        \begin{tikzcd}[sep=small]
        0 \ar[r] & H^{\ast}(X_{m+1};\Q) \ar[d]\ar[r] &  H^{\ast}(X_m;\Q)\oplus H^{\ast}(BH;\Q) \ar[d]\ar[r] & H^{\ast}(X_m\underset{BG}{\times}BH;\Q) \ar[d]\ar[r] & 0 \\
        0 \ar[r] & H^{\ast}(X_m;\Q) \ar[r] &  H^{\ast}(X_{m-1};\Q)\oplus H^{\ast}(BH;\Q) \ar[r] & H^{\ast}(X_{m-1}\underset{BG}{\times}BH;\Q) \ar[r] & 0
        \end{tikzcd}
    \end{equation*}
    where the middle vertical map is identity on $ H^{\ast}(BH;\Q) $, so by snake lemma 
    \[ \ker\left(H^{\ast}(X_{m+1};\Q) \xrightarrow{\pi_m^{\ast}}H^{\ast}(X_m;\Q)\right) \subset \ker\left(H^{\ast}(X_m;\Q) \xrightarrow{\pi_{m-1}^{\ast}} H^{\ast}(X_{m-1}; \Q)\right) \]
    and by induction it suffices to show that $\ker\left(H^{\ast}(X_1;\Q) \xrightarrow{\pi_0^{\ast}}H^{\ast}(X_0;\Q)\right)=0$.
    
    Consider the short exact sequence
    \begin{equation*}
        \begin{tikzcd}[sep=small]
        0 \ar[r] & H^{\ast}(X_1;\Q) \ar[r,"\gamma"] &  H^{\ast}(X_0;\Q)\oplus H^{\ast}(BH;\Q) \ar[r,"\alpha"] & H^{\ast}(X_0\times_{BG}BH;\Q) \ar[r] & 0 
        \end{tikzcd}
    \end{equation*}
    where 
    $\alpha(x,y)=p_{X_0}^{\ast}(x)\otimes1-1\otimes p_{BH}^{\ast}(y)$. 
    For simplicity we write $M = H^{\ast}(X_0;\Q), N = H^{\ast}(BH;\Q),$ and 
    $$\displaystyle{P = H^{\ast}(X_0;\Q)\bigotimes_{H^\ast(BG;\Q)}H^{\ast}(BH;\Q)}\,,$$ 
    then we have 
    \begin{equation*}
        \begin{tikzcd}
            & N \ar[r]\ar[d,"\beta"'] & N\oplus M \ar[r]\ar[d,"\alpha"] & M \ar[r]\ar[d] & 0 \\
            0 \ar[r] & P \ar[r,equal] & P \ar[r] & 0
        \end{tikzcd}
    \end{equation*}
    where all horizontal sequences are exact. Thus by snake lemma there is an exact sequence 
    \begin{equation*}
        \begin{tikzcd}[sep=small]
        0 \ar[r] & \ker(\beta) \ar[r] & \ker(\alpha) \ar[r] & M \ar[r] & \coker(\beta) \ar[r] & 0
        \end{tikzcd}
    \end{equation*}
    It suffices to show $\ker(\beta)=0$. Observe that
    \begin{equation*}
        \beta=p_{BH}^{\ast}:H^{\ast}(BH;\Q) \to \displaystyle{H^{\ast}(X_0;\Q)\bigotimes_{H^\ast(BG;\Q)}H^{\ast}(BH;\Q)}
    \end{equation*}
    is injective because $ H^{\ast}(X_0;\Q) $ is free over $ H^\ast(BG;\Q) $. So we're done with our proof. 
\end{proof}

\subsubsection{Presentation and basis}
Since $G$ is connected, the system of local coefficient rings $\cH^{i}(G/H;\bR)=H^{i}(G/H;\bR)$ is trivial. The Gysin exact sequence associated to the spherical fiber bundle \eqref{eq:fib-F'} gives us $H^{\ast}(BH;\Q)=\Q[V]^W/(\theta)$ where $\theta\in \Q[V]^W$ is the Euler class of the fiber bundle \eqref{eq:fib-F'}.
Combining \Cref{lem:coh-pback} and \Cref{prop:Xm+1}, we see that 
\begin{equation*} 
  H^{\ast}(X_{m+1};\Q) \cong \big\{\, (f, g)\,|\, f \in H^{\ast}(X_m;\Q)\, ,\,g \in H^{\ast}(BH;\Q)\, , g \equiv f \,\mod\,[\theta\cdot H^\ast(X_m,\Q)] \,\big\}\, .\end{equation*}
\Cref{prop:inj} implies that 
\begin{equation*}
    H^{\ast}(X_{m+1};\Q) = H^\ast(BG;\Q) + \theta \cdot H^{\ast}(X_m;\Q) \, .
\end{equation*}
Hence, by \Cref{prop:basis} and \Cref{rmk:EqPre}, we can obtain the desired explicit description of $Q_m(F,F')=H^{\ast}(X_m;\Q)$ and its basis as a free module over $H^\ast(BG;\Q)$.

\subsubsection{\texorpdfstring{$T$}--equivariant rational cohomology}

The natural $G$-action on the fiber $F$ restricts to a $T$-action which is also equivariant formal.

\subsection*{$W$-action}
The canonical map $F \to \pt$ induces a $W$-equivariant map on $T$-equivariant cohomology. Explicitly, the map 
$H^{\ast}(BT;\Q)= H^{\ast}_{T}(\pt;\Q)\to H^{\ast}_{T}(F;\Q)$ 
is given by $f\mapsto f\otimes 1$, and the natural $W$-action on $H^{\ast}_{T}(F;\Q)$ is induced by the $W$-action on $ET\times_T F$. Furthermore $H^\ast_T(F;\Q)^W \cong H^\ast_G(F;\Q)$. 

To obtain the $T$-equivariant cohomology of $G/H$, we first observe the following result (see also \cite{anderson2023equivariant}*{Chapter 3, Example 4.4}). 

\begin{prop}
    Let $G$ be a Lie group and $K$ a Lie subgroup of $G$. For any $G$-space $X$, there are $G$-equivariant isomorphisms
    \[\begin{tikzcd}[sep=tiny]
        \varphi: G\times_{K}X \ar[rr,shift left] && G/K\times X :\phi \ar[ll,shift left]\\
        \left[g,x\right]_{K} \ar[rr,mapsto] && (gK,gx) \\
        \left[g,g^{-1}x\right]_{K} && (gK,x) \ar[ll,mapsto]
    \end{tikzcd}\]
    where the $G$-action on the left hand side is on the first factor and the $G$-action on the right hand side is on the diagonal. 
\end{prop}
\begin{proof}
    It suffices to check the maps $\varphi$ and $\phi$ are well-defined. For any $k\in K$, 
    \[\varphi([gk,k^{-1}x]_{K}) = (gkK, gkk^{-1}x) = (gK,gx)\]
    and
    \[ \phi(gkK,x) = [gk,k^{-1}g^{-1}x]_{K}= [g,g^{-1}x]_{K}\,.\]
\end{proof}
In particular, if we choose $X=G/H$, then we obtain an isomorphism 
\begin{equation}\label{eq:G-ext-iso}
    G\times_{K}G/H \cong G/K\times G/H.
\end{equation}
\begin{prop}[\cite{carlson2015equivariant}*{Proposition 6.2.4}] 
\label{prop:hty-q-homog}
Let $ G $ be a compact Lie group and $ H, K $ be closed subgroups. Then there is a homeomorphism of $ G/H $-bundles over $ BK $
\begin{equation}
\la{eq:GHK} 
\begin{tikzcd}[sep=tiny]
    \varphi: EG \times_{K} G/H \ar[rr,shift left] && BK \times_{BG} BH: \phi \ar[ll,shift left] \\
    \left[e,gH\right]_{K} \ar[rr,mapsto] && (eK,egH) \\
    \left[e,gH\right]_{K} && (eK,egH) \ar[ll,mapsto] 
\end{tikzcd}
\end{equation}
\end{prop}
%
%
\begin{cor}\label{cor:coh-hty-q}
    The rational cohomology of the homotopy quotient $ (G/H)_{hK} $ is given by 
    \[
    H^{\ast}((G/H)_{hK};\Q) \cong H^{\ast}(BK \times_{BG} BH;\Q)  \cong \Tor^{\ast}_{H^\ast(BG;\Q)}\left(H^{\ast}(BK;\Q), H^{\ast}(BH;\Q)\right)\, .
    \]
\end{cor}
\begin{proof}
    The previous proposition gives the first isomorphism. The second isomorphism is similar to \Cref{lem:coh-pback}.
\end{proof}
In particular, when $K=T$ is a maximal torus, there is a natural $W=N_G(T)/T$-action on $EG\times_TG/H$ given by $w\cdot[e,gH]_T=[en^{-1},ngH]$ where $w=nT$. The homeomorphisms in \eqref{eq:GHK} are $W$-equivariant where the $W$-action on $BK\times_{BG}BT$ is induced by the $W$-action on $BT$.
As a corollary, we obtain the following result.
\begin{prop}\label{prop:T-HQ-GH}
    The $T$-equivariant cohomology of $G/H$ is given by 
    \begin{equation*}
    H^{\ast}_{T}(G/H;\Q) = H^\ast(BT;\Q)\bigotimes_{H^\ast(BG;\Q)}H^{\ast}(BH;\Q)= \Q[V]/(\theta)
\end{equation*}
where $\theta$ is the same Euler class as the $G$-equivariant case and we identify it with its image in $\Q[V]=H^\ast(BT;\Q)$. Its $W$-action is inherited from $H^\ast(BT;\Q)=\Q[V]$.
\end{prop}

Similar to the $G$-equivariant case, we can explicit compute $\bQ_m(F,F')=H^\ast_T(F_m;\Q)$ and identify it as a subalgebra of $H^{\ast}_{T}(F;\Q)$ if we replace $BG$ by $BT$ and the two Borel fibrations for join construction by \eqref{eq:fib-F0-hT} and \eqref{eq:fib-F'-hT}. 

\begin{thm}
\la{thm:HTFm}
    The $T$-equivariant cohomology of $F_m$ is given by 
    \begin{equation*}
        \bQ_m(F,F')= \{\, f\in H_T(F;\Q) \, |\, s_\alpha(f)\equiv f \, \mod\, (\theta)^m \cdot H_T(F;\Q) \, \}\, .
    \end{equation*}
    with $W$-action inherited from $H_T(F;\Q)$.
\end{thm}
The other properties for the $T$-equivariant case follow from identical arguments, so we omit the proofs.


\subsubsection{Rational homotopy types}
We can enhance the result of \Cref{thm:borel-tow} from rational cohomology to rational homotopy type.
Recall a space $F$ is \emph{formal} if the rational cohomology $H^\ast(F;\Q)$ of $F$ is quasi-isomorphic to $A_{PL}^\bullet(F)$ the Sullivan-de Rham algebra of $F$ and therefore determines its rational homotopy type. 
There is a different notion of equivariant formality of spaces. 
For a $G$-space $F$, if the Serre spectral sequence associated with the fibration $F\to (F)_{hG}\to BG$ degenerates at the $E_2$-page, $M$ is called \emph{equivariantly formal}. In particular, the $G$-spaces $F_m=F\ast^m G/H$ are equivariantly formal. 
In general the two notions are not related, but in good cases equivariant formality and the formality of the fiber space $F$ implies the formality of $F_{hG}$.

\begin{prop}[\cite{lupton1998variations}*{Proposition 3.2}]
    Let $F\to E \to B$ be a fibration of simply connected spaces of finite type over $\Q$\footnote{A simply connected space is \emph{of finite type} if it has finite-dimensional rational homology groups.}, in which $F$ is formal and rationally elliptic\footnote{A \emph{rationally elliptic space} $F$ is a space such that both $H^n(F;\Q)$ and $\pi_n(X)\otimes \Q$ are finite dimensional. } and $B$ is formal, then $E$ is formal if the associated Serre spectral sequence degenerates at the $E_2$-page.
\end{prop}
\begin{cor}
    If a simply-connected $G$-space $F$ is rationally elliptic, formal and equivariant formal, then $F_{hG}$ is also formal.
\end{cor}
There is a stronger notion of \emph{$G$-formality} introduced in \cite{lillywhite2003formality}, i.e., requiring $A_{PL}^\bullet(F_{hG})$ to be formal in the category $\CDGA_R$ where $R=\Q[V]^W$. Note the canonical map $\eta: A_{PL}^\ast(\pt_{hG})\to A_{PL}^\ast(F_{hG})$ equips $A_{PL}^\ast(F_{hG})$ with an $A_{PL}^\ast(\pt_{hG})$-module structure, and precomposing with the natural inclusion $R \hookrightarrow A_{PL}^\ast(\pt_{hG})$ makes $A_{PL}^\ast(F_{hG})$ a differential graded commutative $R$-algebra structure.

\begin{thm}[\cite{lillywhite2003formality}*{Theorem 5.5}]
    Let $F$ be a simply connected, rationally elliptic $G$-space. If $F$ is formal and equivariantly formal, then $F$ is $G$-formal.
\end{thm}
By \eqref{eq:fmjoin}, $F_m$ is rationally equivalent to a finite wedge of even spheres of dimension at least 2 for $m>0$ and thus simply connected, rationally elliptic and formal. Therefore,
\begin{prop}
\label{prop:formal}
    $F_m$ is formal, $G$-formal and $T$-formal for $m>0$. In particular, $(F_m)_{hG}$ and $(F_m)_{hT}$ are formal for $m>0$.
\end{prop}
It follows that the $G$-equivariant rational cohomology $H_G^\ast(F_m;\Q)$ defines a coaffine stack $\, \cSpec\,[H_G^\ast(F_m;\Q)] \in \cAff_{\Q} \,$ that determines the rational homotopy type of $(F_m)_{hG}$ for any $m>0$. 

\begin{rmk}
Our assumption of $F$ having rational cohomology concentrated in even degrees ensures $F$ is equivariant formal. If, in addition, $F$ is also simply connected of finite type, rationally elliptic and formal, $F$ is also $G$-formal and $T$-formal.
\end{rmk}


\section{Examples}
\label{sec:eg}
In this section, we focus on the two major examples we mentioned in the introduction $F=G/T$ and $F=E_{\com}G$.

\subsection{Flag manifolds}
\label{sec:flag}
Let $EG$ be Milnor's model of universal $G$-bundle and $BG=EG/G$ be a model of classifying space of $G$. As $EG$ is contractible, we can restrict the right $G$-action on $EG$ to obtain a $T$-action, making $EG$ a model of universal $T$-bundle. Consequently, we can use $BT=EG/T$ as our model of classifying space for $T$. Moreover, the Weyl group $W$ naturally acts on $BT$ via: 
\[W\times BT\to BT, \quad (nT,\left[x\right]_{T}) \mapsto \left[xn^{-1}\right]_{T}\,.\]
Let $p: BT \to BG$ be the natural fibration induced by the inclusion map $i: T \hookrightarrow G$. 

\subsubsection{Construction}
Let $ H $ be a closed subgroup of $G$ such that $G/H\cong \bS^{2k-1}$ for some $k\geq 1$. 
Applying \Cref{thm:borel-tow} to the two fibrations $G/T \hookrightarrow BT \to BG$ and $G/H \hookrightarrow BH \to BG$
we obtain a tower of Borel fibrations
\begin{equation*}
    F_m(G/T,G/H)=G/T\ast^m G/H \to X_m(G/T,G/H)=BT\ast_{BG}^m BH \to BG
\end{equation*}
for $m\geq 0$. 
The fiber carries a natural left holonomy action $\Omega BG \times F_m(G/T,G/H) \to F_m(G/T,G/H)$, which under the identification $\Omega BG \simeq G$, corresponds to the diagonal action of $G$ on $F_m(G/T,G/H)$:
\begin{equation}\label{eq:G-act-F-m}
    g\cdot (t_0g_0T+t_1g_1H+\ldots+t_ng_nH) = t_0gg_0T+t_1gg_1H+\ldots+t_ngg_nH \,.
\end{equation}

The Weyl group $W=N_{G}(T)/T$ acts on $G/T$ by $nT\cdot (gT)=gn^{-1}T$, inducing $W$-action on $F_m(G/T,G/H)$ which commutes the $G$-action \eqref{eq:G-act-F-m}, inducing a $W$-action on the homotopy quotient $X_m(G/T,G/H)$ given by 
\begin{equation*}
    nT\cdot[x,t_0g_0T+t_1g_1H+\ldots+t_ng_nH] = [x,t_0g_0n^{-1}T+t_1g_1H+\ldots+t_ng_nH]
\end{equation*}
where $x\in EG$ and $nT\in W$. The inclusions $F_m(G/T,G/H)\hookrightarrow F_{m+1}(G/T,G/H)$ given by 
\begin{equation*}
    t_0g_0T+t_1g_1H+\ldots+t_ng_nH \mapsto t_0g_0T+t_1g_1H+\ldots+t_ng_nH + 0 
\end{equation*}
are $G\times W$-equivariant, and therefore induce $W$-equivariant maps 
$$\pi_m: X_m(G/T,G/H) \to X_{m+1}(G/T,G/H)\,.$$ 
The whisker maps $p_m: X_m(G/T,G/H)\to BG$ are also $W$-equivariant.  

In particular, we see that
$H^{\ast}_W(F_m(G/T,G/H);\Q) \cong \Q$
using the same argument as the rank one case in \cite{berest2023topological} if we replace $G$ by $G/H$, and it follows that we have an algebra isomorphism
\begin{equation*}
    Q_m(G/T,G/H)^W  = H^{\ast}_W(X_m(G/T,G/H);\Q)\cong H^{\ast}(BG;\Q) = \Q[V]^W \,.
\end{equation*}

Gathering all the properties discussed above, we can now formulate a result parallel to \cite{berest2023topological}*{Theorem 3.9}.

\begin{prop}\label{prop:rel-quasi}
    Let $G$ be a compact connected Lie group with torsion free fundamental group, with maximal torus $T \subseteq G $ and Weyl group $ W_G = N_G(T)/T $, and a closed subgroup $H$ such that $G/H\cong \bS^{2k-1}$. 
    Then we have a diagram of $G\times W$-spaces 
    \begin{equation}
    \label{eq:Fm-diag}
    \begin{tikzcd}[sep=small]
        G/T = F_0 \ar[r] & \ldots \ar[r] & F_m \ar[r,"\pi_m"] & F_{m+1} \ar[r,"\pi_{m+1}"] & \ldots 
    \end{tikzcd}
    \end{equation} 
    where $F_m=F_m(G/T,G/H)=(G/T)\ast^m (G/H)$ equipped with diagonal $G$-action and induced $W$-action from $G/T$.  Then 
    \begin{enumerate}
        \item $F_m$ is formal, $G$-formal and $T$-formal for all $m\geq0$.
        \item $Q_m(G/T,G/H)=H^\ast_G(F_m;\Q)$ is a free module over $\Q[V]^W=H^{\ast}(BG;\Q)$ of rank $|W|$ for all $m\geq0$.
        \item $\bQ_m(G/T,G/H)=H^\ast_T(F_m;\Q)$ is a free module over $\Q[V]=H^{\ast}(BT;\Q)$ of rank $|W|$ for all $m\geq0$.
        \item $Q_m(G/T,G/H)^W=\Q[V]^W$. 
        \item $\bQ_m(G/T,G/H)^W=Q_m(G/T,G/H)$.
        \item $\bQ_m(G/T,G/H)=Q_m(G/T,G/H)\otimes_{\Q[V]^W}\Q[V]$.
        \item The diagram \eqref{eq:Fm-diag} induces a $W$-equivariant descending filtration on $H^\ast_G$
        \begin{equation*}
        \Q[V] \hookleftarrow  Q_1(G/T,G/H)  \hookleftarrow \ldots  
        \hookleftarrow Q_m(G/T,G/H) \xhookleftarrow{\,\pi_m^*\,} Q_{m+1}(G/T,G/H) \hookleftarrow \ldots
        \end{equation*}
        such that  $ \,\varprojlim \,Q_m(G/T,G/H) \cong \Q[V]^W $. 
        \item The diagram \eqref{eq:Fm-diag} induces a $W\times W$-equivariant descending filtration on $H^\ast_T$
        \begin{equation*}
        \Q[V]\otimes_{\Q[V]^W}\Q[V] \hookleftarrow  \bQ_1(G/T,G/H)  \hookleftarrow \ldots  
        \hookleftarrow \bQ_m(G/T,G/H) \xhookleftarrow{\,\pi_m^*\,} \bQ_{m+1}(G/T,G/H) \hookleftarrow \ldots
        \end{equation*}
        such that  $ \,\varprojlim \,\bQ_m(G/T,G/H) \cong \Q[V]$. 
    \end{enumerate}
\end{prop}
In particular, we can recover the examples of coaffine stacks of relative quasi-invariants defined in \Cref{sec:cAff}.
\begin{thm}
\label{thm:Qm-rht}
    Let $G$ be a compact connected Lie group with torsion free fundamental group, with maximal torus $T \subseteq G $ and Weyl group $ W_G = N_G(T)/T $, and a closed subgroup $H$ such that $G/H\cong \bS^{2k-1}$. Let $S=T\cap H$ be the maximal torus of $H$, and Weyl group $ W_H = N_H(S)/S $.
    Then, for any $ m \in \bZ_+ $, 
    there are natural equivalences in $\cAff_{\Q} $:
\begin{eqnarray}
    \cSpec \,[C^\ast(X_m(G/T,G/H);\Q)\,] & \simeq & V^c_m(W_G,W_H)\, , 
\end{eqnarray}
where $X_m(G/T,G/H)=EG\times_G F_m(G/T,G/H)$ and $C^\ast(X_m(G/T,G/H);\Q)$ is the rational cochain complex of $X_m(G/T,G/H)$, and $V^c_m(W_G, W_H)$ is defined in \eqref{eq:Vcm}.
\end{thm}
\begin{proof}
    As $X_m(G/T,G/H)$ is formal, 
    \[\cSpec \,[C^\ast(X_m(G/T,G/H);\Q)\,] \simeq \cSpec \,[H^\ast(X_m(G/T,G/H);\Q)\,] \,, \]
    and we can prove by induction on $m$. Note $H^\ast(BT;\Q)=\Q[V]$ and $H^\ast(BS;\Q)=\Q[V_H]$ corresponds to $V^c$ and $V_H^c$ respectively.
\end{proof}

\subsubsection{\texorpdfstring{$G$}--equivariant rational cohomology}
An explicit presentation of the $G$-equivariant cohomology of $F_m(G/T,G/H)$ follows from \Cref{prop:presentation}. 
\begin{prop}
Let $\theta\in H^{\ast}(BG;\Q)=\Q[V]^W$ be the Euler class of the spherical fiber bundle \eqref{eq:fib-F'}, then 
\begin{equation}
\label{prop:QmGTH}
Q_m(G/T,G/H) = \{\, f\in \Q[V] \, |\, s_{\alpha}(f) \equiv f \, \mod\, (\theta)^m \,, \,\forall s_{\alpha}\in \cA_W \,\}
\end{equation}
where $\cA_W$ is the set of reflections in $W$.
\end{prop}

\begin{cor}\label{cor:hilb}
    The Hilbert series of $X_m(G/T,G/H)$ is given by 
    \begin{equation*}
    p_{X_m}(t) =  \frac{1-t^{2mk}}{\prod_{i=1}^n(1-t^{2d_i})} + \frac{t^{2mk}}{(1-t^2)^n}
    \end{equation*}
    where $d_i$ are the degrees of the Weyl group $W=N_{G}(T)$.
\end{cor}
\begin{proof}
    The isomorphism in \eqref{eq:LSSS} implies that the Hilbert series of $X_m(G/T,G/H)$ is given by 
    $$p_{X_m(G/T,G/H)}(t) = p_{F_m(G/T,G/H)}(t) p_{BG}(t).$$ 
    For a compact connected Lie groups $ G $, $p_{BG}(t) = \prod_{i=1}^n\frac{1}{1-t^{2d_i}}$ 
    where $ d_i $ are the degrees of the homogeneous algebraically independent generators of $ H^{\ast}(BG) \cong \Q[V]^{W} $, so it suffices for us to compute $p_{F_m}(t)$. 
    
    By \eqref{eq:fmjoin}, the Hilbert series of $ F_m $ is given by 
    \begin{equation*}
        q_{F_m}(t) = 1 + \sum_{w\in W\backslash{\{e\}}}t^{2l(w)+2mk} = 1- t^{2mk} + t^{2mk} \sum_{w\in W}(t^2)^{l(w)}.
    \end{equation*}
    We can write
    \begin{equation*}
        W(t) = \sum_{w\in W}t^{l(w)} = \prod_{i=1}^n\frac{1-t^{d_i}}{1-t}
    \end{equation*}
    (see, e.g., \cite{humphreys1990reflection}*{Theorem 3.15}). 
    Replacing $ t $ by $ t^2 $ and plugging $ W(t^2) $ in, we get 
    \begin{equation*}
        p_{F_m}(t) = 1- t^{2mk} + t^{2mk} W(t^2) = 1-t^{2mk} + t^{2mk}\prod_{i=1}^n\frac{1-t^{2d_i}}{1-t^2}
    \end{equation*}
    and the desired Hilbert series of $X_m(G/T,G/H)$.
\end{proof}

\subsubsection{\texorpdfstring{$T$}--equivariant rational cohomology}
By \Cref{thm:HTFm}, we have
\begin{prop}
    The $T$-equivariant cohomology of $F_m(G/T,G/H)$ is given by 
    \begin{equation*}
        \bQ_m(G/T,G/H)= \{\, f\in \Q[V]\bigotimes_{\Q[V]^W}\Q[V] \, |\, s_\alpha(f)\equiv f \, \mod\, (\theta)^m \, \}
    \end{equation*}
    where the $W$-action is taken on the first tensor factor.
    Moreover, the natural $W$-action on $G/T$ induces a second $W$-action on both $Q_m$ and $\bQ_m$. With respect to this action, we obtain a $W$-equivariant isomorphism 
    \begin{equation*}
        \bQ_m(G/T,G/H)^{W} \cong Q_m(G/T,G/H) \,.
    \end{equation*}
\end{prop}
In other words, the $G$-equivariant cohomology can be recovered from the $T$-equivariant cohomology by taking $W$-invariants, while retaining the additional $W$-action from $G/T$.

\subsubsection{\texorpdfstring{$G$}--equivariant \texorpdfstring{$K$}--theory}

For a compact Lie group $G$ acting continuously on a compact topological space $X$, $K_{G}(X)$ is defined to be the Grothendieck group of $G$-equivariant complex topological vector bundles over $X$. As shown in\, \cite{segal1968classifying}, this construction extends to a $\bZ/2$-graded multiplicative generalized cohomology theory $K_{G}^{\ast}$ on the category of locally compact $G$-spaces, known as $G$-equivariant topological complex $K$-theory.

We write $K_{G}^{\ast}(X)=K_{G}^{0}(X) \oplus K_{G}^{1}(X)$ with the convention that $K_{G}^{0}(X) \cong K_{G}^{2n}(X)$ and $K_{G}^{1}(X)\cong K_{G}^{2n+1}(X)$ for all $n\in\bZ$. 
When $G$ is trivial, $K_{G}^{\ast}(X)$ coincides with the ordinary complex $K$-theory $K^{\ast}(X)$ of $X$ and when $X$ is a point, $K_{G}^{\ast}(\pt)=R(G)$ is the representation ring of $G$, concentrated in even degrees. Functoriality of $K_{G}^{\ast}$ endows $K_{G}^{\ast}(X)$ with a canonical $R(G)$-module structure, induced by the unique map $X\to\pt$. For further background on equivariant $K$-theory and Atiyah-Segal completion theorem that relates $K_{G}^{\ast}(X)$ and $K(X_{hG})$, we refer the reader to\, \cites{segal1968classifying, atiyah1969equivariant}.

The equivariant $K$-theory of the spaces $F_m(G/T,G/H)$ are closely analogous to the rank-one case studied in\, \cite{berest2023topological}.
Our computation relies on two classical results. 
The first result is a K\"{u}nneth type of formula for equivariant $K$-theory first studied by Hodgkin, which requires $\pi_1(G)$ to be torsion free. 
\begin{thm}[\cite{hodgkin1975equivariant}, \cite{brylinski2000equivariant}*{Theorem 2.3}]
\label{thm:hodgkin}
For a compact connected Lie group $ G $ such that $ \pi_1(G) $ is torsion free, for any $ G $-spaces $ X,Y $, there is a spectral sequence $ E_{r} \Rightarrow K^{\ast}_{G}(X\times Y) $ with $ E_2 $-term 
\[ E_2^{\ast,\ast} = \Tor_{R(G)}^{\ast,\ast}(K_{G}^{\ast}(X),K_{G}^{\ast}(Y)) \]
that converges to $K_{G}^{\ast}(X\times Y)$ where $X\times Y$ is viewed as a $G$-space with diagonal action.
\end{thm}
The second result is a Mayer-Vietoris type of formula (see, e.g., \cite{berest2023topological}*{Lemma 5.2}). 
\begin{lem}
\label{lem:h-pushout}
Let $ f: U \to X $ and  $ g: U \to Y $ be proper equivariant maps of $ G $-spaces. Let $ Z = \hcolim \{ X \xleftarrow{f} U \xrightarrow{g} Y \} $ where the homotopy colimit is taken in the category of $ G $-spaces. Then the abelian groups $ K_{G}^{\ast}(X), K_{G}^{\ast}(Y) $ and $ K_{G}^{\ast}(Z) $ are related by the following six-term exact sequence
\[\begin{tikzcd}
K_{G}^{0}(Z) \ar[r] & K_{G}^{0}(X) \oplus K_{G}^{0}(Y) \ar[r,"f^{\ast}-g^{\ast}"] & K_{G}^{0}(U) \ar[d,"\p"] \\
K_{G}^{1}(U)\ar[u,"\p"] & K_{G}^{1}(X) \oplus K_{G}^{1}(Y) \ar[l,"f^{\ast}-g^{\ast}"'] & K_{G}^{1}(Z) \,. \ar[l]
\end{tikzcd}\]
\end{lem}

Combining the above two results we can show the following proposition.
\begin{prop}\label{prop:K_GFm}
    If $R(H) = R(G)/(\Theta)$ for some $\Theta\in R(G)$, then the $G$-equivariant $K$-theory of $F_m=G/T\ast^mG/H$ (as an Abelian group) is given by 
    \[K_{G}(F_m) = \{\, f\in R(T) \, |\, s_\alpha(f) \equiv f \, \mod\, (\Theta)^m \, \}\,.\]
\end{prop}
\begin{proof}
    As explained in the \Cref{rmk:EqPre}, we will prove for the alternative presentation
    \begin{equation*}
        K_{G}(F_m) = R(G) + \Theta^m \cdot R(T)
    \end{equation*}
    where $R(G)=R(T)^{W}$ is identified with its image in $R(T)$ induced by the map $i:T\hookrightarrow G$.
    
    We will prove it by induction on $m$.
    For simplicity let's write $\cQ_m \coloneqq R(G) + \Theta^m\cdot R(T)$ which can be viewed as a subalgebra of $R(T)$. 
    When $m=0$, we have $ K_{G}^{\ast}(F_0) = K_{G}^{\ast}(G/T) = R(T)$. 
    When $ m \geq 1 $, we have
    \[
    F_{m+1} = \hcolim \, \{\, F_m \leftarrow F_m\times G/H \to G/H \, \} \,.
    \]
    To compute $ K_{G}(F_m\times G/H) $, note the Tor group 
    \[
    \Tor_{R(G)}^{\ast}(K_{G}^{\ast}(F_m),K_{G}^{\ast}(G/H)) = \Tor_{R(G)}^{\ast}(K_{G}(F_m),R(H))
    \]            
    can be identified with $ \Tor^\ast_{R(G)}\left(\cQ_m, R(H)\right)$ which is given by the homology of the complex
    \[
    0 \to \cQ_m \xrightarrow{\cdot \Theta} \cQ_m \to 0 \, .
    \]
    Since $ \cQ_m $ is an integral domain the first homology of the above complex vanishes, thus by \Cref{thm:hodgkin}, the Hodgkin's spectral sequence collapses on the $ E_2 $-page and we have 
    \[ 
    K_{G}^{\ast}(F_m\times G/H) \cong  \cQ_m/\left( \Theta \right)\, . 
    \]
    The projection $ F_m\times G/H \to F_m $ induces the quotient map 
    \(
    \pi: \cQ_m \to \cQ_m/\left( \Theta \right)
    \)
    on equivariant $ K $-theory, and $ F_m\times G/H \to G/H $ induces an injective map
    \(
    i: R(G)/\left(\Theta\right) \hookrightarrow \cQ_m/\left( \Theta \right) .
    \)
    
    By \Cref{lem:h-pushout}, we have the following six term exact sequence
    \[\begin{tikzcd}
    0\oplus 0 \ar[r] & 0 \ar[r] & K_{G}^{0}(F_{m+1}) \ar[d,"{(i_{m,m+1},f_{m+1})}"] \\
    K_{G}^{1}(F_{m+1}) \ar[u] & \cQ_m/\left( \Theta \right) \ar[l,"\partial"'] & \cQ_m \oplus R(G)/\left(\Theta\right) \ar[l,"\pi-i"'] 
    \end{tikzcd}\]
    It can be seen that $ \pi-i $ is surjective and thus $ K^{1}_{G}(F_{m+1}) = 0 $, and $K_{U(n)}^{0}(F_{m+1}) = \ker(\pi-i)$. 
    Furthermore, in this case we see the composite
    \begin{equation*}
        \begin{tikzcd}
            K_{G}^{0}(F_{m+1}) \ar[rr,"{(i_{m,m+1},f_{m+1})}"]  && \cQ_m \oplus R(G)/\left(\Theta\right) \ar[r,"pr_1"] & \cQ_m
        \end{tikzcd}
    \end{equation*}
    is also injective and thus as in \Cref{prop:inj}, we have 
    \[ K_{G}^{0}(F_{m+1}) = \ker(\pi-i) = R(G) + \Theta \cdot \cQ_m = \cP_{m+1}\, . \]
\end{proof}

\subsubsection{\texorpdfstring{$T$}--equivariant \texorpdfstring{$K$}--theory}

By the theorem of Kostant-Kumar\, \cite{kostant1987t}, the $T$-equivariant $K$-theory of the flag manifold $G/T$ is given by $K_{T}(G/T) \cong R(T)\otimes_{R(T)^{W}}R(T)$. Further, under this isomorphism, the homomorphism $R(T)=K_{T}(\pt) \rightarrow K_T(G/T)$ corresponding to the map $G/T \rightarrow \pt$ is identified with the map 
\begin{equation} \label{rttoktgt} R(T) \hookrightarrow R(T) \otimes_{R(T)^W} R(T)\, ,\qquad f \mapsto f \otimes 1\, . \end{equation}
Hence, the natural $R(T)=K_{T}(\pt)$-module structure on $K_T(G/T)$ is given by multiplication on the first copy of $R(T)$. 
By \eqref{eq:G-ext-iso}, we can compute the $T$-equivariant $K$-theory of $G/H$ by observing 
\begin{equation}
    K_{T}(G/H) \cong K_{G}(G\times_{T}G/H) \cong K_{G}(G/T\times G/H) \, ,
\end{equation}
therefore by \Cref{thm:hodgkin}, we obtain the following result for $K_{T}(G/H)$. 
\begin{prop}\label{prop:K_T(G/H)}
    If $R(H) = R(G)/(\Theta)$ for some $\Theta\in R(G)$ then $K_{T}(G/H) \cong R(T)/(\Theta)$.
\end{prop}
Combining the Kostant-Kumar theorem with \Cref{prop:K_T(G/H)}, we can compute the $T$-equivariant $K$-theory of $F_m$ analogously to the computation of the $G$-equivariant $K$-theory.
\begin{prop}
    If $R(H) = R(G)/(\Theta)$ for some $\Theta\in R(G)$, then the $T$-equivariant $K$-theory of $F_m=G/T\ast^mG/H$ is given by 
    \[
    K_T(F_m) = \{\, f\in R(T)\otimes_{R(T)^{W}}R(T) \, |\, s_\alpha(f) \equiv f \, \mod\, (\Theta)^m \, \} \,.
    \]
\end{prop}
The natural $W$-action is on the first copy of $R(T)$, and we have $K_{G}(F_m) = K_{T}(F_m)^{W}$.

\subsubsection{Examples}
\label{sec:examples}

The following classification theorem of homogeneous spheres gives us a complete list of $F'$ we should consider. 
A similar table can be found in Salzmann et al. \cite{salzmann2011compact}*{96.22},
Biller \cite{biller1999actions}*{Theorem 3.1.1}, 
Grundh\"{o}fer-Knarr-Kramer \cite{grundhofer2000flag}*{Proposition 1.1}, 
or Bletz-Siebert \cite{bletz2002homogeneous}*{Theorem 2.4.1}.

\begin{thm}
\label{thm:RHS}
    Let $G$ be a compact connected Lie group and $H$ a closed subgroup such that $G$ acts effectively on $G/H\cong S^{n}$, then there are (up to equivalence) only the following possibilities.
    \begin{enumerate}
        \item $G=\SO(n), H=\SO(n-1), G/H=\bS^{n-1}$.
        \item $G=\SU(k), H=\SU(k-1), G/H=\bS^{2k-1}, k\geq 2$.
        \item $G=\U(k),H=\U(k-1), G/H= \bS^{2k-1}, k\geq 2$. 
        \item $G=\U(1)\cdot\SP(k), H=\U(1)\cdot\SP(k-1), G/H=\bS^{4k-1}, k\geq 2$.
        \item $G=\SP(k), H=\SP(k-1), G/H=\bS^{4k-1}, k\geq 2$.
        \item $G=\SP(1)\cdot\SP(k), H=\SP(1)\cdot\SP(k-1), G/H=\bS^{4k-1}, k\geq 2$.
        \item $G=\Spin(9), H=\Spin(7), G/H=\bS^{15}$.
        \item $G=\Spin(7), H=G_2, G/H=\bS^7$.
        \item $G=G_2,H=\SU(3), G/H\cong S^6$.
    \end{enumerate}
\end{thm}
If we restrict to compact connected Lie groups with torsion-free fundamental group and odd spheres, we only need to consider the five cases 
\[\begin{array}{lllll}
    (\SU(k),\SU(k-1)) & (\U(k),\U(k-1)) & (\SP(k),\SP(k-1)) &
    (\Spin(9), \Spin(7)) & (\Spin(7), G_2)
\end{array}\]
as given in the Introduction.

In what follows, we present the examples $(\U(k),\U(k-1))$ and $(\Spin(7), G_2)$. The remaining classical cases are entirely analogous and may be found in the second author's thesis \cite{liu2023towers}.

\begin{eg}\label{eg:U}
    For $G=\U(n), H= \U(n-1)$ and $F_m=F_m(\U(n)/T,\U(n)/\U(n-1))$, 
    \begin{enumerate}
        \item The $T$-equivariant rational cohomology of $F_m$ is 
        \[
        H_T^\ast(F_m;\Q) = \left\{\, f \in\Q[t_1,\ldots,t_n]\bigotimes_{\Q[c_1,\ldots,c_n]}\Q[t_1,\ldots,t_n] \, \middle|\, s_{ij}(f) \equiv f \, \mod\, (c_n)^m, \forall s_{ij}\in S_n \right\}
        \]
        where $c_i=\sigma_i(t_1,\ldots,t_n)$ is the $i$-th elementary symmetric polynomial, and $s_{ij}$ acts by switching the $i$-th and $j$-th indices in the second factor.
        \item The $G$-equivariant rational cohomology of $F_m$ is 
        \[
        H_G^\ast(F_m;\Q) = \{\, f \in \Q[t_1,\ldots,t_n] \, |\, s_{ij}(f) \equiv f \, \mod\, (c_n)^m, \forall s_{ij}\in S_n \, \}
        \]
        where $c_n=t_1t_2\ldots t_n$ and $s_{ij}$ acts by switching the $i$-th and $j$-th indices.
        \item The $T$-equivariant $ K $-theory of $F_m$ is given (as abelian group) by 
        \[
        K_T^\ast(F_m;\Q) = \left\{\, f \in\bZ[z_1^\pm,\ldots,z_n^\pm] \bigotimes_{\bZ[c_1,\ldots,c_n,c_n^{-1}]} \bZ[z_1^\pm,\ldots,z_n^\pm] \middle| s_{ij}(f) \equiv f \, \mod\, (1\otimes \Theta)^m \right\}
        \]
        where $c_i=\sigma_i(z_1,\ldots,z_n)$ is the $i$-th elementary symmetric polynomial, and
        $$\displaystyle{\Theta = \sum_{i=0}^{n}(-1)^{i}c_{i}= \prod_{i=1}^n(1-z_i)}\,.$$
        \item The $G$-equivariant $ K $-theory of $F_m$ is given (as abelian group) by 
        \[
        K_G^\ast(F_m) = \{\, f\in \bZ[z_1^\pm,\ldots,z_n^\pm] \, |\, s_{ij}(f) \equiv f \, \mod\, (\Theta)^m, \forall s_{ij}\in S_n \, \}
        \]
        where $s_{ij}$ acts by switching the $i$-th and $j$-th indices.
    \end{enumerate}
\end{eg}

\begin{rmk}
For the pair $(\Spin(9), \Spin(7))$, as the Spin groups $\text{Spin}(n)$ are two-fold covering spaces of $\SO(n)$, the rational cohomology of $B\Spin(9)$ and $B\Spin(7)$ agree with the ones of $B\SO(9)$ and $B\SO(7)$ respectively, whose the equivariant cohomology computation is the same as the pair $(\SP(4), \SP(3))$. 
Meanwhile, the condition $R(H)=R(G)/(\Theta)$ for some $\Theta\in R(G)$ in \Cref{prop:K_GFm} fails for $G=\Spin(9), H=\Spin(7)$, as the natural map $R(\Spin(9))\to R(\Spin(7))$ is not surjective. So we skip the $K$-theory computation.
\end{rmk}

\begin{eg}
\label{eg:Spin-G2}
    Consider $(\Spin(7), G_2)$ and $F_m=F_m(\Spin(7)/T,\Spin(7)/G_2)$.
    
    We have 
    \[\begin{array}{rcl}
        H^\ast(BT;\Q)&=&\Q[t_1,t_2,t_3] \,, \\
        H^\ast(B\Spin(7);\Q) &=&\Q[q_1,q_2,q_3]\,, \\
        H^\ast(BG_2;\Q) &=&H^\ast(B\Spin(7);\Q)/(q_2)= \Q[q_1,q_3]
    \end{array}\]
    where $q_1=t_1^2+t_2^2+t_3^2, q_2=t_1^2t_2^2+t_1^2t_3^2+t_2^2t_3^2$ and $q_3=t_1^2t_2^2t_3^2$. Therefore equivariant rational cohomology can be computed using our results.

    Consider the map $R(\Spin(7))\to R(G_2)$ of representation rings induced by subgroup inclusions (see also \cite{yokota1968representation}).
    $G_2$ has two fundamental representations: a 7-dimensional standard representation $i\Q$ given by imaginary octonions, and the adjoint representation $\mathfrak{g}_2$ which is 14-dimensional (see, e.g., \cite{fulton2013representation}). If we write $\alpha=[i\Q],\beta=[\mathfrak{g}_2]\in R(G_2)$, then $R(G_2)=\bZ[\alpha,\beta]$.
    
    Let $V$ be the standard representation of $SO(2n+1)$, and denote $\Lambda^i\coloneqq \Lambda^iV, 1\leq i \leq n$ and let $\Delta^n$ be the spin representations of $\Spin(2n+1)$. Write $\lambda_i=[\Lambda^i]\, ,\Delta=[\O]\in \Spin(7)$, then 
    \[
    R(\Spin(2n+1))=\bZ[\lambda^1,\lambda^2,\ldots,\lambda^{n-1},\Delta^n]
    \]
    satisfying 
    $\Delta^n \cdot \Delta^n = \lambda^n +\lambda^{n-1} + \lambda^{n-2} + \ldots + \lambda^1 + 1$
    (see, e.g., \cite{brocker2003representations}).
    
    The spin representation of $\Spin(7)$ is $\O\cong \bR^8$, which restricts to a transitive $\Spin(7)$-action on $S^7\subseteq \bR^8$ with stabilizer (of any unit vector) a $G_2$ subgroup of $\Spin(7)$, and in fact all $G_2$-subgroups of $\Spin(7)$ are obtained in this way. In particular, $\Spin(7)/G_2 \cong S^7$ and all $G_2$-subgroups of $\Spin(7)$ are conjugate (see, e.g., \cite{varadarajan2001spin}*{Theorem 3}).
    Therefore the spin representation of \( \mathrm{Spin}(7) \) restricts to \( G_2 \) decomposes as $\O |_{G_2} = \mathbf{1} \oplus i\O $ where \( i\O \) is the imaginary octonions viewed as the standard representation of $G_2$ and $\mathbf{1}$ is the trivial $G_2$-action on the reals $\bR\subseteq \Q$.
    
    By \cite{varadarajan2001spin}*{Lemma 1}, $G_2$ can also be embedded in $\SO(7)$.
    For the standard representation $\Lambda^1$ of $\Spin(7)$ coming from the standard representation of $\SO(7)$, it restrict to an $\SO(7)$-action on $S^6$. The stabilizer of any unit vector is $\SO(6)$ and we have $\SO(7)/\SO(6)\cong S^6$.
    The restriction of $\SO(7)$-action on $S^6$ to $G_2$ is nontrivial as $G_2\not\subseteq \SO(6)$, thus $G_2$ acts non-trivially on $\Lambda^1$, showing that $\Lambda^1$ restricts to the imaginary octonion representation of $G_2$. 
    
    Finally, the adjoint representation $\Lambda^2 =\mathfrak{so}(7)$ of $\Spin(7)$ has a copy of $\mathfrak{g}_2$ in it, thus we only need to determine how the remaining $7$-dimensional representation $V$ decomposes as a $G_2$-representation. 
    Use the restriction of the spin representation $\O$ of $\Spin(7)$ to $S^7$ and at $e=(1,0,\ldots,0)\in S^7\subseteq\O$, the restricted $G_2$-action on $T_e(S^7) \cong V$ is precisely the $G_2$ action on imaginary octonions, thus $V\cong i\O$.
    %
    Therefore the restriction map
    \[
    R(\Spin(7))=\bZ\left[\lambda_1,\lambda_2,\Delta\right] \to \bZ\left[\alpha,\beta\right] = R(G_2) 
    \]
    given by $\lambda_1 \mapsto \alpha, \lambda_2 \mapsto \alpha+\beta, \Delta \mapsto 1+\alpha$
    is surjective with kernel generated by \( \Theta=\lambda^1 - \Delta + 1 \) and we have $R(G_2)=R(\Spin(7))/(\Theta)$. 

    Combining what we have, we get
    \begin{enumerate}
        \item The $T$-equivariant rational cohomology of $F_m$ is 
        \[H_T^\ast(F_m;\Q) = \left\{\, f\in \Q[t_1,t_2,t_3]\bigotimes_{\Q[q_1,q_2,q_3]}\Q[t_1,t_2,t_3] \middle| s_\alpha(f) \equiv f \, \mod\, (q_2)^m, \forall s_\alpha\in B_3 \right\}\]
        where $s_\alpha$ acts on the second factor.
        \item The $G$-equivariant rational cohomology of $F_m$ is 
        \[
        H_G^\ast(F_m;\Q) = \{\, f\in \Q[t_1,t_2,t_3] \, |\, s_\alpha(f)\equiv f \, \mod\, (q_2)^m, \forall s_\alpha\in B_3 \, \} \,.
        \]
        \item The $T$-equivariant $ K $-theory of $F_m$ is 
        \[
        K_{T}^{\ast}(F_m) = \left\{\, f\in \bZ[z_1^\pm,z_2^\pm,z_3^\pm]\bigotimes_{\bZ[\lambda_1,\lambda_2,\Delta]}\bZ[z_1^\pm,z_2^\pm,z_3^\pm] \middle| s_{\alpha}(f)\equiv f \, \mod\, (\Theta)^m, \forall s_\alpha\in B_3 \right\}
        \]
        where $\lambda_i=\sigma_i(z_1,z_1^{-1},z_2,z_2^{-1},z_3,z_3^{-1},1), 1\leq i \leq 3$, $\Delta \cdot \Delta= 1+\lambda_1+\lambda_2+\lambda_3$, and 
        $\displaystyle{\Theta = 1+\lambda_1-\Delta.}$
        Note $s_\alpha$ acts on the second factor.
        \item The $G$-equivariant $ K $-theory of $F_m$ is
        \[
        K_{\Spin(7)}^{\ast}(F_m) = \{\, f\in \bZ[z_1^\pm,z_2^\pm,z_3^\pm] \, |\, s_{\alpha}(f)\equiv f \, \mod\, (\Theta)^m, \forall s_\alpha\in B_3 \, \} \,.
        \]
    \end{enumerate}
\end{eg}

\subsection{Classifying spaces for commutativity}
\label{sec:comm}

Instead of considering the fibration $G/T\to BT\to BG$, we are also interested in other $G$-spaces $F$ that are equivariantly formal.
One natural candidate is $E_{\com}G_1$, the identity component of the universal bundle of the classifying space of commuting elements in a Lie group $G$ introduced in \cite{adem2012commuting}. 

\begin{defn}[\cite{adem2012commuting}]\label{defn:Bcom}
Let $ G $ be a topological group and write $BG_{\ast}$ the simplicial classifying space of $G$. Consider the simplicial subspace $ B_{\com}(G)_{\ast} = \{\Hom(\bZ^n,G)\}_{n\geq0} $ of $BG_{\ast}$, and take its geometric realization produces a topological space $B_{\com}G$, called a \textit{classifying space of commutativity} of $G$. 
The simplicial space $E_{\com}G_{\ast} = \{ \Hom(\bZ^n,G)\times G \subset G^{n+1} \}_{n\geq 0}$ can be defined analogously with face and degeneracy maps similar to the simplicial universal $G$-bundle $EG_{\ast}$, its geometric realization is denoted by $E_{\com}G$.
\end{defn}
The projection on the first $n$-coordinates defines a simplicial map 
$E_{\com}(G)_{\ast} \to B_{\com}(G)_{\ast}$ 
which induces a continuous map on geometric realizations 
$p_{\com}:E_{\com}G \to B_{\com}G$. 
Let $E_{\com}G_1$ and $B_{\com}G_1$ denote their identity component respectively, then we have  the following diagram of morphisms of principal $ G $-bundles 
\[\begin{tikzcd}
E_{\com}G_1 \ar[r]\ar[d,"p_{\com}"'] & EG \ar[d,"p"] \\
B_{\com}G_1 \ar[r,"i"] & BG
\end{tikzcd}\]
thus up to homotopy it gives rise to a fibration sequence 
\begin{equation}\label{eq:fib-com}
    E_{\com}G_1 \to B_{\com}G_1 \to BG \, .
\end{equation}
If $ G $ is connected, $ B_{\com}G $ is simply-connected. 
The homotopy theoretical properties of these spaces are studied in \cite{adem2015classifying}. In particular, these spaces admits a homotopy colimit decomposition and they classify transitionally commutative bundles, i.e. principal bundles that have open covers with transition functions commute with each other whenever they are simultaneously defined.

The rational cohomology of $B_{\com}G$ is given in the following result. 
\begin{prop}[\cite{adem2015classifying}*{Proposition 7.1}]
    Suppose that $G$ is a compact connected Lie group with $T\subset G$ a maximal torus and associated Weyl group $W$, then there is a natural isomorphism of rings
    \begin{equation}
    \label{eq:coh-BGcom}
        \alpha_{G}: H^{\ast}(B_{\com}G;\Q) \xrightarrow{\cong} (\Q[V]\otimes \Q[V])^{W}/\cJ_{G}
    \end{equation}
    where $W$ acts diagonally on $\Q[V]\otimes \Q[V]$ and $\cJ_{G}$ is the ideal generated by elements of positive degrees in the image of 
    \[i_1:\Q[V]^W \to (\Q[V]\otimes \Q[V])^W \, , \quad x \mapsto x\otimes 1\,.\]
\end{prop}
Under \eqref{eq:coh-BGcom}, the $\Q[V]^W$-module structure on $H^{\ast}(B_{\com}G;\Q)$ is $f\cdot[x\otimes y]= [x\otimes fy].$ 
In particular we can identify $\Q[V]^W$ as a subalgebra in $H^{\ast}(B_{\com}G;\Q)$. 
As a consequence,
\begin{thm}[\cite{adem2015classifying}*{Theorem 7.2}]
\label{thm:G-com-rank}
    For $G$ a compact connected Lie group, $H^{\ast}(B_{\com}G;\Q)$ is a free module over $\Q[V]^W$ of rank $|W|$. 
\end{thm}
The rational cohomology of the fiber $E_{\com}G$ has a similar presentation.
\begin{cor}[\cite{adem2015classifying}*{Corollary 7.4}]
\label{cor:EcomG}
    Suppose that $G$ is a compact connected Lie group with $T\subset G$ a maximal torus and associated Weyl group $W$, then there is a natural isomorphism of rings
    \[\begin{tikzcd}
        \widetilde{\alpha}_{G}: H^{\ast}(E_{\com}G;\Q) \ar[r,"\cong"] & (H^{\ast}(G/T;\Q)\otimes H^{\ast}(G/T;\Q))^{W}\,.
    \end{tikzcd}\]
\end{cor}

We apply the join construction to \eqref{eq:fib-com} and $G/H \hookrightarrow BH \to BG$ to get a tower of fibrations:
\begin{equation}\label{eq:fib-com-m}
    \begin{tikzcd}[sep=small]
    F_m = E_{\com}G\ast^m G/H \to X_m = B_{\com}G\ast_{BG}^mBH \to BG
    \end{tikzcd}
\end{equation}
for $ m\geq 0 $. 
The space $E_{\com} G$ satisfies the conditions in \Cref{thm:borel-tow}.

\subsubsection{Rational cohomology}

Write 
\[H^{\ast}(B_{\com}G;\Q) = \Q[x_1,\ldots,x_n,y_1,\ldots,y_n]^{W}/(\theta_1(X),\ldots,\theta_n(X))\]
where $\theta_i(X)$ are the images of generators $\theta_i$ of $H^{\ast}(BG;\Q)=\Q[V]^W$ in 
\[i: \Q[V]^W \to (\Q[V]\otimes \Q[V])^W \, , \quad x \mapsto x\otimes 1\, .\]
For simplicity we write $\Q[X,Y]=\Q[x_1,\ldots,x_n,y_1,\ldots,y_n]$. 
By \Cref{prop:presentation}, we get the following result for $X_0=B_{\com}G$. 

\begin{prop}
Let $Q_m(E_{\com}G,G/H)=H^\ast(X_m(E_{\com}G,G/H);\Q)$, then
\[ Q_m(E_{\com}G,G/H) = \{~f\in H^{\ast}(B_{\com}G;\Q) \, |\, s_\alpha(f)\equiv f \, \mod\, (\theta)^m, \forall s_\alpha\in W\, \}\,, \]
which is a free module over $\Q[V]^W$ of rank $|W|$.
\end{prop}


\subsubsection{Examples}

We compute the algebras $Q_m(E_{\com}G,G/H)$ for $G=\U(n),\SU(n),$ and $\SP(n)$, and give explicit descriptions of their bases as free modules over $H^{\ast}(BG)$ by known results for basis of $(\Q[V]\otimes \Q[V])^W$ as $\Q[V]^W$-module (see, e.g., \cites{allen1994descent,adin2001flag,vaccarino2005ring,gomez2013basis,griffeth2023diagonal}). 
\begin{eg}
    For $G=\U(n), H= \U(n-1)$ and 
    $X_m=X_m(E_{\com}\U(n),\U(n)/\U(n-1))$,
    \[
    Q_m(E_{\com}G,G/H) = \left\{\, f \in \left(\Q[X,Y]^{S_n}/(c_1(X),\ldots,c_n(X))\right) \middle| s_{ij}(f) \equiv f \, \mod\, (c_n(Y))^m \,\right\}
    \]
    where $S_n$ acts diagonally by $\sigma\cdot f(X,Y) = f(\sigma\cdot X,\sigma\cdot Y)$, and $c_i(X)=\sigma_i(x_1,\ldots,x_n)$ and $c_i(Y)=\sigma_i(y_1,\ldots,y_n)$. 
    The {\it descent monomial} associated with $ \sigma \in S_n $ is defined to be 
    \[h_{\sigma} = \prod_{\sigma^{-1}(i)>\sigma^{-1}(i+1)}(x_1\ldots x_i) \prod_{\sigma(j)>\sigma(j+1)} (y_{\sigma(1)}\ldots y_{\sigma(j)})\,.
    \]
    Consider the averaging operator 
    \[\rho:\Q[X,Y] \to \Q[X,Y]^{S_n}\,, \quad f \mapsto \frac{1}{|S_n|}\sum_{\sigma\in S_n}\sigma\cdot f\]
    and let $ g_{\sigma} = \rho(h_{\sigma}) \in \Q[X,Y]^{S_n} $. 
    The descent monomials form a basis of $H^{\ast}(B_{\com}G;\Q)$ over $\Q[V]^W$ (see e.g., \cite{allen1994descent}*{Theorem 1.3} and \cite{adem2015classifying}*{Section 8.1}), thus $H_{U(n)}^{\ast}(F_m;\Q)$ decomposes as a graded module 
    \[ H^{\ast}(X_m;\Q) \cong \Q[c_1(Y),\ldots,c_n(Y)] \bigoplus \bigl(\bigoplus_{\sigma\in S_n\backslash\{e\}} g_{\sigma}c_n(Y)^m \cdot \Q[c_1(Y),\ldots,c_n(Y)] \bigr) \]
    over $ H^{\ast}(B\U(n);\Q) $ with basis
    $\{ 1,g_{\sigma}c_n(Y)^m\}_{\sigma\in S_n\backslash\{e\}}.$
    Let 
    \[\text{maj}(\sigma)\coloneqq \sum_{\sigma(i)>\sigma(i+1)}i \]
    be the {\it major index} $\text{maj}(\sigma)$ of $\sigma\in S_n$, then the Hilbert series of $H^{\ast}(X_m;\Q)$ is 
    \[ p_{X_m}(t) = \frac{1-t^{2mn}}{\prod_{i=1}^n(1-t^{2i})} + \frac{t^{2mn}(\sum_{\sigma\in S_n}t^{2(\text{maj}(\sigma)+\text{maj}(\sigma^{-1}))})}{\prod_{i=1}^n(1-t^{2i})}\,. \]
\end{eg}

\begin{eg}
    For $G=\SU(n), H= \SU(n-1)$, $X_m=X_m(E_{\com}\SU(n),\SU(n)/\SU(n-1))$, 
    \[
    Q_m(E_{\com}G,G/H) = \left\{\, f \in \left(\Q[X,Y]^{S_n}/(c_1(X),\ldots,c_n(X),c_1(Y))\right) \middle| s_{ij}(f) \equiv f \, \mod\, (c_n(Y))^m \,\right\}
    \]
    by \cite{adem2015classifying}*{Section 8.2} and it decomposes as a graded module 
    \[ H^{\ast}(X_m;\Q) \cong \Q[c_2(Y),\ldots,c_n(Y)] \bigoplus \bigl(\bigoplus_{\sigma\in S_n\backslash\{e\}} g_{\sigma}c_n(Y)^m \cdot \Q[c_2(Y),\ldots,c_n(Y)]\bigr) \]
    over $ H^{\ast}(B\SU(n);\Q) $ with basis
    $\{\, 1,g_{\sigma}c_n(Y)^m\,\}_{\sigma\in S_n\backslash\{e\}\,}$.
    The Hilbert series of $H^{\ast}(X_m;\Q)$ is 
    \[ p_{X_m}(t) = \frac{1-t^{2mn}}{\prod_{i=2}^n(1-t^{2i})} + \frac{t^{2mn}(\sum_{\sigma\in S_n}t^{2(\text{maj}(\sigma)+\text{maj}(\sigma^{-1}))})}{\prod_{i=2}^n(1-t^{2i})}\,. \]
\end{eg}

\begin{eg}
    For $G=\SP(n), H= \SP(n-1)$ and 
    $X_m=X_m(E_{\com}\SP(n),\SP(n)/\SP(n-1)),$
    \[
    Q_m(E_{\com}G,G/H) = \left\{\, f \in \left(\Q[X,Y]^{B_n}/(q_1(X),\ldots,q_n(X))\right) \,\middle|\, s_\alpha(f) \equiv f \, \mod\, (c_n(Y))^m \,\right\}
    \]
    where $q_i(X)=\sigma_i(x_1^2,\ldots,x_n^2)$ and $q_i(Y)=\sigma_i(y_1^2,\ldots,y_n^2)$. 
    We view $B_n$ as the group of signed permutation on $\bI_n=\{-n,-n+1,\ldots,-1,1,\ldots,n\}$. 
    Given $ \sigma\in B_n $, define 
    \[\begin{array}{l}
    d_i(\sigma) \coloneqq \vert \{i\leq j\leq n-1 | \sigma(j)>\sigma(j+1)\} \vert \\
    \varepsilon_i(\sigma) \coloneqq \begin{cases} 
    0, & \sigma(i)>0\,, \\
    1, & \sigma(i)<0\,.
    \end{cases} \\
    f_i(\sigma) \coloneqq 2d_i(\sigma) + \varepsilon_i(\sigma)\,. 
    \end{array}\] 
    The {\it diagonal signed descent monomial} associated to $ \sigma \in B_n $ is 
    \[ b_{\sigma} = \prod_{i=1}^nx_i^{f_i(\sigma^{-1})} y_i^{f_{|\sigma^{-1}(i)|}(\sigma)}\,. \]
    Consider the averaging operator 
    \[\rho: \Q[X,Y] \to \Q[X,Y]^{B_n}\,, \quad f \mapsto \frac{1}{|B_n|}\sum_{\sigma\in B_n}\sigma\cdot f\]
    and let $ c_{\sigma} = \rho(b_{\sigma}) \in \Q[X,Y]^{B_n} $. The diagonal signed descent monomials form a basis of $H^{\ast}(B_{\com}G;\Q)$ over $\Q[V]^W$ (see e.g., \cite{gomez2013basis}*{Theorem 1} and \cite{adem2015classifying}*{Section 8.3}), thus $H^{\ast}(X_m;\Q)$ decomposes as a graded module 
    \[ H^{\ast}(X_m;\Q) \cong \Q[q_1(Y),\ldots,q_n(Y)] \bigoplus \bigl(\bigoplus_{\sigma\in B_n\backslash\{e\}} c_{\sigma}q_n(Y)^m \cdot \Q[q_1(Y),\ldots,q_n(Y)]\bigr) \]
    over $ H^{\ast}(B\SP(n);\Q) $ with basis 
    $\{ 1,c_{\sigma}q_n(Y)^m\}_{\sigma\in B_n\backslash\{e\}}.$
    
    Let $\text{fmaj}(\sigma)\coloneqq 2(\text{maj}(\sigma)+\text{neg}(\sigma))$ be the {\it flag major index} of a signed permutation $\sigma$, 
    where $\text{neg}(\sigma)\coloneqq |\{1\leq i\leq n \, |\, \sigma(i)<0 \}|$.
    Then the Hilbert series of $H^{\ast}(X_m)$ is given by
    \[ p_{X_m}(t) = \frac{1-t^{4mn}}{\prod_{i=1}^n(1-t^{4i})} + \frac{t^{4mn}(\sum_{\sigma\in B_n}t^{2(\text{fmaj}(\sigma)+\text{fmaj}(\sigma^{-1}))})}{\prod_{i=1}^n(1-t^{4i})}\,. \]
\end{eg}

\section{Generalizations}
\label{sec:gen}
The join (relative Ganea) construction over $BG$ has a natural generalization as follows.
\subsection{Generalized join construction}
Let $ \widetilde{G} $ be a closed subgroup of $G$ containing $T$, with Weyl group $\widetilde{W} = N_{\widetilde{G}}(T)/T$. 
Take $H\subseteq G$ a closed subgroup such that $\widetilde{G}/H \cong \bS^{2k-1}$ for some $ k \ge 1$. 
Take $F$ a $\Q$-finite $G$-space and $X=EG\times_G F$. 
If the fibration $X\to BG$ factors through $X\to B\widetilde{G}$, we write $K\coloneqq \text{fib}_\ast(X\to B\widetilde{G})$.
We have
\[
\begin{tikzcd}
    K \ar[r]\ar[d,equal] & F \ar[r]\ar[d] & G/\widetilde{G}\ar[d] \\
    K \ar[r]\ar[d] & X\ar[d] \ar[r] & B\widetilde{G}\ar[d] \\
    \pt \ar[r] & BG \ar[r] & BG
\end{tikzcd}
\]
where all columns are fibrations. As the second and third rows are fibrations, the top row is also a fibration. 
Furthermore, consider the following commutative diagram
\[
\begin{tikzcd}
    F \ar[r]\ar[d] & G/\widetilde{G}\ar[r]\ar[d] & \pt\ar[d] \\
    X \ar[r] & B\widetilde{G}\ar[r] & BG
\end{tikzcd}
\]
as the right and outer square are (homotopy) pullbacks, so is the left square.

Applying the relative Ganea construction iteratively to the top, bottom, and outer squares in the following pullback diagrams
\[
\begin{tikzcd}
    K \ar[r]\ar[d] & \pt \ar[d] & \widetilde{G}/H \ar[l]\ar[d] \\
    F \ar[r]\ar[d] & G/\widetilde{G}\ar[d] & G/H \ar[l]\ar[d] \\
    X \ar[r] & B\widetilde{G} & BH \ar[l]
\end{tikzcd}
\]
we get fibrations
\begin{eqnarray}
    \label{eq:fib-KF}
    K_m \to F_m \to G/\widetilde{G} \\
    \label{eq:fib-K}
    K_m \to \widetilde{X}_m \to B\widetilde{G}
\end{eqnarray}
where $K_m=K\ast^m \widetilde{G}/H$, $F_m=F\ast_{G/\widetilde{G}}^m G/H$ and $\widetilde{X}_m=X\ast_{B\widetilde{G}}(BH)^{\ast m}$. 
Note we have another the fibration 
\begin{equation}
    \label{eq:fib-F}
    F_m \to \widetilde{X}_m \to BG 
\end{equation}
which can be read from the following composite of homotopy pullback diagrams
\begin{equation*}
\begin{tikzcd}
    F_m \ar[r]\ar[d] & G/\Tilde{G} \ar[r]\ar[d] & \pt \ar[d] \\
    X_m \ar[r] & B\Tilde{G} \ar[r] & BG \,
\end{tikzcd}
\end{equation*}
where the left hand side homotopy pullback diagram follows from \cite{Doe98}*{Theorem 3.1}.

In particular, $\widetilde{X}_m$ fits into a diagram of spaces
\begin{equation}
\label{eq:Xrel}
X=\widetilde{X}_0 \to \widetilde{X}_1 \to \ldots \to \widetilde{X}_m \to \ldots \
\end{equation}

The homogeneous space $G/\widetilde{G}$ has only finite even degree rational cohomology (see e.g., \cite{greub1976connections}*{Chapter XI, Section 2, Theorem VII}).  
The Serre spectral sequence associated to the fibration $\widetilde{G}/T\hookrightarrow G/T \twoheadrightarrow G/\widetilde{G}$ degenerates and thus $\dim_{\Q}H^{\ast}(G/\widetilde{G};\Q)= |W|/|\widetilde{W}|$. 
We have another fibration 
$G/\widetilde{G} \to B\widetilde{G} \to BG$
whose Leray-Serre spectral sequence also collapses, thus $ H^{\ast}(B\widetilde{G};\Q) $ is a finite generated free module over $ H^\ast(BG;\Q) $ of rank $ | W|/|\widetilde{W}| $. 
If $K$ is a $\Q$-finite space whose rational cohomology is concentrated in even degrees, then by \eqref{eq:fib-KF} that $F$ is a $\Q$-finite space whose rational cohomology is concentrated in even degrees,
and $\dim_\Q H^\ast(F;\Q)=\dim_\Q H^\ast(K;\Q) \cdot \frac{|W|}{|\widetilde{W}|}$. 

Using \eqref{eq:fib-F} and \eqref{eq:fib-K} we get a generalization of \Cref{thm:borel-tow}. 
\begin{thm}
\label{thm:gen-borel}
For all $ m \ge 0 $, the algebras $\widetilde{Q}_m(F,\widetilde{G},H)=H^\ast(\widetilde{X}_m;\Q)$ are free modules over $H^\ast(B\widetilde{G};\Q)$ of rank $\dim_{\Q}[H^{\ast}(K,\,\Q)]$ and thus also free module over $H^\ast(BG;\Q)$ of rank $\dim_{\Q}[H^{\ast}(F,\,\Q)]$ and hence $($graded$)$ Cohen-Macaulay. The maps in the diagram \eqref{eq:Xrel} induces an {\rm injective} algebra homomorphism on rational cohomology, defining the descending filtration
\begin{equation*}
\widetilde{Q}_0(F,\widetilde{G},H) \hookleftarrow  \widetilde{Q}_1(F,\widetilde{G},H)  \hookleftarrow \ldots  
\hookleftarrow \widetilde{Q}_m(F,\widetilde{G},H) \xhookleftarrow{\,\pi_m^*\,} \widetilde{Q}_{m+1}(F,\widetilde{G},H) \hookleftarrow \ldots
\end{equation*}
such that  $ \,\varprojlim \,\widetilde{Q}_m(F,\widetilde{G},H) \cong \Q[V]^{\widetilde{W}} $. Furthermore, $K_m,F_m$ and $\widetilde{X}_m$ are formal spaces for $m>0$.
\end{thm}
We also get a presentation of $\widetilde{Q}_m(F,\widetilde{G},H)$ as \eqref{prop:QmGTH} if we replace $G$ by $\widetilde{G}$.
\begin{prop}
   Let $\Tilde{\theta}\in H^{\ast}(B\Tilde{G};\Q)$ be the Euler class of the spherical fiber bundle $\Tilde{G}/H\hookrightarrow BH \to B\Tilde{G}$, then 
    \begin{equation}\label{eq:HQ-Xm-rel}
        H^{\ast}(\widetilde{X}_m;\Q) = \{\, f \in H^{\ast}(X;\Q)\, |\, s_\alpha(f)\equiv f \, \mod\, (\Tilde{\theta})^m , \, \forall s_\alpha\in \widetilde{W}\, \}
    \end{equation}
    where $\widetilde{W}=W_{\widetilde{G}}=N_{\widetilde{G}}(T)/T$ is the Weyl group of $\widetilde{G}$.
\end{prop}

\subsubsection{Examples}
\begin{eg}\label{eg:rel-U}
Let $ G= \U(n) $ and $ \widetilde{G} = G_{k} = \U(n-k) \times \U(k) $. Let $ H = H_{k} = \U(n-k) \times \U(k-1) $. The Weyl group of $G_{k}$ is $S_{n-k}\times S_{k}$, and
\[ H^{\ast}(BG_{k};\Q)\cong \Q[V]^{W_{k}} = \Q[t_1\ldots,t_n]^{S_{n-k}\times S_{k}} = \Q[b_1,\ldots,b_{n-k},c_1,\ldots,c_{k}]\]
where $b_i = \sigma_i(t_1,\ldots,t_{n-k}) $ and $ c_i = \sigma_i(t_{n-k+1},\ldots,t_n)$ are the symmetric polynomials. Let $d_i$ be the degrees of the homogeneous algebraically independent generators of $H^{\ast}(BG_{k})$ in the order we present here, i.e. 
\[ d_i= \begin{cases}
    \deg(b_i)= i \,, & 1\leq i \leq n-k \,, \\
    \deg(c_{i-n+k})=i-n+k \,, & n-k+1\leq i \leq n \,.
\end{cases}\]
In this case, the fibration \eqref{eq:fib-K} is given by
\begin{equation}\label{eq:fib-mk}
    \begin{tikzcd}[sep=small]
    F_{m,k} = (G/T)\ast_{G/G_{k}} (G/H_{k})^{\ast m} \ar[r,hook] & X_{m,k} = (BT)\ast_{BG_{k}}(BH_{k})^{\ast m} \ar[r] & BG \, .
    \end{tikzcd}
\end{equation}
The $G$-rational cohomology of $F_{m,k}$ is given by
    \[H^{\ast}_G(F_{m,k};\Q) = H^{\ast}(X_{m,k};\Q) = \{\, f\in \Q[t_1,\ldots,t_n] \, |\, s_\alpha(f)\equiv f \, \mod\, (c_k)^m, \forall s_\alpha \in S_{n-k}\times S_k \, \}\]
and the Hilbert series of $X_{m,k}$ is
\[ p_{X_{m,k}}(t) = \frac{1-t^{2mk}}{\prod_{i=1}^{k}(1-t^{2i})\prod_{i=1}^{n-k}(1-t^{2i})} + \frac{t^{2mk}}{(1-t^2)^n}\,.\]
\end{eg}
\begin{eg}\label{eg:rel-Sp}
Let $G=\SP(n), G_{k}=\SP(n-k)\times \SP(k)$ and $H_{k}=\SP(n-k)\times \SP(k-1)$. Then $G_{k}/H_{k}\cong \bS^{4k-1}$. 
In this case, $ \Q[V] = \Q[t_1,\ldots,t_n] $, $ H^{\ast}(BG_{k};\Q) = \Q[p_1,\ldots,p_{n-k},q_1,\ldots,q_{k}] $ 
where $ p_i=\sigma_i(t_1^2,\ldots,t_{n-k}^2), q_i = \sigma_i(t_{n-k+1}^2,\ldots,t_n^2) $. 
%
The $G$-rational cohomology of $F_{m,k}$ is 
\[H^{\ast}_G(F_{m,k};\Q) = H^{\ast}(X_{m,k};\Q) = \{\, f\in \Q[t_1,\ldots,t_n] \, |\, s_\alpha(f)\equiv f \, \mod\, (q_k)^m, \forall s_\alpha \in B_{n-k}\times B_k \, \}\]
and the Hilbert series of $X_{m,k}$ is
\[ p_{X_{m,k}}(t) = \frac{1-t^{4mk}}{\prod_{i=1}^{k}(1-t^{4i})\prod_{i=1}^{n-k}(1-t^{4i})} + \frac{t^{4mk}}{(1-t^2)^n}\,.\]
\end{eg}

\begin{eg}
\label{eg:quasi-partial}
    Let $G$ be a compact connected Lie group with a maximal torus $T$ and Weyl group $W$ with associated root system $\cR$. Take a positive root $\alpha\in\cR_+$ and $T_\alpha\subseteq T$ a singular subtorus corresponding to $\alpha$, $G_\alpha=C_G(T_\alpha)$ the centralizer of $T_\alpha$ in $G$, then for $\Tilde{G}=G_\alpha$ and $H=T_\alpha$, we get a $G$-space 
    \begin{equation}
        F_{m_\alpha}=F_{m_\alpha}(G,T)=(G/T) \ast_{(G/G_\alpha)} (G/T_\alpha)^{\ast m_\alpha} \cong G\times_{G_\alpha} \left[(G_\alpha/T)\ast (G_\alpha/T_\alpha)^{\ast m_\alpha}\right]
    \end{equation} 
    which is the $G$-space assigned to $W_\alpha$ for the functor $(G\times_T\cF_{2m})^{\natural,+}:\cS(W)\to\Top^G$ defined in \cite{berest2025quasiflag}*{(6.38)} for even multiplicity, and its $G$-equivariant rational cohomology is 
    \[
    H^\ast_G(F_{m_\alpha};\Q) = \Q[V]^{W_\alpha} + \alpha^{2m_\alpha}\cdot \Q[V] = \{\, p \in \Q[V]\, |\, s_{\alpha}(p) \equiv p \, \mod\, \langle \alpha \rangle^{2 k_{\alpha}}\, \}
    \]
    where $W_\alpha=N_{G_\alpha}(T)/T$ is the Weyl group associated to $(G_\alpha,T)$.
\end{eg}


\subsection{Filtered join construction}\label{sec:fil}

We can iterate the above generalization to obtain a filtered join construction if we have
\begin{enumerate}
    \item a sequence of Lie subgroups $T \hookrightarrow G_1 \hookrightarrow \ldots \hookrightarrow G_{l} = G$, such that $G_i\subsetneq G_{i+1}$, 
    \item for each $i>0$, a Lie subgroup $H_i$ such that $G_i/H_i\cong S^{2k_i-1}$ is an odd dimensional sphere,
    \item for each $1\leq i \leq l$, a multiplicity $m_i\in \bN$,
\end{enumerate}
we can construct a Borel fibration $F_{(m)}^l\hookrightarrow X_{(m)}^l \twoheadrightarrow BG$ associated to each multiplicity $m=(m_1,\ldots,m_l)$
as follows.
\begin{enumerate}
    \item The fibration \eqref{eq:fib-F0} and $G/H_1\hookrightarrow BH_1\twoheadrightarrow BG$ factors through $BG_1$ and we can construct the iterated relative join fibration
    \begin{equation*}
       \begin{tikzcd}[sep=small]
           F_{(m)}^1 = F\ast_{G/G_1} (G/H_1)^{\ast m_1} \ar[r,hook] & X_{(m)}^1=X\ast_{BG_1} (BH_1)^{\ast m_1} \ar[r,"p_1"] & BG \, . 
       \end{tikzcd}
    \end{equation*}
    Note $p_1:X_1\to BG$ factors through $BG_2$. 
    \item Assume in step $i$, we have a Borel fibration 
    \begin{equation*}
    \begin{tikzcd}[sep=small]
        F_{(m)}^i = F_{(m)}^{i-1}{\ast}_{G/G_i}^{m_i} G/H_i \ar[r,hook] & X_{(m)}^i=X_{(m)}^{i-1}\ast_{BG_i}^{m_i} BH_i \ar[r,"p_i^m"] & BG 
    \end{tikzcd}
    \end{equation*}
    factoring through $BG_i$ and thus also factoring through $BG_{i+1}$, then we can construct the space 
    $X_{(m)}^{i+1}\coloneqq X_{(m)}^i\ast_{BG_{i+1}}^{m_{i+1}}BH_{i+1}$ and let $F_{(m)}^{i+1}=\hfib\{X_{(m)}^i\xrightarrow{p_i} BG\}$ to obtain the following Borel fibration
    \begin{equation*}
        \begin{tikzcd}[sep=small]
            F_{(m)}^{i+1} = F_{(m)}^i\ast_{G/G_{i+1}}^{m_{i+1}} G/H_{i+1} \ar[r,hook] & X_{(m)}^{i+1}=X_{(m)}^i\ast_{BG_{i+1}}^{m_{i+1}} BH_{i+1} \ar[r,"p_{i+1}^m"] & BG
        \end{tikzcd}
    \end{equation*}
    and the fibration $p_{i+1}^m$ factors through $BG_{i+1}$, thus also through $BG_{i+2}$.
    \item By induction, we obtain a fibration at the last step $l$, 
    \begin{equation}\label{eq:fil-gen-borel-fib-m}
       \begin{tikzcd}[sep=small]
           F_{(m)}^l = G/T\ast_{G/G_1}^{m_1}\ldots\ast_{G/G_l}^{m_l} G/H_l \ar[r,hook] & X_{(m)}^l = BT\ast_{BG_1}^{m_1}\ldots\ast_{BG_l}^{m_l}BH_l \ar[r,"p_l^m"] & BG \, .
       \end{tikzcd} 
    \end{equation}
\end{enumerate}

\begin{thm}
    Let $\theta_i$ be the Euler class of the spherical fiber bundle $G_i/H_i\hookrightarrow BH_i \to BG_i$, and identify $\theta_i\in H^{\ast}(BG_i;\Q)$ with its image in $H_G^{\ast}(F;\Q)$. 
    The rational cohomology of $X_{(m)}^l$ can be identified as a subalgebra of $H_G^{\ast}(F;\Q)$ and is given by
    \[H^{\ast}(X_{(m)}^l) = \{~f\in H_G^{\ast}(F;\Q) \, |\, s_i(f) \equiv f \, \mod\, (\theta_i)^{m_i}, \forall s_i\in W_i \, \} \, \]
    where $s_i$ are the reflections in $W_i = N_{G_i}(T)/T$.
\end{thm}
\begin{proof}
    This can be proved by induction on $l$.
\end{proof}

\subsubsection{Examples}

We consider the examples when $G= \U_{\bF}(n)$ where $\bF=\Q$ or $\mathbb{H}$, i.e. $G=\U(n)$ or $\SP(n)$. We choose $F=G/T$.

Let $\lambda= (\lambda_1,\ldots,\lambda_{k}) $ be a partition of $n$, i.e. $\lambda_i>0$ and $\sum_{i=1}^{k}\lambda_i = n$. Let $\gamma=(\gamma_1,\ldots,\gamma_{l})$ be another partition of $n$, we say $\lambda \leq \gamma$ if $\lambda$ is a refinement of $\gamma$, i.e. $k>l$ and there exists a surjective order preserving map $\varphi:[k]\twoheadrightarrow [l]$ where $[k]=\{1\to2\to\ldots\to k\}$ is the poset of $k$ elements, such that $\sum_{j\in\varphi^{-1}(i)}\lambda_{j}=\gamma_i$. Let $\Lambda$ be the poset of all partitions of $n$.

For each partition $\lambda$ of $n$, we associate to it a subgroup $G_{\lambda}=U_{\bF}(\lambda_1)\times U_{\bF}(\lambda_2) \times \ldots \times U_{\bF}(\lambda_k)$ of $G= U_{\bF}(n)$. For instance, $(1,\ldots,1)$ is the smallest partition of $n$ and the associated subgroup is $G_{(1,\ldots,1)}=T$ a maximal torus of $G$; $(n)$ is the biggest partition of $n$ and the associated subgroup is $G_{(n)}=G$. Whenever $\lambda\leq \gamma$, we have a canonical inclusion $G_{\lambda}\hookrightarrow G_{\gamma}$, which induces a map $BG_{\lambda}\to BG_{\gamma}$.

We choose a sequence of partitions $\lambda_0=(1,\dots,1)<\lambda_1<\ldots<\lambda_l$ for $1\leq l\leq n-1$, and associated to $G_i=G_{\lambda_i}$, we choose a subgroup $H_i=H_{\lambda_i'}$ of corank $1$ associated to a partition $\lambda_i'$ given by $H_{\lambda_i'}=U_{\bF}(\lambda_{i,1})\times\ldots\times U_{\bF}(\lambda_{i,j_i}-1) \times \ldots \times U_{\bF}(\lambda_{i,k_i})$ where each factor is the same as the factor in $G_{\lambda_1}$, except for the $j_i$-th component we consider the subgroup $U_{\bF}(\lambda_i-1)$ in $U_{\bF}(\lambda_i)$. 

In this case each $G_i/H_i\cong \bS^{d(\lambda_{i,j_i})-1}$ where $d=\dim_{\bR}\bF$. 
The sequence of subgroups $(G_i,H_i)$ satisfies the requirement of the filtered join construction so for any multiplicity $m=(m_1,\ldots,m_l)$, we can construction a Borel fibration as in \eqref{eq:fil-gen-borel-fib-m} by taking $X=BG_0=BT$. 

\begin{eg}
    Let $\lambda_0=(1,\ldots,1)<\lambda_1=(n)$ and $\lambda_1'=(n-1)$, we recover \Cref{eg:U}.
\end{eg}
\begin{eg}
    For $1\leq k\leq n$, let $\lambda_0=(1,\ldots,1)<\lambda_1=(n-k,k)$ and $\lambda_1'=(n-k,k-1)$, this recovers \Cref{eg:rel-U} and \Cref{eg:rel-Sp}.
\end{eg}
\begin{eg}
    Let $G=\U(3)$ and $\lambda_0=(1,1,1)<\lambda_1=(2,1)<\lambda_2=(3)$ be a maximal sequence of partitions. 
    In this case $G_1=\U(2)\times\U(1), G_2=\U(3)$ and we can choose $\lambda_1'=(1,1),\lambda_2 = (2)$ such that $H_1=\U(1)\times\{1\}\times\U(1), H_2=\U(2)\times\{1\}$. For any multiplicity $m=(m_1,m_2)$, the rational cohomology of $X_{(m)}^2$ is 
    \[H^\ast(X_{(m)}^l;\Q) = \{\, f\in \Q[t_1,t_2,t_3] \,|\, s_i(f) \equiv f \,\mod\, (\theta_i)^{m_i}, \forall s_i \in W_i \,\} \]
    where $\theta_1=t_1t_2$, $W_1=\langle s_{12}\rangle$ and $\theta_2=t_1t_2t_3$, $W_2=S_3$.
    A basis of $H^{\ast}(X_{(m)}^2;\Q)$ is given by \\
    $\{ 1, t_3(t_1t_2t_3)^{m_2},t_3^2(t_1t_2t_3)^{m_2},t_2(t_1t_2t_3)^{m_2}(t_1t_2)^{m_1}, t_2t_3(t_1t_2t_3)^{m_2}(t_1t_2)^{m_1}$, $t_2t_3^2(t_1t_2t_3)^{m_2}(t_1t_2)^{m_1} \}.$
\end{eg}

\appendix
\section{Abstract Ganea construction}
\label{sec:ganea}
In this Appendix, we re-examine the classical fibre-cofibre construction and its relative version in the context of $\infty$-categories. 

\subsection{\texorpdfstring{$\infty$}--categories}
We begin by recalling a few basic definitions related to $\infty$-categories. Our main references are Lurie's seminal monograph \cite{lurie2009higher} and his online notes \cite{HA} and \cite{kerodon}, which notation and terminology we will mostly follow. As in {\it loc. cit}, 
by an $\infty$-category we mean an $(\infty,1)$-category, or more specifically, 
a quasi-category (in the sense of Joyal \cite{J02}). 
Recall that a {\it quasi-category} is a (not necessarily small) simplicial 
set $ \Cc \in \sset $ that satisfies the lifting property of a weak Kan complex: every map of simplicial sets $ f: \Lambda^n_i \to \mathscr{C} $ from the $i$-th horn $\Lambda^n_i \subset \Delta^n $ of an $n$-simplex can be extended to the full $n$-simplex $ \widetilde{f}: \Delta^n \to \mathscr{C} $ provided that $ 0 < i < n $. 
The $\infty$-categories can be viewed as a common generalization of ordinary (1-)categories and topological spaces.
\begin{eg}
\la{Cat}
If $\mathscr{C}$ is an ordinary category, its simplicial nerve $ N \Cc  $ is a weak Kan complex with the additional property that the above extensions $ \widetilde{f} $ are unique for all maps $f$ (see, e.g., \cite{lurie2009higher}*{1.1.2.2}). It is well known that the nerve $ N \Cc $ determines the category $ \Cc $ uniquely (up to isomorphism): in this way, we may identify `$\Cc = N \Cc$' and think of $ \Cc $ as an $\infty$-category. 
\end{eg}
\begin{eg}
\la{Top}
If $S$ is a topological space, its singular complex $ {\rm Sing}(S) = \{\Hom(\dDel^n, S)\}_{n \ge 0}$ is a Kan complex, i.e. a weak Kan complex with the additional property that, besides
the maps from inner horns, the extensions $ \widetilde{f} $ also exist for all maps $f$ from the external horns $ \Lambda^n_0 \subset \Delta^n $ and $\Lambda^n_n \subset \Delta^n $. It is well known that the simplicial set  ${\rm Sing}(S)$ determines the space $S$ up to (weak) homotopy equivalence: in this way, we may identify `$ {\rm Sing}(S) = S$' and think of a space $S$ 
as an $\infty$-category (which is, in fact, an $\infty$-groupoid).
\end{eg}
Basic categorical and homotopy-theoretic notions extend to $\infty$-categories in a natural way. 
Given an $\infty$-category $\Cc $, we refer to the set of its $0$-simplices $\Cc_0 = \Hom(\Delta^0, \Cc)$ as the {\it objects} of $\Cc$ and the set of its $1$-simplices $ \Cc_1 = \Hom(\Delta^1, \Cc) $ as the {\it morphisms} of $\cC$. For two objects $ X,Y \in \Cc_0 $,
the {\it mapping space} $\,{\rm Map}_{\Cc}(X,Y)\,$ in $\Cc$ is then defined by the pullback 
%
\begin{equation}\la{Mapdef}
\begin{diagram}[small]
{\rm Map}_{\Cc}(\SEpbk X,Y) &\rTo & {\rm Fun}(\Delta^1,\Cc)\\
\dTo & & \dTo_{(\delta^{\ast}_1,\,\delta^{\ast}_0)}\\
\Delta^0 & \rTo^{(X,Y)} & \Cc \times \Cc\\
\end{diagram}
\end{equation}
where 
\begin{equation}
\la{Funcat}
\Fun(\Dc,\Cc) := \Map_{\sset}(\Dc, \Cc) = \{\Hom(\Dc \times \Delta^n,\,\Cc)\}_{n \ge 0}
\end{equation}
stands for the standard mapping complex (internal `Hom') in the category of simplicial sets, and
the right vertical map is the restriction along the canonical (coface) maps\footnote{Recall that, in general, the $i$-th coface map $ \delta_i: \Delta^{n-1} \to \Delta^n $ on standard simplices corresponds to the 
(unique) order-preserving inclusion $ [n-1] \into [n] $ that does {\it not} include the element $i$ in its image.} $\,\Delta^0 \xhookrightarrow{\,\delta_1\,}\Delta^1 \xhookleftarrow{\,\delta_0\,} \Delta^0$. The set of morphisms from $X$ to $Y$ in $ \Cc $ is identified with the set of $0$-simplices in this mapping space: i.e., $\, \Hom_{\Cc}(X,Y) = \Map_{\Cc}(X,Y)_0 $, while the set of homotopy classes of morphisms in $\Cc$ is given by $\,[X,Y]_{\Cc} := \pi_0[\Map_{\Cc}(X,Y)]\,$. The latter form an ordinary category denoted by $ h\Cc $ (or $\pi \Cc$) and called the {\it homotopy category} of $ \Cc $. A morphism $\, f \in \Hom_{\Cc}(X,Y) $ is called a {\it homotopy equivalence} (or simply, an {\it equivalence}) in $ \Cc$ if it induces an isomorphism in $ h\Cc$. 

An alternative (equivalent) way to define the homotopy category $h\Cc $ is to apply to (the underlying simplicial set of)  $\Cc$ the functor $ h: \sset \to {\rm Cat} $, the left adjoint of the classical nerve functor $N: {\rm Cat} \to \sset $ (see \cite{lurie2009higher}*{1.2.3}). The unit of the adjunction $(h,N)$ yields then a canonical map in $\sset$ 
(a functor between $\infty$-categories) $\,\Cc \to N[h\Cc]\,$  that acts as the identity on vertices (objects) of $\Cc_0$. Given a(ny) collection of objects in $ \Cc$, say $ S \subseteq \Cc_0 $, we then
define the {\it full $($$\infty$-$)$subcategory $ \Cc_S \subseteq \Cc $ spanned by $S$} to be the pullback in $\sset$:
\begin{equation*}\la{subcat}
\begin{diagram}[small]
{\Cc}_S \SEpbk  &\rTo & \Cc\\
\dTo & & \dTo\\
N(h\Cc_S) & \rTo & N(h\Cc)\\
\end{diagram}
\end{equation*}
where $ h\Cc_S \subseteq h\Cc $ denote the (ordinary) full subcategory in $h\Cc $ spanned by
$S$.

It is a fundamental fact of the theory that, for any $\infty$-category $ \Cc $,  the mapping space $ \Map_{\Cc}(X,Y) $ defined by \eqref{Mapdef} is a Kan complex (and hence,  a topological space) for all $ X,Y \in \Cc_0 $. Another fundamental fact is that, for any $\infty$-category $\Cc$, the simplicial set  \eqref{Funcat} is itself an $\infty$-category called the {\it functor $\infty$-category} (see \cite{lurie2009higher}*{1.2.7.3}). The objects of $ \Fun(\Dc, \Cc) $
are simply the maps of simplicial sets $ \cF: \Dc \to \Cc $, which we refer to as {\it functors
from $\Dc$ to $\Cc$}; the morphisms  $ \Fun(\Dc, \Cc)_1 $ are given by simplicial maps 
$ \eta: \Dc \times \Delta^1 \to \Cc $ called the {\it natural transformations}. If $ \Dc $
and $\Cc $ are both $\infty$-categories, a functor $ \cF: \Dc \to \Cc $ is called an 
{\it equivalence of $\infty$-categories} if it is an equivalence in $ \Fun(\Dc,\Cc)$. 

\subsection{Limits and colimits}
Let $\Cc$ be an $\infty$-category, and $ I$ a fixed (small) simplicial set. We refer to the functors $ I \to \Cc $ (i.e., the objects of the $\infty$-category $ \Fun(I, \Cc)$) as the {\it $ I$-diagrams} in $\Cc$. Every object $ A \in \Cc_0 $ yields a {\it constant} $I$-diagram (at $A$) by composition 
$$
\ua:\, I \to \Delta^0 \xrightarrow{A} \Cc  
$$
where $ I \to \Delta^0 $ is the canonical (terminal) map in $\sset$. Then, given any $I$-diagram $\cF: I \to \Cc $, a {\it left cone on $\cF$} with vertex $ A \in \Cc_0 $ is defined to be a natural transformation $ \alpha:\, \ua \to \cF$  in $ \Fun(I, \Cc) $, i.e. a functor $\,\alpha: I\times\Delta^1 \to \Cc \,$ such that $ \alpha|_{I \times\{0\}} = \ua $ and $ \alpha|_{I \times \{1\}} = \cF $. Similarly, we define a {\it right cone on $\cF$} to be a natural transformation $ \beta: \cF \to \ua $, i.e. $\,\beta: I\times\Delta^1 \to \Cc \,$  in $ \Fun(I, \Cc) $ such that $ \beta|_{I \times\{0\}} = \cF $ and $ \beta|_{I \times \{1\}} = \ua $.
There is an alternative way to define left and right cones.
Using the join operation `$\star $' on $\sset$,  we first construct two simplicial sets $ I^{\triangleleft} := \Delta^0 \star I $ and  $ I^{\triangleright} := I \star \Delta^0 $, called the left and right cones on $I$, each being equipped
with a pair of canonical maps $\, \Delta^0 \xhookrightarrow{i_0} I^{\star} \xhookleftarrow{i}\,I\,$, called the vertex and base inclusions (see \cite{lurie2009higher}*{1.2.8.4}). A left (resp., right) cone on an $I$-diagram 
$ \cF $ with vertex $ A \in \Cc_0 $ can then be described as a functor 
$ \widetilde{F}: \, I^{\triangleleft} \to \Cc $ (resp., $ \widetilde{\cF}:  \,I^{\triangleright} \to \Cc $)  such that $\,  i^* \widetilde{F} = F \,$ and $\,  i_0^* \widetilde{F} = A\,$, where $ i^* $ and 
$ i_0^*$ are the restrictions along the base and the vertex inclusions. 

Now, extending the standard definitions of limits and colimits in category 
theory, we can introduce their $\infty$-categorical analogs as universal (respectively, left and right) cones, where the term `universal' is understood in the following homotopical sense 
({\it cf.} \cite{kerodon}*{7.1.1.1}]).
\begin{defn}
A left cone $ \alpha:\, \ua \to \cF $ represents the {\it limit of $\cF$} if the  composite map 
\begin{equation*}
\Map_{\Cc}(X,A)\,\to\,  \Map_{\Fun(I, \Cc)}(\ux, \ua)\, \xrightarrow{\alpha_*}\, \Map_{\Fun(I, \Cc)}(\ux, \cF)  
\end{equation*}
is a homotopy equivalence of spaces (Kan complexes) for every object $X \in \Cc_0 $. 

Dually, a right cone $ \beta:\, \cF \to \ua $ represents the {\it colimit of $\cF$}  if the composite map 
\begin{equation*}
\Map_{\Cc}(A,X)\, \to \,  \Map_{\Fun(I, \Cc)}(\ua, \ux)\, 
\xrightarrow{\beta^*}\, \Map_{\Fun(I, \Cc)}(\cF, \ux)  
\end{equation*}
is a homotopy equivalence of spaces (Kan complexes)  for every object $X \in \Cc_0 $. 
\end{defn} 
We write $\, \blim_I(\cF) = A \,$ if there is a left cone $ \alpha $ with vertex $A$ that represents the limit of an $I$-diagram $F$, and dually, $\, \bcolim_I(\cF) = A \,$ if there is a right cone $ \beta $ with vertex $A$ that represents the colimit of an $I$-diagram $F$.

We consider two basic examples that play a special role in Ganea construction.

\begin{eg}[Pushout squares]
\la{PushSq}
Let $ I = \Delta^1 \vee \Delta^1 : = \Delta^1 \sqcup_{\Delta^0} \Delta^1 $, where 
the two copies of $ \Delta^1 $ are amalgamated over $\Delta^0$ 
via the inclusion $ \delta_1: \Delta^0 \into \Delta^1 $. This simplicial set is an $\infty$-category isomorphic to the nerve of the poset $ [1] \sqcup_{[0]} [1] = \{1 \leftarrow 0 \rightarrow 1'\} $. As in ordinary category theory, the diagrams $ \cF: I \to \Cc $ indexed 
by $I$ are visualized as $\,\{X \xleftarrow{f} Z \xrightarrow{g} Y\}\,$, where $ X,Y,Z\in \Cc_0 $ and $f,g \in \Cc_1 $. The colimits of such diagrams are referred to as the pushouts in $\Cc$. Observe that  $\, I^{\triangleright} = (\Delta^1 \vee \Delta^1)\star \Delta^0 \cong  \Delta^1 \times \Delta^1 \,$ in $ \sset $: thus, if $ \widetilde{\cF}:  I^{\triangleright} \to \Cc $ is a right cone with vertex $A \in \Cc_0 $ representing $ \bcolim_I(F) = A $, the diagram
$ \widetilde{\cF} $ is a commutative square in $\Cc $ which we refer to as a {\it pushout square} and depict it as 
\begin{equation}
\la{push}
\begin{diagram}[small]
Z & \rTo^{f} & X\\
\dTo^{g} & & \dTo_i\\
Y & \rTo^j & \NWpbk A\\
\end{diagram}
\end{equation}
Note that, in an ordinary category, being `commutative' is a {\it property} of a square diagram.
In contrast, to give a commutative square \eqref{push} in an arbitrary $\infty$-category
amounts to giving --- besides the two pairs $ \{\, f, i\} $ and $ \{ g, j\} $ of composable morphisms (1-simplices) in $\Cc$ --- the {\it additional} data consisting of another morphism $ h: Z\to A $, and two homotopies (2-simplices) 
$ \sigma_1 \in \Cc_2 $ and $ \sigma_2 \in \Cc_2 $ that exhibit $h$ as compositions
$h \simeq i \circ f $ and $ h \simeq j \circ g $.
\end{eg}
\begin{eg}[Pullback squares] 
\la{PullSq}
Let $ I = (\Delta^1 \vee \Delta^1)^{\rm op}$ be the opposite
simplicial set to that of \Cref{PushSq}. This is an $\infty$-category isomorphic to the nerve of the opposite poset $ ([1] \sqcup_{[0]} [1])^{\rm op} = \{1 \rightarrow 0 \leftarrow 1'\} $.  The limits of diagrams $\cF: I \to \Cc $ are called pullbacks in $\Cc$, and the left cones $ \widetilde{F}: I^{\triangleleft} \to \Cc $ representing such limits are referred to as {\it pullback squares}. If $\, F = \{X \xrightarrow{f} Z \xleftarrow{g} Y\}\,$ and $ \blim_I(\cF) = A  $,  the corresponding pullback square $ \widetilde{\cF} $ is depicted as
\begin{equation}
\la{pull}
\begin{diagram}[small]
A \SEpbk & \rTo & X\\
\dTo & & \dTo_f\\
Y & \rTo^g & Z\\
\end{diagram}
\end{equation}
\end{eg}

%

\subsection{The fiber-cofiber construction}
\la{AS2}
From now on, we assume that $ \Cc $ is 
a pointed $\infty$-category. Recall that $\Cc$ is {\it pointed} if it contains a distinguished object $ 0 \in \Cc_0 $
which is both initial and terminal (in the sense that the mapping spaces
$ \Map_{\Cc}(0,X)$ and $ \Map_{\Cc}(X,0)$ are contractible for all objects $ X \in \Cc_0$).
We write $\, \Box := \Delta^1 \times \Delta^1 \,$ and label the vertices of this simplicial set as
\begin{equation}
\la{sqvert}
\begin{diagram}[small]
(0,0) & \rTo & (1,0)\\
\dTo & & \dTo\\
(0,1) & \rTo & (1,1)\\
\end{diagram}
\end{equation}
%
%
There are four canonical inclusions $\Delta^1 \into \Box $ 
corresponding to the four arrows in \eqref{sqvert} that we denote by 
\begin{equation}
\la{boxin}
\Delta^1 \cong \Delta^1 \times \Delta^0 \, \doublerightarrow{1 \times \delta_1}{ 1 \times \delta_0}\, \Delta^1 \times \Delta^1 \quad \mbox{and} \quad
\Delta^1 \cong \Delta^0 \times \Delta^1
\, \doublerightarrow{\delta_1 \times 1}{\delta_0 \times 1} \Delta^1 \times \Delta^1\,.
\end{equation}
(Thus, $ 1\times \delta_1 $ corresponds to the top horizontal arroW \, , $\, \delta_1 \times 1 $ to the left vertical arroW \, , etc.)

Following \cite{HA}, we define a {\it triangle} in $ \Cc $ to be a functor 
$ \cF: \Box \to \Cc $ such that $ (\delta_1 \times \delta_0)^*\cF = 0 \,$: that is, 
a $\Box$-diagram in $\Cc $ with the zero object $0$ at vertex $ (0,1)$:
\begin{equation}
\la{tria}
\begin{diagram}[small]
X & \rTo^{f} & Y\\
\dTo & & \dTo_{g}\\
0 & \rTo & Z\\
\end{diagram}
\end{equation}
We call a triangle a {\it fiber} (resp., {\it cofiber}) {\it sequence} in $\Cc $ if it is a pullback (resp., pushout) square in $\Cc$ (see Examples\, \ref{PushSq} and \ref{PullSq}). 
We write $ {\rm Tria}(\Cc)$, $\, {\rm Tria}_{\rm fib}(\Cc) $ and $  {\rm Tria}_{\rm cof}(\Cc) $   
for the full subcategories of the $\infty$-category $ \Fun(\Box, \Cc) $  
spanned by all triangles, the fiber and the cofiber ones, respectively. Clearly, there are natural inclusions 
$\,{\rm Tria}_{\rm fib}(\Cc) \subseteq  {\rm Tria}(\Cc) \supseteq  {\rm Tria}_{\rm cof}(\Cc)\,$.

Our goal is to construct a pair of adjoint endofunctors $\Fib $ and $\Cof$ on the $\infty$-category of morphisms in $\Cc$ called the fiber and cofiber functors.  We will use the following special case of \cite{lurie2009higher}*{Proposition~4.3.2.15}, which is an important principle in the theory of $\infty$-categories that has many applications (see, e.g., \cites{DKS23, DKSS24}).
\begin{prop}[Lurie]
\la{LuProp}
Let $ \Cc $ be an $\infty$-category, $\,I $ a small $\infty$-category, $\, I_0\subseteq I $ a full subcategory of $I$, and $ \cK \subseteq \Fun(I_0, \Cc) $ a full subcategory of the functor $\infty$-category $\Fun(I_0, \Cc)$. Let $ \Lc_u(\cK) $ $($resp., 
$ \Rc_u(\cK) $$)$ denote the full subcategory of $\Fun(I,\Cc) $ spanned by
all functors $\cF: I \to \Cc $ which are left $($resp., right$)$ Kan extensions of $ \cF|_{I_0} \in \cK $ along the inclusion $ u: I_0 \into I $.

\vspace*{1ex}

$(1)$ 
Assume that, for every diagram $I_0 \to \Cc $ in $\cK$ and every object $ x \in I $,
the induced diagram on the coslice category $ u_{/x} \to I_0 \to \Cc $ admits a colimit. Then, the restriction functor $\,u^*: \Lc_u(\cK) \to \cK \,$ is a trivial 
Kan fibration, and hence admits a $($homotopically$)$ unique section denoted by
$u_{!}: \cK \to \Fun(I, \Cc)$.

\vspace*{1ex}

$(2)$ 
Assume that, for every diagram $I_0 \to \Cc $ in $\cK$ and every object $ x \in I $,
the induced diagram on the slice category $ u_{x/} \to I_0 \to \Cc $ admits a limit. Then, the restriction functor $\,u^*: \Rc_u(\cK) \to \cK \,$ is a trivial 
Kan fibration and hence admits a $($homotopically$)$ unique section denoted by
$u_{*}: \cK \to \Fun(I, \Cc)$.
\end{prop}
Note that if the $\infty$-category $ \Cc $ is bicomplete, the assumptions of both parts of \Cref{LuProp} hold automatically, and the restriction functor $u^*: \Fun(I, \Cc) \to \Fun(I_0, \Cc) $ restricts to equivalences of $\infty$-categories: 
$$ 
\Lc_u(\cK)\, \xrightarrow{\sim} \,\cK\, \xleftarrow{\sim}\,\Rc_u(\cK) \, , 
$$
the section functors $ u_! $ and $ u_* $ providing the inverses.

To apply the above result in our situation we consider the forgetful functor $ 
{\rm Tria}(\Cc) \to \, \Fun(\Delta^1, \Cc) $ that associates to a triangle
\eqref{tria} the corresponding map $ f $. Observe that this last functor is just the restriction along the inclusion $ 1 \times \delta_1: \Delta^1 \into \Box $
(see \eqref{boxin}): 
\begin{equation}
\la{rfunc}
(1 \times \delta_1)^*:\, {\rm Tria}(\Cc)\,\to \, \Fun(\Delta^1, \Cc)    
\end{equation}
To construct the cofiber functor we factor the above inclusion $ 1 \times \delta_1 $ as
\begin{equation}\la{1del1}
1 \times \delta_1:\, \Delta^1 \xhookrightarrow{u} \Delta^1 \vee \Delta^1 \xhookrightarrow{v} \Box\, ,
\end{equation}
where $\,\Delta^1 \vee \Delta^1 $ is the $\infty$-category introduced in \Cref{PushSq} and $ v = (\delta_1 \times 1) \vee (1 \times \delta_1) $ is the embedding identifying $\,\Delta^1 \vee \Delta^1\,$ with the full subcategory of $\Box \,$ spanned by the vertices 
$ \{(0,1) \leftarrow (0,0) \rightarrow (1,0)\} $ (see  \eqref{sqvert}). Now, assuming that $\Cc $ admits (at least, finite) limits and colimits, we let $ \cK = \Fun(\Delta^1, \Cc) $ and apply \Cref{LuProp} twice: first,  taking the  right Kan extension along the functor $u$ and then the left Kan extension
along $v$. Since $u$ is a sieve functor, the right Kan extension along $u$ is the extension 
by the terminal (i.e., zero) object. The $\infty$-subcategory $\Rc_u(\cK) \subseteq \Fun(\Delta^1 \vee \Delta^1, \Cc)$ is thus spanned by the diagrams of the form $\{0 \leftarrow X \to Y\}$. On the other hand, the functor $v$ is isomorphic to the base inclusion of the right cone $ (\Delta^1 \vee \Delta^1) \into (\Delta^1 \vee \Delta^1)^{\triangleright} $ (see \Cref{PushSq}), hence the
left Kan extensions along $v$ are simply the pushouts. It follows that 
$\,\Lc_v(\Rc_u(\cK)) = {\rm Tria}_{\rm cof}(\Cc)\,$  in $ \Fun(\Box, \Cc) $.
By \Cref{LuProp},  the restriction functors
$$
{\rm Tria}_{\rm cof}(\Cc) \xrightarrow{v^*} \Rc_u(\Fun(\Delta^1, \Cc)) \xrightarrow{u^*} \Fun(\Delta^1, \Cc)
$$
are both trivial fibrations in $\sset $, and hence so is their composition $\,u^* v^* = (vu)^* = (1 \times \delta_1)^*\,$. Thus, \eqref{rfunc} admits a section given by the composition  $ v_! \,u_* $ of Kan extensions, using which we define 
\begin{equation}
\la{Cofun}
\Cof:\, \Fun(\Delta^1, \Cc) \xrightarrow{u_*} \Rc_u(\Fun(\Delta^1, \Cc)) 
\xrightarrow{v_!}  {\rm Tria}(\Cc) \xrightarrow{(\delta_0 \times 1)^*} \Fun(\Delta^1, \Cc)
\end{equation}
We call \eqref{Cofun} the {\it cofiber functor} on $ \Fun(\Delta^1, \Cc)$: 
pictorially, it can be described as
\[\begin{tikzcd}
	(A & X)
	\arrow["f", from=1-1, to=1-2]
\end{tikzcd}
\quad \longmapsto \quad
\begin{tikzcd}
	A & X \\
	0
	\arrow["f", from=1-1, to=1-2]
	\arrow[from=1-1, to=2-1]
\end{tikzcd}
\quad \longmapsto \quad
\begin{tikzcd}
	A & X \\
	  0 & X'
	\arrow["f", from=1-1, to=1-2]
	\arrow[from=1-1, to=2-1]
	\arrow["{\Cof(f)}", from=1-2, to=2-2]
	\arrow[from=2-1, to=2-2]
\end{tikzcd}
\quad \longmapsto \quad
\begin{tikzcd}
	X \\
	X'
	\arrow["{\Cof(f)}", from=1-1, to=2-1]
\end{tikzcd}
\]
%
%
where the intermediate triangle is the cofiber sequence associated to $f$.

Next, observe that the  functor $(\delta_0 \times 1)^*$
used in the last step of definition \eqref{Cofun} can itself be factored as $\,(\delta_0 \times 1)^* = p^* q^*\,$,
where the restrictions $p^*$ and $q^*$ are taken along the inclusions 
$\,\delta_0 \times 1:\, \Delta^1  \xhookrightarrow{p} (\Delta^1 \vee \Delta^1)^{\rm op} \xhookrightarrow{q} \Box $, which are `formally dual' to those in \eqref{1del1}. Thus,  $ \Cof $ can be written as the composition of four functors:
\begin{equation*}
\Fun(\Delta^1, \Cc) \xrightarrow{u_*} \Rc_u(\Fun(\Delta^1, \Cc)) 
\xrightarrow{v_!}  {\rm Tria}(\Cc) \xrightarrow{q^*}  \Lc_p(\Fun(\Delta^1, \Cc)) \xrightarrow{p^*} \Fun(\Delta^1, \Cc)\,.   
\end{equation*}
Each of these four functors has a right adjoint. Indeed, by \Cref{LuProp}, the first and the last functors, $u_*$ and $p^*$, are equivalences of $ \infty$-categories whose inverses $u^*$ and $p_!$
are both left and right adjoints. The functors $v_!$ and $ q^* $ are given by left Kan extension
and restriction and hence admit right adjoints: $ v^*$ and $q_*$, respectively. Hence, the composition 
\begin{equation*}
\Fun(\Delta^1, \Cc) \xleftarrow{u^*} \Rc_u(\Fun(\Delta^1, \Cc)) 
\xleftarrow{v^*}  {\rm Tria}(\Cc) \xleftarrow{q_*}  \Lc_p(\Fun(\Delta^1, \Cc)) \xleftarrow{p_!} \Fun(\Delta^1, \Cc)\,.   
\end{equation*}
represents a right adjoint to the cofiber functor $\Cof$. Since $u^* v^* = (vu)^* = (1 \times \delta_1)^* $, we can rewrite this right adjoint in the form similar to \eqref{Cofun}:
\begin{equation}
\la{Fifun}
\Fib:\, \Fun(\Delta^1, \Cc) \xrightarrow{p_!}  \Lc_p(\Fun(\Delta^1, \Cc)) 
\xrightarrow{q_*} {\rm Tria}(\Cc) \xrightarrow{(1 \times \delta_1)^*} \Fun(\Delta^1, \Cc)
\end{equation}
and call $\Fib$ the {\it fiber functor} on $ \Fun(\Delta^1, \Cc)$: 
pictorially, \eqref{Fifun} can be described as follows:
\[
\begin{tikzcd}
	X \\
	B
	\arrow["p"', from=1-1, to=2-1]
\end{tikzcd}    
\quad    \longmapsto \quad
\begin{tikzcd}
	& X \\
	0 & B
	\arrow["p", from=1-2, to=2-2]
	\arrow[from=2-1, to=2-2]
\end{tikzcd}
\quad    \longmapsto \quad
\begin{tikzcd}
	F & X \\
	0 & B
	\arrow["{\Fib(p)}", from=1-1, to=1-2]
	\arrow[from=1-1, to=2-1]
	\arrow["p", from=1-2, to=2-2]
	\arrow[from=2-1, to=2-2]
\end{tikzcd}    
\quad   \longmapsto \quad    
\begin{tikzcd}
	(F & X)
	\arrow["{\Fib(p)}", from=1-1, to=1-2]
\end{tikzcd}
\]
%
%
where the triangle in the middle is the fiber sequence associated to the morphism $ g $.
Thus, we have established the following result.
\begin{prop}
\la{Propadj}  
Let $\Cc $ be a pointed $\infty$-category with $($finite$)$ limits and colimits. The 
functors $\Cof$ and $\Fib$ defined by \eqref{Cofun} and \eqref{Fifun} form an adjunction of $\infty$-categories
\begin{equation}
\la{CFib}
\Cof: \,\Fun(\Delta^1, \Cc) \rightleftarrows \Fun(\Delta^1, \Cc):\, \Fib
\end{equation}
which we call the {\rm fiber-cofiber adjunction}.
\end{prop}
\begin{rmk}
The fiber and cofiber functors for pointed $\infty$-categories were defined in \cite{HA}
(for a brief outline of the construction of these functors based on \Cref{LuProp} 
see {\it loc. cit.}, Remark~1.1.1.7). To avoid confusion, however, we should point out that our definition differs from that of \cite{HA}: the functors $\Cof$ and $\Fib$ are defined in \cite{HA} as functors $\Fun(\Delta^1,\Cc) \to \Cc $ with values in $ \Cc$, and therefore do not form an adjunction. Our adjunction \eqref{CFib} is analogous to the fiber-cofiber adjunction for
Grothendieck derivators constructed by M. Groth in \cite{Gr13} and used in his subsequent 
work on abstract representation theory in the context of stable derivators (see, e.g., \cites{GS18, GR19}). 
\end{rmk}

The advantage of (re)defining $\Cof$ and $\Fib $ as adjoint functors is the existence of two natural transformations --- the adjunction unit and counit --- that can be used to construct
a canonical $\infty$-functor $ \Xi: \Fun(\Delta^1, \Cc) \to \Fun(\cube, \Cc) $ with values in 3D cubical diagrams satisfying the (second) Mather cube axiom.
 The key to this construction is the natural involution ('reflection symmetry') of $\Box$ interchanging the `vertical' and `horizontal' morphisms. Applying this involution to the unit and counit diagrams determines the connecting morphisms in the Ganea tower. In what follows
we will focus on the counit leaving the dual (unit) construction as an exercise to the reader. 

Recall that, by definition, the counit of \eqref{CFib} is a natural transformation (1-simplex) in $\Fun(\Fun(\Delta^1, \Cc), \,\Fun(\Delta^1, \Cc))$ represented by the functor
\begin{equation}
\la{counit}
\varepsilon:\, \Fun(\Delta^1, \Cc) \times \Delta^1\,\to\,\Fun(\Delta^1, \Cc)
\end{equation}
such that $\, \varepsilon|_{\Fun(\Delta^1, \Cc) \times \{0\}}\,=\,\Cof \circ \Fib\,$ and $ \,\varepsilon|_{\Fun(\Delta^1, \Cc) \times \{1\}}\,=\,\id_{\Fun(\Delta^1, \Cc)}\,$. 
We rewrite this functor using the following canonical isomorphisms in $\sset$:
\begin{eqnarray}
\Fun(\Fun(\Delta^1, \Cc) \times \Delta^1,\, \Fun(\Delta^1, \Cc)) &\cong &
\Fun(\Fun(\Delta^1, \Cc),\,  \Fun(\Delta^1,\, \Fun(\Delta^1, \Cc)))\nonumber\\
&\cong & 
\Fun(\Fun(\Delta^1, \Cc),\,  \Fun(\Delta^1 \times \Delta^1, \Cc)) \nonumber\\
& = & 
\Fun(\Fun(\Delta^1, \Cc),\,  \Fun(\Box, \Cc))\la{Delcube}
\end{eqnarray}
It is easy to see that, under these isomorphisms,  \eqref{counit} corresponds to the functor
\begin{equation}
\la{counit1}
\widetilde{\varepsilon}:\, \Fun(\Delta^1, \Cc)\,\to\,\Fun(\Box, \Cc)  
\end{equation}
the values of which on the two vertical edges of $\Box$ are 
\begin{equation}
\la{counit2}
(\delta_1 \times 1)^* \widetilde{\varepsilon} \,=\,\Cof \circ \Fib\, ,\quad
(\delta_0 \times 1)^* \widetilde{\varepsilon} \,
\,=\,\id_{\Fun(\Delta^1, \Cc)}\, ,
\end{equation}
and the value on the top horizontal edge is
\begin{equation}
\la{horc}
 (1 \times \delta_1)^*\widetilde{\varepsilon} \,=\, s^*_0 \, \delta_1^*  \, ,  
\end{equation}
where the right-hand side is the composition of two restriction functors%
\begin{equation}
\la{const}
\Fun(\Delta^1, \Cc) \xrightarrow{\delta_1^*} \Fun(\Delta^0, \Cc) \xrightarrow{s^*_0}\Fun(\Delta^1, \Cc) 
\end{equation}
sending a morphism $ X \to Y $ in $\Cc$ to the identity morphism $\id_X$ on its source. 
Now, the restriction of \eqref{counit1} to the remaining (bottom horizontal) edge of $\Box$ is 
an interesting functor that we christen in the following
\begin{defn}
We call $ \,\Gan := (1 \times \delta_0)^* \widetilde{\varepsilon}:\, \Fun(\Delta^1, \Cc) \to \Fun(\Delta^1, \Cc) \,$  the {\it Ganea functor} on $\Cc$. Pictorially, it can be represented by
\begin{equation}
\la{Gan}
\begin{tikzcd}
	X \\
	B
	\arrow["p"', from=1-1, to=2-1]
\end{tikzcd}    
\quad   \stackrel{\widetilde{\varepsilon}}{\longmapsto} \quad
 \begin{tikzcd}
	X & X \\
	{X_1} & B
	\arrow["{\id}", from=1-1, to=1-2]
	\arrow["{\pi(p)}"', from=1-1, to=2-1]
	\arrow["p", from=1-2, to=2-2]
	\arrow["{p_1}", from=2-1, to=2-2]
\end{tikzcd}
\quad  \longmapsto \quad  
\begin{tikzcd}
	X_1 \\
	B
	\arrow["{p_1}"', from=1-1, to=2-1]
\end{tikzcd}   
\end{equation}
where $\pi(p) := \Cof(\Fib(p))$ is the composition of the functors \eqref{Cofun} and \eqref{Fifun}.
\end{defn}
The above diagrammatic description of the functor $\Gan$ suggests that there is a natural transformation   
\begin{equation}
\la{Gantr}
\pi: \id_{\Fun(\Delta^1, \Cc)} \to \Gan 
\end{equation}
obtained by `flipping the square in the middle of \eqref{Gan} across its diagonal'. Formally, this can be defined as follows. 
Let us identify the standard $2$-simplex $\Delta^2 = N([2]) $ with the full $\infty$-subcategory of $\Box $ spanned by the vertices $\{(0,0),(1,0),(1,1)\}$
(see  \eqref{sqvert}):
\begin{equation*}
\la{sqvert1}
\begin{tikzcd}
    (0,0) \ar[r]\ar[dr] & (1,0) \ar[d] \\
    & (1,1)
\end{tikzcd}
\end{equation*}
and let $\,i: \Delta^2 \into \Box \,$ denote the corresponding inclusion. By \Cref{LuProp}$(2)$, the restriction
functor $i^*: \Rc_i(\Fun( \Delta^2, \Cc)) \to \Fun( \Delta^2, \Cc)$ is a trivial fibration that admits a section: 
\begin{equation*}
\la{rKan} 
i_*: \Fun( \Delta^2, \Cc) \to \Rc_i(\Fun( \Delta^2, \Cc)) \into \Fun(\Box, \Cc)
\end{equation*}
which is the right Kan extension along  $i$. Using this Kan extension, we define the functor 
%
\begin{equation}
\la{tilpi} 
\widetilde{\pi}:\, \Fun(\Delta^1, \Cc) \xrightarrow{\widetilde{\varepsilon}}
\Fun(\Box, \Cc) \xrightarrow{\tau^*} \Fun(\Box, \Cc) 
\xrightarrow{i^*} \Fun(\Delta^2, \Cc)  \xrightarrow{i_*} \Fun(\Box, \Cc)
\end{equation}
where $\widetilde{\varepsilon}$ is the counit \eqref{counit1} and $\tau^* $ is the automorphism of $ \Fun(\Box, \Cc) $ 
induced by the canonical involution $ \tau: \Box \to \Box \,$ interchanging the two factors in $ \Box = \Delta^1 \times \Delta^1 $. 
\begin{lem}
Under the identification \eqref{Delcube}, the functor \eqref{tilpi} corresponds to
a map
$$
\pi:\, \Fun(\Delta^1, \Cc) \times \Delta^1 \to \Fun(\Delta^1, \Cc)
$$
defining the natural transformation \eqref{Gantr}.
\end{lem}
\begin{proof}
We need to show that   
\begin{equation}
\la{counit3}
(\delta_1 \times 1)^*\, \widetilde{\pi}\,\cong \,\id_{\Fun(\Delta^1, \Cc)}\, , \quad
(\delta_0 \times 1)^*\, \widetilde{\pi} \, \cong\, \Gan\, .
\end{equation}
First, let us observe that $\,i: \Delta^2 \into \Box \,$ is a fully faithful functor which can be obtained by taking the nerve of the map of posets $\, [2] \into [1] \times [1]$, $\, \{0\le 1 \le 2\} \mapsto \{(0,0) \le (1,0) \le (1,1)\} $. Hence
the Kan extension functor $i_*$ is also fully faithful, and there is an equivalence 
\begin{equation}
\la{ii}
i^* i_* \,\cong\, \id_{\Fun(\Delta^2,\Cc)}
\end{equation}
Next, observe that the only vertex of $ \Box $ that is not in the image of 
$i$ is $ (0,1) $. The corresponding coslice $\infty$-category
$\, i_{(0,1)/} = \Box_{(0,1)/} \times_{\Box} \Delta^2 \,$ is isomorphic to the nerve of the subposet $ [2]_{\ge (0,1)} \subseteq [2] $, which consists of the 
single element $ \{(1,1)\}$. Hence $ i_{(0,1)/} \cong \Delta^0 $ and the forgetful functor $ i_{(0,1)/} \to \Delta^2 $ is isomorphic to the vertex inclusion 
$ (1,1): \Delta^0 \to \Delta^2 $. Given a diagram $ \cF: \Delta^2 \to \Cc $, its right Kan extension 
$i_*\cF: \Box \to \Cc $ at vertex $(0,1)$ is therefore given by
$$
(i_*\cF)(0,1)\,= \,\blim\{i_{(0,1)/} \to \Delta^2 \xrightarrow{\cF} \Cc\}\,\cong\,
\,\blim\{\Delta^0 \xrightarrow{(1,1)} \Delta^2 \xrightarrow{\cF} \Cc\} = \cF(1,1)\, ,
$$
the value of $ i_*\cF$ at the morphism $ (0,1) \to (1,1) $ in $\Box$ being the identity morphism $ \id_{\cF(1,1)} \in \Cc_1 $. Thus, for any $\cF \in \Fun(\Delta^2, \Cc)$, the functor $ i_*\cF $ can be diagrammatically described by
%
\[\begin{tikzcd}
	{\cF(0,0)} & {\cF(1,0)} \\
	{\cF(1,1)} & {\cF(1,1)}
	\arrow[from=1-1, to=1-2]
	\arrow[from=1-1, to=2-1]
	\arrow[from=1-1, to=2-2]
	\arrow[from=1-2, to=2-2]
	\arrow[equals, from=2-1, to=2-2]
\end{tikzcd}\]
where the diagonal arrow represents the morphism 
$ \cF[(0,0) \to (1,1)] \in \Cc_1 $. 
We see that this last morphism is equivalent to  $ i_*\cF[(0,0) \to (0,1)] $, which is represented by the left vertical arrow in the above diagram. More precisely, this
equivalence can be expressed as an equivalence of functors
$\,\Fun(\Delta^2, \Cc) \to \Fun(\Delta^1, \Cc)\,$:
\begin{equation}
\la{isodel1}
(\delta_1 \times 1)^* \,i_*  \,\cong \, \delta_1^*     
\end{equation}
where the $\delta_1^*$ on the right is the restriction functor
along the coface inclusion $ \delta_1: \Delta^1 \to \Delta^2 $
opposite to the vertex $(1,0)$. Using \eqref{isodel1}, we can now verify 
the first condition in \eqref{counit3} by a formal calculation:
\begin{eqnarray*}
(\delta_1 \times 1)^*\, \widetilde{\pi} 
& = &
(\delta_1 \times 1)^*\,i_*\,i^*\,\tau^*\,\widetilde{\varepsilon}
\,\cong\, 
\delta_1^*\,i^*\,\tau^*\,\widetilde{\varepsilon}
\,= \,
(\tau\,i\,\delta_1)^*\widetilde{\varepsilon} \\*[1ex]
&=& (i\,\delta_1)^*\widetilde{\varepsilon} \,
= \,\delta_1^*\,i^*\widetilde{\varepsilon}
\,\cong\, (\delta_0 \times 1)^*\widetilde{\varepsilon}
\,=\, \id_{\Fun(\Delta^1, \Cc)}\, ,
\end{eqnarray*}
where we also used the (obvious) fact that the image of $\, i\,\delta_1: \Delta^1 \into \Box $ is stable under the involution $\tau$. 

For the second condition in \eqref{counit3}, we notice that the map $ \delta_0 \times 1: \Delta^1 \to \Box $ can be factored as $ \delta_0 \times 1 = i \, \delta_0 $, where $ \delta_0: \Delta^1 \to \Delta^2 $ is the edge inclusion opposite to the vertex $(0,0)$. Using \eqref{ii}, we then have
\begin{eqnarray*}
(\delta_0 \times 1)^*\, \widetilde{\pi} 
&=&
(\delta_0 \times 1)^*\,i_*\,i^*\,\tau^*\,\widetilde{\varepsilon}
\,=\,
\delta_0^* \,i^*\,i_* \,i^*\,\tau^*\,\widetilde{\varepsilon}
\,\cong\,
\delta_0^* \,i^*\,\tau^*\,\widetilde{\varepsilon}
\,=\, (\tau\,i\,\delta_0)^* \,\widetilde{\varepsilon}\\*[1ex] 
&=& 
(\tau\,(\delta_0\times 1))^* \,\widetilde{\varepsilon} 
\,=\,(\tau\,(\delta_0\times 1))^* \,\widetilde{\varepsilon} 
\,=\,(1 \times \delta_0)^* \,\widetilde{\varepsilon} \,=\,\Gan\,.
\end{eqnarray*}
This completes the proof of the lemma.
\end{proof}
Our next goal is to prove an $\infty$-categorical analog of the classical Ganea Theorem showing how the (homotopy) fiber of a morphism in $\Cc$ changes under the functor $\Gan $.
To this end we consider a `comparison map' $\, \phi: \Fib \to \Fib \circ\Gan $ defined
by the natural transformation
\begin{equation}
\la{phi}
\phi:\,\Fun(\Delta^1, \Cc) \times \Delta^1 \xrightarrow{\pi} \Fun(\Delta^1, \Cc)
\xrightarrow{\Fib} \Fun(\Delta^1, \Cc)\, ,
\end{equation}
where $ \Fib $ is the fiber functor on $\Fun(\Delta^1, \Cc) $ defined by \eqref{Fifun}.
Under \eqref{Delcube}, this natural transformation corresponds to a functor
$ \widetilde{\phi}: \Fun(\Delta^1, \Cc) \to \Fun(\Box, \Cc) $ that takes a morphism
$ (X \xrightarrow{p} B) \in \Fun(\Delta^1, \Cc)_0 $ to the commutative square 
\begin{equation}
    \la{fibsq}
\begin{tikzcd}
	F & {F_1} \\
	X & {X_1}
	\arrow[from=1-1, to=1-2]
	\arrow["j"', from=1-1, to=2-1]
	\arrow["{j_1}", from=1-2, to=2-2]
	\arrow["\pi", from=2-1, to=2-2]
\end{tikzcd}    
\end{equation}
where $\,j := \Fib(p)\,$ and $ j_1 := \Fib(\Gan(p)) $ in $ \Fun(\Delta^1, \Cc)$.
We describe the square \eqref{fibsq} by realizing it as a face of a 3-dimensional 
cubical diagram with a special property formalized in \Cref{defn:cube-axiom} below.
To simplify the notation we set $ \cube := \Delta^1 \times \Delta^1 \times\Delta^1  $ and denote
by $ i_{12}: \, \Delta^1 \times \Delta^1 \times\{0\}  \into \cube \,$,
$ i_{12}': \,\Delta^1 \times \Delta^1 \times\{1\}  \into \cube $, etc. the six natural inclusions of $ \Box $ into $\cube$. We call these inclusions the faces of $\cube $, referring to $ i_{kl} $ and $ i_{kl}' $ (with same $k,l$) as the {\it opposite}
faces. Given a functor $ \cF: \cube \to \Cc $,  we will also refer to the restriction functors $ i_{kl}^* \,\cF: \Box \to \Cc $ along the faces of $ \cube $ as {\it faces of $ \cF $}.
Pictorially, $ \cF $ can be visualized as 
\begin{equation}\la{cubic}
        \begin{tikzcd}[row sep=tiny, column sep=small]
            & \bullet \ar[dl]\ar[dd]\ar[rr] && \bullet \ar[dl]\ar[dd] \\
            \bullet \ar[rr,crossing over]\ar[dd] && \bullet & \\
            & \bullet \ar[dl]\ar[rr] && \bullet \ar[dl] \\
            \bullet \ar[rr] && \bullet \ar[from=uu,crossing over," "',near end]
        \end{tikzcd}
    \end{equation}
The next definition is a slight modification of \cite{kerodon}*{Definition 7.7.5.1}.
\begin{defn}
\label{defn:cube-axiom} A diagram $ \cF: \cube \to \Cc $ is called a {\it Mather cube} in $\Cc $ if $($`up to rotation'$)$ the bottom face of $\cF$ is a pushout square, while the four side faces are pullback squares  in $\Cc$, see \eqref{cubic}. We say that $\Cc$ satisfies the {\it $($second$)$ Mather cube axiom} if the top face of any Mather cube in $\Cc$ is also a pushout square in $\Cc$.   
\end{defn}
\noindent
We remark that the Mather cube axiom for a (presentable) 
$\infty$-category $\Cc$ can be reformulated by saying that
the pushouts in $\Cc $ are universal in the sense that
the pullback functor $f^*: \Cc_{/S} \to \Cc_{/T} $ preserves
pushouts for any morphism $f: T \to S $ in $\Cc$ (see \cite{kerodon}*{7.7.5.2}).

\begin{eg}
\la{Ex1}
The $\infty$-category ${\mathcal S}_{\ast}$ of (pointed) topological spaces satisfies the Mather cube axiom. This is essentially the contents of Mather's classical theorem \cite{mather1976pull} (see \cite{kerodon}*{7.7.5.3}). 
\end{eg}
\begin{eg}
\la{Ex2}
Recall that an $\infty$-category $ \Cc $ is an {\it $\infty$-topos} if it can be obtained by an (accessible) left exact localization $\, {\mathscr P}(\I)  \to \Cc \,$ of the $\infty$-category ${\mathscr P}(\I)$ of ($\mathcal S$-valued) presheaves on some small $\infty$-category $\I$ (see \cite{lurie2009higher}*{Def. 6.1.0.4}). By a theorem of Ch. Rezk, the $\infty$-topoi are characterized among the $\infty$-categories by the 3 properties: $\,(1)$ $\Cc$ is presentable, $\,(2)$ all colimits in $\Cc$ are universal, and $(3)$ the (non-full) subcategory of Cartesian morphisms in $\Fun(\Delta^1, \Cc)$ is closed under colimits (see \cite{lurie2009higher}*{6.1.6.8} and also \cite{ABFJ20}*{Example 2.2.4}). Properties $(1)$ and $(2)$ show, in particular, that the Mather
cube axiom holds in any (pointed) $\infty$-topos $\Cc$. This is a far-reaching generalization of \Cref{Ex1}.
\end{eg}
\begin{eg}
\la{Ex3}
By  \cite{kerodon}*{7.7.5.5}, the $\infty$-category $\Cat_{\infty}$ of small $\infty$-categories (quasi-categories) does not satisfy the Mather cube axiom.
\end{eg}
We are now in position to state the main result of this Appendix, which we refer to as 
the Ganea Theorem for $\infty$-categories.
\begin{thm}\la{GaTh}
Let $\Cc $ be a pointed $\infty$-category with $($finite$)$ limits and colimits. 
There is a canonical functor $ \Xi:\, \Fun(\Delta^1, \Cc)\,\to\, \Fun(\cube, \Cc) $
represented $($on vertices$)$ by
\begin{equation}
\la{Gancube}
  \begin{tikzcd}[scale cd=1]
	X \\
	B
	\arrow["p"', from=1-1, to=2-1]
\end{tikzcd}  
\qquad \longmapsto\qquad
\begin{tikzcd}[column sep=tiny
, row sep=tiny]
            & \Omega B \times F \ar[dl]\ar[dd]\ar[rr] && F \ar[dl]\ar[dd, "j"] \\
            \Omega B \ar[rr,crossing over]\ar[dd] && F_1 & \\
            & F \ar[dl]\ar[rr, "j_1"] && X \ar[dl, "\pi"] \\
            0 \ar[rr] && X_1 \ar[from=uu,crossing over," "',near end]
        \end{tikzcd}
\end{equation}
whose image lies in the full $\infty$-subcategory of $\Fun(\cube, \Cc)$ spanned by the 
Mather cubes. Consequently, if $\Cc $ satisfies the $($second$)$ Mather cube axiom, there is
an equivalence in $\Cc$
\begin{equation}
F_1 \,\cong \,F \ast \Omega B := \bcolim\{ F \leftarrow  F \times \Omega B \to \Omega B\}\, ,     
\end{equation}
where $ \Omega B = \blim\{0 \to B \leftarrow 0\}$ is the loop object on $B \in \Cc_0$.
\end{thm}
\begin{proof} 
We construct $\Xi$ in four steps each of which consists in performing some natural operation on functor categories. Note that the functors $\, \Fun(\Delta^1, \Cc) \to \Fun(\cube, \Cc) \,$ can be interpreted as natural transformations of functors $ \Fun(\Delta^1, \Cc) \to \Fun(\Box, \Cc) $, using the canonical isomorphism
\begin{equation}
\la{isobox}
\Fun(\Fun(\Delta^1, \Cc),\, \Fun(\cube, \Cc))\, \cong\
\Fun(\Fun(\Delta^1, \Cc) \times \Delta^1,\, \Fun(\Box, \Cc)) 
\end{equation}
Under \eqref{isobox}, the functor $\Xi$ itself corresponds  to a functor $ \Fun(\Delta^1, \Cc) \times\Delta^1 \to \Fun(\Box, \Cc) $ whose restriction to $ \Fun(\Delta^1, \Cc) \times \{1\} $ is isomorphic to the diagram  \eqref{fibsq}. 
Using `rotational symmetries' of the simplicial cube $\cube$, we can also regard $\Xi $ as a natural transformation of functors corresponding to other pairs of opposite faces. This gives us an additional flexibility.

As a first step, we use the fiber functor \eqref{Fifun} in combination with Kan extensions along  the inclusions \eqref{1del1} to define 
\begin{equation}
 \la{compF}
\cF:\, \Fun(\Delta^1, \Cc) \xrightarrow{\Fib} \Fun(\Delta^1, \Cc) \xrightarrow{v_{!} \,u_*} 
\Tria_{\rm cof}(\Cc) \subseteq \Fun(\Box, \Cc)
\end{equation}
This functor assigns to a morphism $ X \xrightarrow{p} B $ the cofiber sequence 
\begin{equation}
    \la{fibsq1}
\begin{tikzcd}
	F & {X} \\
	0 & {X_1}
	\arrow["j", from=1-1, to=1-2]
	\arrow[from=1-1, to=2-1]
	\arrow["{\pi}", from=1-2, to=2-2]
	\arrow[from=2-1, to=2-2]
\end{tikzcd}    
\end{equation}
where the top arrow $j: F \to X $ in the fiber inclusion map.

Next, we consider the standard (colimit) adjunction (see \cite{kerodon}*{Proposition 7.1.1.18}):
\begin{equation*}
\bcolim_{\Box}:\, \Fun(\Box, \Cc)\, \leftrightarrows\, \Cc: \partial_{\Box} 
\end{equation*}
where the right adjoint is the diagonal functor of (shape $\Box$). This adjunction naturally extends
to an adjunction between functor categories  (see, e.g., \cite{La21}*{Proposition 5.1.16}): 
\begin{equation}
\la{KBox}
\bcolim_{\Box}^K:\, \Fun(K, \Fun(\Box, \Cc))\, \leftrightarrows\, \Fun(K, \Cc): \partial^K _{\Box} 
\end{equation}
where  $ K  $ is any simplicial set. Note that the functor $\bcolim_{\Box}^K$  is easy to compute: since $\Box $ has a final object (the vertex $(1,1)$, see \eqref{sqvert}), it is  given by composition of the canonical isomorphism  $\Fun(K, \Fun(\Box, \Cc)) \cong \Fun(\Box, \Fun(K, \Cc)) $ with evaluation at $(1,1)$.

Now, taking $ K := \Fun(\Delta^1, \Cc) $ in \eqref{KBox} and regarding the functor  \eqref{compF} as an object in $ \Fun(\Fun(\Delta^1, \Cc), \Fun(\Box, \Cc)) $, we have the natural map (homotopy equivalence)
\begin{equation}
\la{Kequiv}
\Map_{\Fun(\Fun(\Delta^1, \Cc), \Cc)}(\bcolim_{\Box}^K(\cF),\,\delta_0^*)\,
\xrightarrow{\eta^* \circ \,\partial^K_{\Box}}\,
\Map_{\Fun(\Fun(\Delta^1, \Cc), \Fun(\Box,\Cc))}(\cF,\,\partial_{\Box}^K(\delta_0^*))
\end{equation}
determined by the unit $\eta $ of the adjunction \eqref{KBox}. Observe that the Ganea functor $\Gan$ defined by \eqref{Gan} can be viewed as an 
object in the functor category
$$ 
\Fun(\Fun(\Delta^1, \Cc), \,\Fun(\Delta^1,\Cc))  \,\cong\, \Fun(\Delta^1,\, \Fun(\Fun(\Delta^1, \Cc), \,\Cc))\,. 
$$ 
This object satisfies the restriction conditions: $\, \delta^*_0(\Gan) = \delta^*_0 \,$ and $ \delta^*_1(\Gan)\,\cong \, \bcolim_{\Box}^K(\cF) $ in $\Fun(\Fun(\Delta^1, \Cc), \Cc) $; 
hence, by \eqref{Mapdef}, it can be viewed as a vertex ($0$-simplex) of the mapping space in the domain of \eqref{Kequiv}, i.e. $\,\Gan \in \Hom_{\Fun(\Fun(\Delta^1, \Cc), \Cc)}(\bcolim_{\Box}^K(\cF),\,\delta_0^*)\,$. The map \eqref{Kequiv} then sends $ \Gan $ 
to 
\begin{equation}
\la{adjGan}
\eta^*(\partial_{\Box}^K \, \Gan):= \partial_{\Box}^K \circ \Gan \circ \eta \,\in\,
\Hom_{\Fun(\Fun(\Delta^1, \Cc), \Fun(\Box,\Cc))}(\cF,\,\partial_{\Box}^K(\delta_0^*))\,,
\end{equation}
which is a morphism of functors $ \cF \to \partial_{\Box}^K(\delta_0^*) $ in
$\Fun(\Fun(\Delta^1, \Cc), \Fun(\Box,\Cc))$. By definition, such morphisms are represented by
the functors $\,\Fun(\Delta^1, \Cc) \times \Delta^1 \to  \Fun(\Box,\Cc)\,$, which, in turn,
correspond bijectively --- via the isomorphism \eqref{isobox} --- to the functors 
$\,\Fun(\Delta^1, \Cc) \to \Fun(\cube, \Cc) \,$. Thus, the morphism \eqref{adjGan} determines a cubical functor
\begin{equation}
\la{Gann}
\Gan^{\natural}: \ \Fun(\Delta^1, \Cc) \to \Fun(\cube, \Cc)   
\end{equation}
which is easily seen to be represented (on vertices) by
\begin{equation}
\la{FBfun}
  \begin{tikzcd}[scale cd=1]
	X \\
	B
	\arrow[from=1-1, to=2-1]
\end{tikzcd}  
\qquad \longmapsto\qquad
\begin{tikzcd}[column sep=tiny
, row sep=tiny]
            & F \ar[dl]\ar[dd]\ar[rr] && B \ar[dl]\ar[dd] \\
            X \ar[rr,crossing over]\ar[dd] && B & \\
            & 0 \ar[dl]\ar[rr] && B \ar[dl] \\
            X_1 \ar[rr] && B \ar[from=uu,crossing over," "',near end]
        \end{tikzcd}
\end{equation}
For the final step,
we extend the fiber functor \eqref{Fifun} to cubic diagrams by
\begin{equation*}
\Fib_{\Box}: \, \Fun(\cube, \Cc) \cong \Fun(\Box,\,\Fun(\Delta^1, \Cc)) 
\xrightarrow{\Fun(\Box,\Fib)}    \Fun(\Box,\,\Fun(\Delta^1, \Cc)) \cong \Fun(\cube, \Cc)
\end{equation*}
Informally speaking, $\Fib_{\Box}$ `interprets' a cubic diagram in $\Cc$ as a commutative square of morphisms (cube's horizontal arrows) and then `replaces' each of these morphisms with its fiber. When combined with \eqref{Gann}, this yields the functor
\begin{equation}
\la{Xiprime}
\Xi' := \Fib_{\Box} \circ \Gan^\natural :\  
\Fun(\Delta^1, \Cc) \to \Fun(\cube, \Cc) \to  \Fun(\cube, \Cc)
\end{equation}
which can be depicted as
\begin{equation}
\la{FBfun2}
  \begin{tikzcd}[scale cd=1]
	X \\
	B
	\arrow[from=1-1, to=2-1]
\end{tikzcd}  
\quad \longmapsto\quad
\begin{tikzcd}[column sep=tiny, row sep=tiny]
            & F \ar[dl]\ar[dd]\ar[rr] && B \ar[dl]\ar[dd] \\
            X \ar[rr,crossing over]\ar[dd] && B & \\
            & 0 \ar[dl]\ar[rr] && B \ar[dl] \\
            X_1 \ar[rr] && B \ar[from=uu,crossing over," "',near end]
        \end{tikzcd}
\quad \longmapsto\quad
\begin{tikzcd}[column sep=tiny, row sep=tiny]
            &  \Omega B \times F \ar[dl]\ar[dd]\ar[rr] && F \ar[dl]\ar[dd] \\
            F \ar[rr,crossing over]\ar[dd] && X & \\
            & \Omega B \ar[dl]\ar[rr] && 0 \ar[dl] \\
            F_1 \ar[rr] && X_1 \ar[from=uu,crossing over," "',near end]
\end{tikzcd}           
\end{equation}
Now, observe that, by construction, the right face of the target cube in \eqref{FBfun2} is the cofiber sequence \eqref{fibsq1}, and hence a pushout square in $\Cc$. On the other hand, the four faces adjacent to this right face arise from applying the fiber functor 
\eqref{Fifun} to morphisms over the same base;  by transitivity of pullbacks (see \cite{kerodon}*{Proposition~7.6.2.28}), they are pullback squares in $\Cc$. Thus, 
the image of \eqref{FBfun2} is a Mather cube, i.e. the functor $\Xi'$ takes  values in the full subcategory of $ \Fun(\cube, \Cc) $ spanned by such cubes. To complete the proof it remains to note that the functor $\Xi $ can be obtained from $\Xi'$ by `rotating the cube'
(i.e., twisting $\Xi'$ with an automorphism of $\,\Fun(\cube, \Cc) \,$ defined by a suitable transformation of the simplicial set $\cube$).
\end{proof}

Next, using the natural transformation \eqref{Gantr}, we can construct a diagram of endofunctors:
\begin{equation}
\la{Gtower}
\Gan^{\bullet}:\ \id \xrightarrow{\pi} \Gan \xrightarrow{\pi_1} \Gan^2 \xrightarrow{\pi_2} \Gan^3 \to \ldots 
\end{equation}
where the functors $ \Gan^m: \Fun(\Delta^1, \Cc) \to \Fun(\Delta^1, \Cc)\,$  and the morphisms $ \pi_m: \Gan^m \to \Gan^{m+1} $ are defined inductively by composition:
$\,\Gan^m := \Gan \circ \Gan^{m-1} $ and
$$
\pi_m: \,\Fun(\Delta^1, \Cc) \times \Delta^1 \xrightarrow{\pi_{m-1}} \Fun(\Delta^1, \Cc)
\xrightarrow{\Gan} \Fun(\Delta^1, \Cc)\ ,\quad \forall\,m \ge 1\,.
$$
Pictorially, when evaluated on objects, the diagram \eqref{Gtower} reads
 \begin{equation*}
    \label{eq:ganea-tow}
    \begin{tikzcd}
        X \ar[r, "\pi"]\ar[d,"p"] & X_1 \ar[d,"p_1"]\ar[r, "\pi_1"] & X_2 \ar[d,"p_2"]\ar[r, "\pi_2"] & \ldots \ar[r, "\pi_{m-1}"] & X_m \ar[r, "\pi_m"]\ar[d,"p_m"] & \ldots \\
        B \ar[r,equal] & B \ar[r,equal] & B \ar[r,equal] & \ldots \ar[r,equal] & B \ar[r,equal] & \ldots 
    \end{tikzcd}
    \end{equation*}

Note that the diagram $\Gan^\bullet $ is indexed by the simplicial set $\, I^{\infty} := \Delta^1 \sqcup_{\Delta^0} \Delta^1 \sqcup_{\Delta^0} \ldots $, which is {\it not} an
$\infty$-category. However, $I^{\infty}$ can be identified with the spine of 
the infinite simplex $ \Delta^\infty = N \Z_+ $ (the nerve of the ordered set $\Z_+$ of natural numbers), and the restriction map  $ j^*: \Fun(\Delta^\infty, \Dc) \to \Fun(I^{\infty}, \Dc) $ along the inclusion
$j: I^{\infty} \into \Delta^\infty $ is known to be a trivial Kan fibration for any 
$\infty$-category $\Dc$ (see \cite{kerodon}*{Theorem 1.5.7.1}). Hence, by \cite{lurie2009higher}*{Proposition 4.3.2.17}, the functor $j^*$ has a left adjoint given by left Kan extension:
\begin{equation}
\la{resext}
j_{!}:\ \Fun(I^{\infty}, \Dc)\,\leftrightarrows\,\Fun(\Delta^\infty, \Dc)\ : j^*
\end{equation}
Taking $ \Dc = \Fun(\Fun(\Delta^1, \Cc), \Fun(\Delta^1, \Cc)) $ and applying the corresponding   $j_!$ to the diagram \eqref{Gtower}, we can extend it to a $\Delta^\infty$-diagram (mapping telescope):
$$
\widetilde{\Gan}^\bullet := j_! \,\Gan^\bullet:\, \Delta^\infty \,\to \,\Fun(\Fun(\Delta^1, \Cc), \Fun(\Delta^1, \Cc))\,,
$$
which we call the {\it Ganea tower} in $\Cc$. Since the left Kan extensions preserve
colimits, we have a canonical equivalence of endofunctors on $\Fun(\Delta^1, \Cc)$:
\begin{equation}
\la{eqcolim}    
\bcolim_{I^\infty} \,\Gan^{\bullet}\,\xrightarrow{\sim}\,
\bcolim_{\Delta^\infty} \,\widetilde{\Gan}^{\bullet}\,
\end{equation}
induced by the  restriction functor in \eqref{resext}.

Now, consider the functor $\,s_0^*\,\delta_0^*: \Fun(\Delta^1, \Cc) \to \Fun(\Delta^1, \Cc)\,$ that takes a morphism $ X \to B $ in $\Cc$ to the identity $\id_B$ on its base.
There is a natural transformation $\,\alpha: \id \to s_0^* \,\delta^*_0 \,$ of endofunctors on 
$ \Fun(\Delta^1, \Cc) $ that can be formally defined as the unit of the 
adjunction $\,\delta^*_0 :\Fun(\Delta^1, \Cc)\,\leftrightarrows\,\Cc: s_0^*\,$.
It is easy to check that $ \Gan \, s_0^* \,\delta_0^* = s_0^* \delta^*_0 $, and hence $ \Gan^m \, s_0^* \,\delta_0^* = s_0^* \delta^*_0 $ for all $ m \ge 0 $. Then, by composition, $ \alpha $ induces morphisms of endofunctors $\,\alpha_m: \Gan^m \to s_0^* \,\delta_0^* \,$,
satisfying $ \alpha_{m+1} \circ \pi_m = \alpha_m $ for all $ m \ge 0 $. This defines a right cone $\Gan^{\bullet} \to c({s_0^* \,\delta_0^*}) $ on the $I^{\infty}$-diagram \eqref{Gtower} and hence a canonical map
\begin{equation}
\la{Gcolim}
\bcolim_{I^\infty} \,\Gan^{\bullet}\,\to \, s_0^*\,\delta_0^*  
\end{equation}
The left Kan extension functor along the inclusion $j: I^\infty \into \Delta^\infty $ preserves constant diagrams, hence it maps the right cone $\,\Gan^{\bullet} \to c({s_0^* \,\delta_0^*}) \,$ to a right cone $ \,\widetilde{\Gan}^\bullet \to c({s_0^* \,\delta_0^*})\, $ on the Ganea tower. It follows that the map \eqref{Gcolim} factors
through \eqref{eqcolim} inducing a natural morphism  in $\Fun(\Fun(\Delta^1, \Cc), \Fun(\Delta^1, \Cc))$:
\begin{equation}
\la{GGcolim}
\bcolim_{\Delta^\infty} \,\widetilde{\Gan}^{\bullet}\,\to \, s_0^*\,\delta_0^*    
\end{equation}

When $ \Cc = {\mathcal S}_{\ast} $ is the $\infty$-category of pointed spaces, the map \eqref{GGcolim} is an equivalence: evaluated levelwise (see the proof of \Cref{Gacor} below), it amounts to
the homotopy decomposition $\,\hocolim_m (X_m) \xrightarrow{\sim} B \,$ of a pointed space $B$,
which is part of Ganea's classical fiber-cofiber construction. It is therefore natural to ask in general: 
\begin{center}
{\it When is the map \eqref{GGcolim} an equivalence}?
\end{center}
We do not know the answer to this question for arbitrary $\infty$-categories; however, thanks to
the recent work of M. Anel, G. Biedermann, E. Finster and A. Joyal \cite{ABFJ20},
we can provide the following general (but, perhaps, not unexpected) answer for $\infty$-topoi.
\begin{prop}
\la{Gacor}
Assume that a $($pointed$)$ $\infty$-topos $\Cc$ is hypercomplete $($i.e. satisfies a descent condition with respect to hypercoverings in the sense of \cite{DHI}$)$.
Then, the map \eqref{GGcolim}  is an equivalence of endofunctors on $ \Fun(\Delta^1, \Cc)$.   
\end{prop}
\begin{proof}
The proof is parallel to the classical one for topological spaces, except that the classical Blakers-Massey  Theorem is replaced by the generalized one for $\infty$-topoi
proved in \cite{ABFJ20} and the classical Ganea Theorem is replaced by our \Cref{GaTh}. We provide details for reader's convenience.

First, in view of the equivalence \eqref{eqcolim}, 
we may replace the map \eqref{GGcolim} with \eqref{Gcolim}. 
To evaluate the colimit in \eqref{Gcolim}, we will use Proposition~7.1.8.3 of \cite{kerodon} that reduces colimits in diagram categories to the levelwise ones. More precisely, we identify
$$
\Fun(\Fun(\Delta^1, \Cc), \,\Fun(\Delta^1, \Cc))\,\cong\,
\Fun(\Fun(\Delta^1, \Cc) \times \Delta^1, \, \Cc)
$$
and apply the above proposition to the functor $\,
\Gan^\bullet: K \to \Fun(B, \, \Cc) $ with $ K = I^{\infty} $
$B = \Fun(\Delta^1, \Cc) \times \Delta^1 $ ({\it cf.} the notation of \cite{kerodon}). It reduces our problem to showing that, for every object $(p: X \to B ) \in \Fun(\Delta^1, \Cc)_0 $, the following map is 
an equivalence in $\Cc $:
\begin{equation}
\la{objmap}
\bcolim_{I^{\infty}}\,X_{\bullet}(p)\,\to\, B\ 
\end{equation}
where the functor $ X_\bullet(p):\, I^{\infty} \to \Cc $ is defined by $ X_m(p) := \delta^*_1(p_m) = \delta^*_1 \,\Gan^m(p) $ for $ m \in \Z_+$.

Next, recall (see \Cref{Ex2}) that any $\infty$-topos is presentable (hence, admits all (small) limits and colimits) and satisfies the  Mather cube axiom.
We can therefore apply our \Cref{GaTh}. By induction,
this theorem yields 
\begin{equation}
\la{Fm}
F_m \,\cong \,F_{m-1} \ast \Omega B\, ,\quad \forall\,m\ge0\, ,
\end{equation}
where $ (F_m \to X_m) := \Fib(p_m) $ for $ X \xrightarrow{p} B\,$. Now, we can use the results of \cite{ABFJ20} to deduce from \eqref{Fm} that the objects $F_m \in \Cc_0$ are at least $(m-1)$-connected (see \cite{ABFJ20}*{Definition~3.3.4}).
Specifically, this fact follows by induction from the Dual Blakers-Massey Theorem \cite{ABFJ20}*{3.5.1} applied to the pullback squares
\begin{equation}
\la{pullFm}
\begin{diagram}[small]
F_m \times \Omega B  & \rTo & F_m\\
\dTo & & \dTo\\
\Omega B & \rTo & 0\\
\end{diagram}
\end{equation}
in combination with \cite{ABFJ20}*{Corollary~3.3.4 (iv)}.
Thus, as in the case of ordinary topological spaces, we conclude that the map \eqref{objmap} is $\infty$-connected. To complete the proof it remains to note
that in a hypercomplete $\infty$-topos, every $\infty$-connected morphism is an equivalence (see \cite{lurie2009higher}*{Theorem 6.5.3.12}).
\end{proof}

\subsection{Relative Ganea construction} 
\la{AS3}
In this section, we extend \Cref{GaTh} to the relative setting. The construction 
is parallel to that of Section~\Cref{AS2}, we therefore outline only its main steps, focusing on differences and modifications that we need to make. We  keep the assumption that $\Cc $ is a pointed $\infty$-category admitting (finite) limits and colimits.

In the relative case, we replace the $\infty$-category $ \Fun(\Delta^1, \Cc) $ by two different (non-equivalent) $\infty$-categories: $\, \Fun(\Delta^1 \vee \Delta^1, \,\Cc)$ and $ \FunDD $, which are spanned by the pushout and pullback diagrams in $ \Cc $, respectively  (see Examples\, \ref{PushSq} and \ref{PullSq}). The relative version of the fiber-cofiber adjunction \eqref{CFib} is then defined as the composition of two adjunctions:
\begin{equation}
\la{Radj}
q^* \,v_!:\, \Fun(\Delta^1 \vee \Delta^1, \,\Cc)\,\rightleftarrows\, \Fun(\Box, \Cc) \,\rightleftarrows\,  \FunDD : v^*\,q_*
\end{equation}
using the Kan extension and restriction functors along the inclusions 
$$
\Delta^1 \vee \Delta^1 \xhookrightarrow{v} \Box \xhookleftarrow{q} (\Delta^1 \vee \Delta^1)^{\rm op}
$$
Pictorially,
\[
q^*\,v_!\, :\quad 
\begin{tikzcd}[column sep=small, row sep=small]
	A & X \\
	Y
	\arrow[from=1-1, to=1-2]
	\arrow[from=1-1, to=2-1]
\end{tikzcd}
\quad \longmapsto \quad
\begin{tikzcd}[column sep=small, row sep=small]
	A & X \\
	  Y & X\amalg_A Y
	\arrow[from=1-1, to=1-2]
	\arrow[from=1-1, to=2-1]
	\arrow[from=1-2, to=2-2]
	\arrow[from=2-1, to=2-2]
\end{tikzcd}
\quad \longmapsto \quad
\begin{tikzcd}[column sep=small, row sep=small]
	& X \\
	Y & X\amalg_A Y
	\arrow[from=1-2, to=2-2]
	\arrow[from=2-1, to=2-2]
\end{tikzcd}
\]

\[
v^*\,q_*\, :\quad 
\begin{tikzcd}[column sep=small, row sep=small]
	& X \\
	Y & B
	\arrow[from=1-2, to=2-2]
	\arrow[from=2-1, to=2-2]
\end{tikzcd}
\quad \longmapsto \quad
\begin{tikzcd}[column sep=small, row sep=small]
	X\times_B Y & X \\
	  Y & B
	\arrow[from=1-1, to=1-2]
	\arrow[from=1-1, to=2-1]
	\arrow[from=1-2, to=2-2]
	\arrow[from=2-1, to=2-2]
\end{tikzcd}
\quad \longmapsto \quad
\begin{tikzcd}[column sep=small, row sep=small]
	X\times_B Y & X \\
	Y
	\arrow[from=1-1, to=1-2]
	\arrow[from=1-1, to=2-1]
\end{tikzcd}
\]
where $ X\amalg_A Y $ and $ X\times_B Y $ denote the (homotopy) pushout and pullback
in $\Cc$, respectively. Note that \eqref{Radj} being  adjoint functors follows immediately from their definition (by formal properties of Kan extensions). 

To construct a relative analog of the Ganea functor \eqref{Gan} we look at the counit
of the above adjunction:
$$
\varepsilon:\, q^* v_!\, v^* q_* \, \to\, \id_{\FunDD}
$$
and interpret it as a functor
\begin{equation*}
\widetilde{\varepsilon}:\, \FunDD\,\to\,\Fun(\BBox,\Cc)    
\end{equation*}
where $\BBox$ denotes the simplicial set 
$ (\Delta^1 \vee \Delta^1)^{\rm op} \times \Delta^1 $ (which is a (small) $\infty$-category). 
Pictorially, $\widetilde{\varepsilon}$ is given by
\[
(X \to B \leftarrow Y)
\quad \longmapsto \quad
\begin{tikzcd}[column sep=small, row sep=small]
	X & {X \ast_B Y} & Y \\
	X & B & Y
	\arrow[from=1-1, to=1-2]
	\arrow[equals, from=1-1, to=2-1]
	\arrow[from=1-2, to=2-2]
	\arrow[from=1-3, to=1-2]
	\arrow[equals, from=1-3, to=2-3]
	\arrow[from=2-1, to=2-2]
	\arrow[from=2-3, to=2-2]
\end{tikzcd}
\]
where the top row represents the functor $\,q^* v_!\, v^* q_*\,$
and the bottom the identity functor on $ \FunDD $. 
The relative Ganea functor is now defined by restricting $\,\widetilde{\varepsilon}\,$ along the inclusion $j: \Delta^1 \into \BBox $ that identifies $ \Delta^1 $ as the common edge  in $ \BBox \,$:
\begin{equation}
\la{RGan}
\Gan^{\rm rel}: \, 
\FunDD \,\xrightarrow{\widetilde{\varepsilon}} \,\Fun(\BBox,\Cc) \,\xrightarrow{j^*}
\,\Fun(\Delta^1, \Cc)\,.
\end{equation}
This sends $\,(X \to B \leftarrow Y) \mapsto (X\ast_B Y \to B)\,$. 
The relative analog of \eqref{compF}  is the functor
\begin{equation}
\la{RF}
\cF^{\rm rel}: \, \FunDD \,\xrightarrow{v^* q_*} \,\Fun(\Delta^1 \vee \Delta^1, \Cc)\,
\xrightarrow{v_!}\, \Fun(\Box, \Cc)\, ,
\end{equation}
which takes a pair of maps over a common base to the pushout square of their relative  join:
\[
\begin{tikzcd}[column sep=small, row sep=small]
	& X \\
	Y & B
	\arrow[from=1-2, to=2-2]
	\arrow[from=2-1, to=2-2]
\end{tikzcd}
\quad \longmapsto \quad
\begin{tikzcd}[column sep=small, row sep=small]
	X\times_B Y & X \\
	Y
	\arrow[from=1-1, to=1-2]
	\arrow[from=1-1, to=2-1]
\end{tikzcd}
\quad \longmapsto \quad
\begin{tikzcd}[column sep=small, row sep=small]
	X\times_B Y & X \\
	  Y & X \ast_B Y
	\arrow[from=1-1, to=1-2]
	\arrow[from=1-1, to=2-1]
	\arrow[from=1-2, to=2-2]
	\arrow[from=2-1, to=2-2]
\end{tikzcd}
\]
To construct the relative version of the functor \eqref{Gann} we use the 
colimit adjunction \eqref{KBox} with $\,K := \FunDD\,$.
The homotopy equivalence \eqref{Kequiv} is then replaced by
\begin{equation}
\la{Keqrel}
\Map_{\Fun(K, \Cc)}(\bcolim_{\Box}^K(\cF^{\rm rel}),\,\varphi_0^*)\,
\xrightarrow{(\eta^K)^* \,\partial^K_{\Box}}\,
\Map_{\Fun(K, \Fun(\Box,\Cc))}(\cF^{\rm rel},\,\partial_{\Box}^K(\varphi_0^*))\,,
\end{equation}
where the restriction functor $\varphi^*_0: \FunDD \to \Cc\,$ is defined by the 
inclusion $ \varphi_0: \Delta^0 \into (\Delta^1 \vee \Delta^1)^{\rm op} $ identifying
$\Delta^0 $ with the `middle' vertex of $(\Delta^1 \vee \Delta^1)^{\rm op}$. Now, the relative Ganea functor 
\eqref{RGan} can be viewed as an object of the functor category
$$
\Fun(\FunDD,\,\Fun(\Delta^1, \Cc))\,\cong\,
\Fun(\Delta^1,\,\Fun(\FunDD, \,\Cc))
$$
satisfying the  conditions $\,\delta_0^*(\Gan^{\rm rel}) = \varphi_0^* \,$
and $ \delta^*_1(\Gan^{\rm rel}) \cong \bcolim^K_{\Box}(\cF^{\rm rel})$.
Hence, it determines a $0$-simplex of the mapping space in the domain of  \eqref{Keqrel}, 
that maps to a morphism in $ \Hom_{\Fun(K, \Fun(\Box,\Cc))}(\cF^{\rm rel},\,\partial_{\Box}^K(\varphi_0^*))$. This last morphism can then be interpreted as the cubical functor
\begin{equation}
\la{RGann}
\Gan^{{\rm rel}, \natural} :=  \,\partial^K_{\Box} \circ \Gan^{\rm rel}\circ \eta^K:  \ \FunDD \to \Fun(\cube, \Cc)   \,,
\end{equation}
generalizing \eqref{Gann}. Composing now \eqref{RGann} with the cubical fiber functor $\Fib_{\Box}$, we define 
the relative version of the functor \eqref{Xiprime}:
\begin{equation}
\la{RXi}
\Xi^{\rm rel} := \Fib_{\Box}\circ\Gan^{{\rm rel}, \natural} :\, \FunDD \to \Fun(\cube, \Cc)\, .
\end{equation}
Pictorially, the functor \eqref{RXi} can be visualized as
\begin{equation}
\la{RXipic}
\begin{tikzcd}[column sep=small, row sep=small]
	& X \\
	Y & B
	\arrow[from=1-2, to=2-2]
	\arrow[from=2-1, to=2-2]
\end{tikzcd}
%
\quad \longmapsto\quad
\begin{tikzcd}[column sep=tiny, row sep=tiny]
            &  F_X \times F_Y \ar[dl]\ar[dd]\ar[rr] && X \times_B Y \ar[dl]\ar[dd] \\
            F_X \ar[rr,crossing over]\ar[dd] && X & \\
            & F_Y \ar[dl]\ar[rr] && Y \ar[dl] \\
            F_{X,Y} \ar[rr] && X \ast_B Y \ar[from=uu,crossing over," "',near end]
\end{tikzcd}           
\end{equation}
where 
$\, (F_{X,Y} \to X\ast_B Y) := \Fib\,[\Gan^{\rm rel}(X \to B \leftarrow Y)] $.
The following is a generalization of \Cref{GaTh} to the relative setting.
\begin{thm}\la{RGaTh}
Let $\Cc $ be a pointed $\infty$-category with $($finite$)$ limits and colimits. 
The functor $ \Xi^{\rm rel}: \FunDD \to \Fun(\cube, \Cc) $ defined by 
\eqref{RXi} takes values in the full $\infty$-subcategory of $\Fun(\cube, \Cc)$ spanned by the Mather cubes. Consequently, if $\Cc $ satisfies the $($second$)$ Mather cube axiom, there is
an equivalence in $\Cc$
\begin{equation*}
F_{X,Y} \,\cong \, F_X \ast F_Y \,     
\end{equation*}
where $ F_X $ and $ F_Y $ are the fibers of the morphisms $X \to B $ and $ Y \to B $, respectively.
\end{thm}
\begin{rmk} \label{compcolim}
\Cref{RGaTh} is a relative version of \Cref{GaTh}: the latter follows from
the former if we specialize to the diagram $ (X \to B \leftarrow 0)$
in $\FunDD $. The model-categorical versions of these results are established by J.-P. Doeraene
in \cite{Doe98} as a consequeunce of the so-called Join Theorem (see Theorem 3.1, Proposition 4.2 and Corollary 4.3 in  {\it loc. cit.}): these are now standard results in the literature on abstract Lusternik-Schnirelmann invariants  (see \cite{Doe93}, \cite{HL94}, \cite{DT95} and references therein). We point out that these model-categorical results, which we use in the main body of the paper, can, in fact, be formally deduced from our \Cref{GaTh} and \Cref{RGaTh} modulo the well-known relation between (homotopy) limits and colimits in model categories and $\infty$-categories 
(see \cite{lurie2009higher}). To be precise, any (pointed) $\infty$-category $\Cc$ admits a fully faithful (basepoint preserving) embedding $\Cc \hookrightarrow \underline{\cM}$ into the $\infty$-category $\underline{\cM}:=\mathcal{N}^{\Delta}(\cM^{\rm cf})$ underlying a pointed, simplicial model category $\cM$ (see \cite{lurie2009higher}*{Appendix A.2}). Then, by {\cite{lurie2009higher}*{Theorem 4.2.4.1}}, the limits and colimits in $\Cc$ can be realized as {\it homotopy} limits and colimits in $\cM$. Suppose that $ I$ is a small category such that for any $I$-shaped diagram in $\Cc$ (i.e., any $\infty$-functor $NI \to \Cc$), its colimit in $\underline{\cM}$ lies in $\Cc$. Then, for any functor $\cF: NI \to \Cc$,
$$ 
\bcol_{NI}(\cF) \,\simeq \,{\rm hocolim}_{I}(\cF)\, ,
$$
where the homotopy colimit on the right hand side is taken in the model category $\cM$. Dually, if for any $I$-shaped diagram in $\Cc$ (i.e., a $\infty$-functor $NI \to \Cc$), its limit in $\underline{\cM}$ lies in $\Cc$, then for any functor $ \cF: NI \to \Cc$,
$$ \blim_{NI}(\cF) \,\simeq \,{\rm holim}_I(\cF)\, ,$$
where the homotopy limit on the right hand side is taken in the model category $\cM$.
\end{rmk}
%
%

\bibliography{reference.bib}{}
\bibliographystyle{amsalpha}

\end{document}

%% file: preamble.tex
\usepackage{fontenc}
\usepackage{textcmds}  
\usepackage{amsmath, amssymb, amsfonts, amstext, verbatim, amsthm, thmtools, mathrsfs}
\numberwithin{equation}{section} 
\usepackage{mathtools}
\usepackage[mathcal]{eucal}
\usepackage{microtype}
\usepackage[all]{xy}
\usepackage{tikz-cd}
\usepackage{tikz} 
\usepackage{enumerate}
\usepackage{enumitem}
\usepackage[modulo]{lineno}
\usepackage[colorlinks=true,linkcolor=blue,citecolor=blue,urlcolor=blue,citebordercolor={0 0 1},urlbordercolor={0 0 1},linkbordercolor={0 0 1}]{hyperref} 
\usepackage[shortalphabetic]{amsrefs} 
\usepackage[nameinlink]{cleveref}

\def\makeCal#1{%
\expandafter\newcommand\csname c#1\endcsname{\mathcal{#1}}}
\def\makeFrak#1{%
\expandafter\newcommand\csname f#1\endcsname{\mathfrak{#1}}}
\def\makeScr#1{%
\expandafter\newcommand\csname s#1\endcsname{\mathscr{#1}}}

\count@=0
\loop
\advance\count@ 1
\edef\y{\@Alph\count@}%
\expandafter\makeCal\y
\expandafter\makeFrak\y
\ifnum\count@<26
\repeat

\def\CDGA{\textbf{CDGA}}
\DeclareMathOperator{\coker}{coker}
\DeclareMathOperator{\com}{com}

\DeclareMathOperator{\hcolim}{hocolim}
\DeclareMathOperator{\hfib}{hofib}

\DeclareMathOperator{\Hom}{Hom}

\newcommand{\pt}{\mathbf{pt}}

\DeclareMathOperator{\SO}{SO}
\DeclareMathOperator{\SP}{Sp}
\DeclareMathOperator{\Spin}{Spin}
\DeclareMathOperator{\SU}{SU}
\DeclareMathOperator{\Tor}{Tor}

\DeclareMathOperator{\U}{U}


\def\ar{\arrow}

\def\p{\partial}

\def\p{{\partial}}
\def\red{\textcolor{red}}

\theoremstyle{plain}
\newtheorem{thm}{Theorem}[section]
\newtheorem{cor}[thm]{Corollary}
\newtheorem{lem}[thm]{Lemma}

\newtheorem{prop}[thm]{Proposition}

\theoremstyle{definition}
\newtheorem{defn}[thm]{Definition}

\newtheorem{eg}[thm]{Example}
\newtheorem{rmk}[thm]{Remark}

\theoremstyle{remark}

\makeatletter
\renewcommand\paragraph{\@startsection{paragraph}{4}{\z@}%
  {1ex \@plus .2ex}%
  {-1em}%
  {\normalfont\bfseries}}
\makeatother

%